\documentclass[reqno,12pt,letterpaper]{amsart}
\usepackage[proof]{sdmacros}
\usepackage{epigraph}
\usepackage[T2A]{fontenc}
\usepackage[utf8]{inputenc}
\usepackage{tikz-cd}

\title{Notes on hyperbolic dynamics}
\author{Semyon Dyatlov}
\email{dyatlov@math.berkeley.edu}
\address{Department of Mathematics, University of California, Berkeley, CA 94720}

\begin{document}

\begin{abstract}
These expository notes present a proof of the Stable/Unstable Manifold Theorem
(also known as the Hadamard--Perron Theorem).
They also give examples of hyperbolic dynamics: geodesic flows on surfaces of negative
curvature and dispersing billiards.
\end{abstract}

\maketitle

\setlength{\epigraphwidth}{3.75in}

\epigraph{Примерно каждые пять лет, если не чаще, кто-нибудь заново ``открывает''
теорему Адамара~--- Перрона, доказывая ее либо по схеме доказательства
Адамара, либо по схеме Перрона. Я сам в этом повинен\dots\\\medskip

\emph{Every five years or so, if not more often, someone ``discovers'' again
the Hadamard--Perron Theorem, proving it using either Hadamard's or Perron's
method. I have been guilty of this myself\dots}}{Dmitri Anosov \cite[p.~23]{Anosov}}

\addtocounter{section}{1}
\addcontentsline{toc}{section}{1. Introduction}

These expository notes are intended as an introduction to some aspects of hyperbolic dynamics, with
emphasis on the Stable/Unstable Manifold Theorem (also known as the Hadamard--Perron Theorem),
partially following~\cite{KaHa}. They are structured
as follows:
\begin{itemize}
\item In~\S\ref{s:base} we present the Stable/Unstable Manifold Theorem in a simple setting,
capturing the essential components of the proof without some of the technical and notational complications.
\item In~\S\ref{s:general} we give a proof of the general Stable/Unstable Manifold Theorem
for sequences of transformations on $\mathbb R^n$
with canonical stable/unstable spaces at the origin, building on the special case in~\S\ref{s:base}.
\item In~\S\ref{s:maps-and-flows} we use the results of~\S\ref{s:general} to prove the Stable/Unstable Manifold
Theorem for general hyperbolic maps and flows.
\item In~\S\ref{s:examples} we give two important examples of hyperbolic systems:
geodesic flows on surfaces of negative curvature (\S\ref{s:surf-neg})
and dispersing billiard ball maps (\S\ref{s:billiard}).
\end{itemize}
For the (long and rich) history of hyperbolic dynamics we refer the reader to~\cite{KaHa}.

\section{Stable/unstable manifolds in a simple setting}
  \label{s:base}

In this section we present the Stable/Unstable Manifold Theorem
(broken into two parts, Theorem~\ref{t:stun-1} in~\S\ref{s:base-1} and Theorem~\ref{t:stun-2}
in~\S\ref{s:base-4})
under several simplifying assumptions:
\begin{itemize}
\item we study iterates of a single map $\varphi$, defined on a neighborhood
of $0$ in $\mathbb R^2$;
\item $\varphi(0)=0$ and $d\varphi(0)$ is a hyperbolic matrix,
with eigenvalues $2$ and~$1\over 2$;
\item the map $\varphi$ is close to the linearized map $x\mapsto d\varphi(0)\cdot x$
in the $C^{N+1}$ norm.
\end{itemize}
These assumptions are made to make the notation below simpler, however they
do not impact the substance of the proof. As explained below in~\S\S\ref{s:general}--\ref{s:maps-and-flows},
the arguments of this section can be adapted to the setting of general hyperbolic
maps and flows.

\subsection{Existence of stable/unstable manifolds}
  \label{s:base-1}

Throughout this section we use the following notation for $\ell_\infty$ balls in~$\mathbb R^2$:
$$
\overline B_\infty(0,r):=\{(x_1,x_2)\in\mathbb R^2\colon \max(|x_1|,|x_2|)\leq r\}.
$$
We assume that $U_\varphi,V_\varphi\subset \mathbb R^2$ are open sets
with $\overline B_\infty(0,1)\subset U_\varphi\cap V_\varphi$ and
$$
\varphi:U_\varphi\to V_\varphi
$$
is a $C^{N+1}$ map (here $N\geq 1$ is fixed) which satisfies the following assumptions:
\begin{enumerate}
\item $\varphi(0)=0$;
\item The differential $d\varphi(0)$ is equal to
\begin{equation}
  \label{e:hypas-1}
d\varphi(0)=\begin{pmatrix} 2 & 0 \\ 0 & 1/2 \end{pmatrix};
\end{equation}
\item for a small constant $\delta>0$ (chosen later in Theorems~\ref{t:stun-1} and~\ref{t:stun-2})
and all multiindices $\alpha$ with $2\leq |\alpha|\leq N+1$, we have
\begin{equation}
  \label{e:hypas-2}
\sup_{U_\varphi}|\partial^\alpha \varphi|\leq \delta;
\end{equation}
\item $\varphi$ is a diffeomorphism onto its image.
\end{enumerate}
We remark that assumptions~(3) and~(4) above can be arranged to hold locally by zooming
in to a small neighborhood of 0, see~\S\ref{s:reduce-1} below.

It follows from~\eqref{e:hypas-1} that the space $E_u(0):=\mathbb R\partial_{x_1}$
is preserved by the linearized map $x\mapsto d\varphi(0)\cdot x$ and vectors in this space
are expanded exponentially by the powers of $d\varphi(0)$. Similarly
the space $E_s(0):=\mathbb R\partial_{x_2}$ is invariant
and contracted exponentially by the powers of $d\varphi(0)$. We call $E_u(0)$ the \emph{unstable space}
and $E_s(0)$ the \emph{stable space} of $\varphi$ at~0. 

The main results of this section are the nonlinear versions of the above observations:
namely there exist one-dimensional \emph{unstable/stable submanifolds} $W_u,W_s\subset\mathbb R^2$
which are (locally) invariant under the map $\varphi$; the iterates of $\varphi$
are exponentially expanding on the unstable manifold and exponentially contracting
on the stable one.

We construct the unstable/stable manifolds as graphs of $C^N$ functions.
For a function $F:[-1,1]\to\mathbb R$ we define its unstable/stable graphs
\begin{equation}
  \label{e:stun-graphs}
\mathcal G_u(F):=\{x_2=F(x_1),\ |x_1|\leq 1\},\quad
\mathcal G_s(F):=\{x_1=F(x_2),\ |x_2|\leq 1\}
\end{equation}
which are subsets of $\mathbb R^2$.

Theorem~\ref{t:stun-1} below asserts 
existence of unstable/stable manifolds. The fact that
$\varphi$ is expanding on $W_u$ and contracting on $W_s$
is proved later in Theorem~\ref{t:stun-2}.
\begin{theo}
  \label{t:stun-1}
Assume that $\delta$ is small enough (depending only on $N$) and
assumptions~(1)--(4) above hold. Then there exist $C^N$ functions
$$
F_u,F_s:[-1,1]\to [-1,1],\quad
F_u(0)=F_s(0)=0,\quad
\partial_{x_1}F_u(0)=\partial_{x_2}F_s(0)=0
$$
such that, denoting the graphs (see Figure~\ref{f:stun-base})
\begin{equation}
  \label{e:W-u-def}
W_u:=\mathcal G_u(F_u),\quad
W_s:=\mathcal G_s(F_s),
\end{equation}
we have 
\begin{equation}
  \label{e:graph-inv}
\varphi(W_u)\cap \overline B_\infty(0,1)=W_u,\quad
\varphi^{-1}(W_s)\cap \overline B_\infty(0,1)=W_s.
\end{equation}
Moreover, $W_u\cap W_s=\{0\}$.
\end{theo}
\begin{figure}
\includegraphics{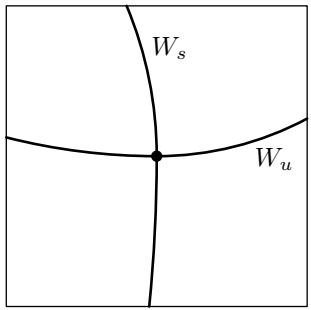}
\caption{The manifolds $W_u,W_s$.
The square is $\overline B_\infty(0,1)$, the horizontal direction is $x_1$,
and the vertical direction is $x_2$.}
\label{f:stun-base}
\end{figure}
The proof of the theorem, given in~\S\S\ref{s:base-2}--\ref{s:base-3} below,
is partially based on the proof of the more general Hadamard--Perron theorem
in~\cite[Theorem~6.2.8]{KaHa}. The main idea is to
show that the action of $\varphi$ on unstable graphs is a contraction mapping
with respect to an appropriately chosen metric (and same with the action
of $\varphi^{-1}$ on stable graphs). There are however two points in which
our proof differs from the one in~\cite{KaHa}:
\begin{itemize}
\item We run the contraction mapping argument on the metric space of $C^N$ functions
whose $N$-th derivative has Lipschitz norm bounded by 1, with the $C^N$ metric.
This is slightly different from the space used in~\cite[\S6.2.d, Step 3]{KaHa}
and it requires having $N+1$ derivatives of the map $\varphi$ to obtain
$C^N$ regularity for invariant graphs (rather than $N$ derivatives as in~\cite{KaHa}).
The upshot is that we do not need separate arguments for establishing regularity
of the manifolds~$W_u,W_s$ \cite[\S6.2.d, Steps 1--2, 4--5]{KaHa}.

\item We only consider the action of $\varphi$ on the ball $\overline B_\infty(0,1)$
rather than extending it to the entire $\mathbb R^2$ as in~\cite[Lemma~6.2.7]{KaHa}.
Because of this parts~(3)--(4) of Theorem~\ref{t:stun-2} below have a somewhat different proof
than the corresponding statement~\cite[Theorem~6.2.8(iii)]{KaHa}.
\end{itemize}

\noindent\textbf{Notation:} In the remainder of this section we denote by $C$ constants which depend only on $N$
(in particular they do not depend on $\delta$), and write
$R=\mathcal O(\delta)$ if $|R|\leq C\delta$. We assume that $\delta>0$ is chosen small (depending only on $N$).

\subsection{Action on graphs and derivative bounds}
  \label{s:base-2}

We start the proof of Theorem~\ref{t:stun-1} by considering the action of $\varphi$
on unstable graphs. (The action of $\varphi^{-1}$ on stable graphs is handled similarly.)
This action, denoted by $\Phi_u$ and called the \emph{graph transform}, is defined by
\begin{lemm}
  \label{l:gmap}
Let $F:[-1,1]\to \mathbb R$ satisfy
\begin{equation}
  \label{e:gmap-1}
F(0)=0,\quad
\sup |\partial_{x_1} F|\leq 1.
\end{equation}
Then there exists a function
$$
\Phi_uF:[-1,1]\to\mathbb R,\quad
\Phi_uF(0)=0
$$
such that (see Figure~\ref{f:gmap}) 
\begin{equation}
  \label{e:gmap-def}
\varphi(\mathcal G_u(F))\cap \{|x_1|\leq 1\}=\mathcal G_u(\Phi_u F).
\end{equation}
\end{lemm}
\begin{figure}
\includegraphics[scale=0.35]{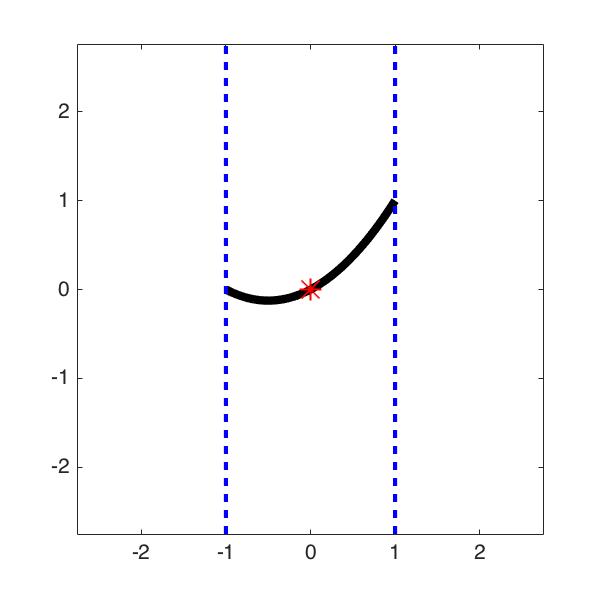}
\includegraphics[scale=0.35]{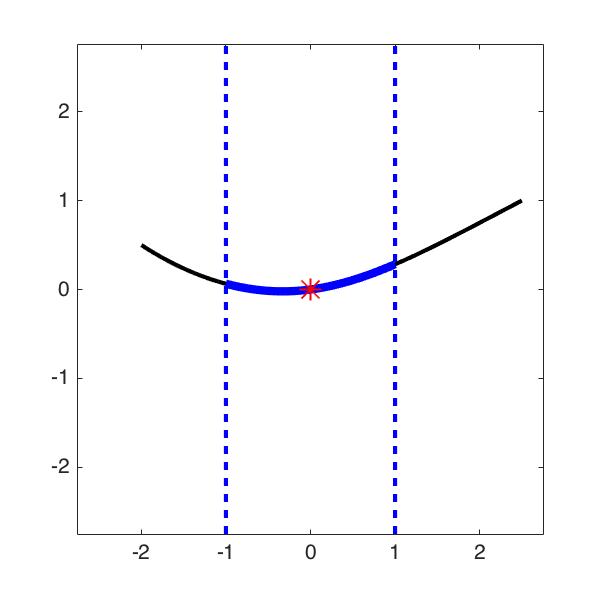}
\caption{Left: the graph $\mathcal G_u(F)$ for some function $F$ satisfying~\eqref{e:gmap-1}. Right:
the image of $\mathcal G_u(F)$ under $\varphi$.
The solid blue part is the graph $\mathcal G_u(\Phi_u F)$.
Figures~\ref{f:gmap}--\ref{f:exc-2} are plotted numerically using the map
$\varphi(x_1,x_2)=(2x_1+{1\over 2}x_2^2,{1\over 2}x_2+{1\over 2}x_1^2)$.}
\label{f:gmap}
\end{figure}
\begin{proof}
Define the components $\varphi_1,\varphi_2:\overline B_\infty(0,1)\to\mathbb R$ of $\varphi$ by
\begin{equation}
  \label{e:phi-12-def}
\varphi(x)=(\varphi_1(x),\varphi_2(x)),\quad
x\in \overline B_\infty(0,1).
\end{equation}
Next, define the functions $G_1,G_2:[-1,1]\to\mathbb R$ by
\begin{equation}
  \label{e:G-12-def}
G_1(x_1):=\varphi_1(x_1,F(x_1)),\quad
G_2(x_1):=\varphi_2(x_1,F(x_1)),
\end{equation}
so that the manifold $\varphi(\mathcal G_u(F))$ has the form
$$
\varphi(\mathcal G_u(F))=\{(G_1(x_1),G_2(x_1))\colon |x_1|\leq 1\}.
$$
To write $\varphi(\mathcal G_u(F))$ as a graph, we need to show that $G_1$
is invertible.
Note that $G_1(0)=0$. We next compute
$$
\partial_{x_1}G_1(x_1)=\partial_{x_1}\varphi_1(x_1,F(x_1))+\partial_{x_2}\varphi_1(x_1,F(x_1))\cdot\partial_{x_1}F(x_1).
$$
Together~\eqref{e:hypas-1} and~\eqref{e:hypas-2} imply for all $(x_1,x_2)\in \overline B_\infty(0,1)$
$$
\partial_{x_1}\varphi_1(x_1,x_2)=2+\mathcal O(\delta),\quad
\partial_{x_2}\varphi_1(x_1,x_2)=\mathcal O(\delta),
$$
so for $\delta$ small enough
\begin{equation}
  \label{e:G-1-der}
\partial_{x_1}G_1(x_1)=2+\mathcal O(\delta)\geq {3\over 2}\quad\text{for all }
x_1\in [-1,1].
\end{equation}
Therefore $G_1$ is a diffeomorphism and its image contains $[-1,1]$.
It follows that $\varphi(\mathcal G_u(F))\cap \{|x_1|\leq 1\}=\mathcal G_u(\Phi_u F)$ for
the function $\Phi_u F$ defined by
\begin{equation}
  \label{e:Phi-u-def}
\Phi_uF(y_1)=G_2(G_1^{-1}(y_1)),\quad y_1\in [-1,1]
\end{equation}
where
\begin{equation}
  \label{e:G-1-inv}
G_1^{-1}:[-1,1]\to [-1,1]
\end{equation}
is the inverse of $G_1$.
\end{proof}
We now want to estimate the function $\Phi_u F$ in terms of $F$,
ultimately showing that $\Phi_u$ is a contraction with respect
to a certain norm. For that we use the following
formula for the derivatives of $\Phi_u F$:
\begin{lemm}
  \label{l:derform}
Let $1\leq k\leq N$.
Assume that $F\in C^k([-1,1];\mathbb R)$ satisfies
\begin{equation}
  \label{e:gmap-2}
F(0)=0,\quad \sup |\partial_{x_1}^j F|\leq 1\quad\text{for all }j=1,\dots,k.  
\end{equation}
Then we have for all $y_1\in [-1,1]$ and $G_1^{-1}$ defined in~\eqref{e:G-1-inv}
\begin{equation}
  \label{e:derform}
\partial^k_{x_1}(\Phi_u F)(y_1)=L_k(x_1,F(x_1),\partial_{x_1}F(x_1),\dots,\partial_{x_1}^kF(x_1)),\quad
x_1:=G_1^{-1}(y_1)
\end{equation}
where the function $L_k(x_1,\tau_0,\dots,\tau_k)$, depending on $\varphi$
but not on $F$, is continuous on the cube $Q_k:=[-1,1]^{k+2}$. Moreover
$L_k(x_1,\tau_0,\dots,\tau_k)=2^{-k-1}\tau_k+\mathcal O(\delta)$,
with the remainder satisfying the derivative bounds
\begin{equation}
  \label{e:derform-2}
\begin{gathered}
\sup_{Q_k}\big|\partial^\alpha_{x_1}\partial^{\beta_0}_{\tau_0}\dots\partial^{\beta_k}_{\tau_k}
\big(L_k(x_1,\tau_0,\dots,\tau_k)-2^{-k-1}\tau_k\big)\big|\leq C_{\alpha\beta}\,\delta\\
\text{for all}\quad\alpha,\beta_0,\dots,\beta_k\quad\text{such that}\quad
\alpha+\beta_0+k\leq N+1.
\end{gathered}
\end{equation}
\end{lemm}
\Remark In the linear case $\varphi(x)=d\varphi(0)\cdot x$ we have
$G_1(x_1)=2x_1$, $G_2(x_1)={1\over 2}F(x_1)$,
therefore $\Phi_u F(y_1)={1\over 2}F({1\over 2}y_1)$.
The meaning of~\eqref{e:derform-2} is that the
action of $\Phi_u$ on the derivatives of~$F$ for nonlinear $\varphi$
is $\mathcal O(\delta)$-close to the linear case. 
\begin{proof}
For $x\in \overline B_\infty(0,1)$ define the matrix
$$
A(x):=d\varphi(x)=(A_{jk}(x)),\quad
A_{jk}(x):=\partial_{x_k}\varphi_j(x).
$$
By~\eqref{e:hypas-1} and~\eqref{e:hypas-2} we have
\begin{equation}
  \label{e:A-prop}
A(x)=\begin{pmatrix} 2& 0 \\ 0 & 1/2\end{pmatrix}+\mathcal O(\delta),
\end{equation}
and the remainder in~\eqref{e:A-prop} is $\mathcal O(\delta)$ with
derivatives of order $\leq N$.

We argue by induction on $k$. 
For $k=1$, from the definition~\eqref{e:Phi-u-def}
of $\Phi_u F$ we have
$$
\partial_{x_1}(\Phi_u F)(y_1)={\partial_{x_1}G_2(x_1)\over \partial_{x_1}G_1(x_1)},\quad
x_1:=G_1^{-1}(y_1)
$$
so~\eqref{e:derform} holds with
\begin{equation}
  \label{e:lulla}
L_1(x_1,\tau_0,\tau_1)={A_{21}(x_1,\tau_0)+A_{22}(x_1,\tau_0) \tau_1\over
A_{11}(x_1,\tau_0)+A_{12}(x_1,\tau_0) \tau_1}.
\end{equation}
From~\eqref{e:A-prop} we see that $L_1(x_1,\tau_0,\tau_1)={1\over 4}\tau_1+\mathcal O(\delta)$
and the stronger remainder estimate~\eqref{e:derform-2} holds.

Now assume that $2\leq k\leq N$ and~\eqref{e:derform}, \eqref{e:derform-2} hold for $k-1$.
Then by the chain rule~\eqref{e:derform}
holds for~$k$ with
$$
L_{k}(x_1,\tau_0,\dots,\tau_{k}):={\partial_{x_1} L_{k-1}(x_1,\tau_0,\dots,\tau_{k-1})
+\sum_{j=0}^{k-1} \partial_{\tau_j}L_{k-1}(x_1,\tau_0,\dots,\tau_{k-1})\tau_{j+1}
\over A_{11}(x_1,\tau_0)+A_{12}(x_1,\tau_0)\tau_1}.
$$
It is straightforward to check that $L_k(x_1,\tau_0,\dots,\tau_k)=2^{-k-1}\tau_k+\mathcal O(\delta)$
and the stronger remainder estimate~\eqref{e:derform-2} holds.
\end{proof}
Armed with Lemma~\ref{l:derform} we estimate the derivatives
of $\Phi_u F$ in terms of the derivatives of $F$. Let $1\leq k\leq N$.
We use the following seminorm
on $C^k([-1,1];\mathbb R)$:
\begin{equation}
  \label{e:C-k-norm}
\|F\|_{C^k}:=\max_{1\leq j\leq k}\sup|\partial^j_{x_1}F|.
\end{equation}
We will work with functions satisfying $F(0)=0$, on which
$\|\bullet\|_{C^k}$ is a norm:
\begin{equation}
  \label{e:F-sup-bound}
F(0)=0\quad\Longrightarrow\quad
\sup|F|\leq \|F\|_{C^1}.
\end{equation}
To establish the contraction property (see the remark following
the proof of Lemma~\ref{l:derest-2}) we also
need the space of functions $C^{k,1}([-1,1];\mathbb R)$
with Lipschitz continuous $k$-th derivative, endowed with the seminorm
\begin{equation}
  \label{e:C-k-1-norm}
\|F\|_{C^{k,1}}:=\max\bigg(\|F\|_{C^k},\sup_{x_1\neq\tilde x_1}
{|\partial^k_{x_1}F(x_1)-\partial^k_{x_1}F(\tilde x_1)|\over |x_1-\tilde x_1|}\bigg).
\end{equation}
Note that $\|F\|_{C^k}\leq \|F\|_{C^{k,1}}\leq \|F\|_{C^{k+1}}$.

Our first estimate implies that $\Phi_u$ maps the unit
balls in $C^k$, $C^{k,1}$ into themselves:
\begin{lemm}
  \label{l:derest-1}
Let $1\leq k\leq N$ and assume that $F(0)=0$ and $\|F\|_{C^k}\leq 1$.
Then
\begin{equation}
  \label{e:derest-1.1}
\|\Phi_u F\|_{C^k}\leq {1\over 4}\|F\|_{C^k}+C\delta.
\end{equation}
If additionally $\|F\|_{C^{k,1}}\leq 1$ then
\begin{equation}
  \label{e:derest-1.2}
\|\Phi_u F\|_{C^{k,1}}\leq {1\over 4}\|F\|_{C^{k,1}}+C\delta.
\end{equation}
\end{lemm}
\begin{proof}
Let $y_1\in [-1,1]$ and $x_1:=G_1^{-1}(y_1)\in [-1,1]$.
By Lemma~\ref{l:derform} we have for all $j=1,\dots,k$
$$
|\partial_{x_1}^j (\Phi_uF)(y_1)|\leq 2^{-j-1}|\partial_{x_1}^jF(x_1)|+C\delta
\leq {1\over 4}\|F\|_{C^k}+C\delta
$$
which implies~\eqref{e:derest-1.1}.

Next, assume that $\|F\|_{C^{k,1}}\leq 1$. Take $y_1,\tilde y_1\in [-1,1]$ such that $y_1\neq \tilde y_1$. Put
$x_1:=G_1^{-1}(y_1)$, $\tilde x_1:=G_1^{-1}(\tilde y_1)$. Then
by Lemma~\ref{l:derform}
$$
\begin{aligned}
|\partial_{x_1}^k (\Phi_u F)(y_1)-\partial_{x_1}^k(\Phi_u F)(\tilde y_1)|
&\leq
2^{-k-1}|\partial_{x_1}^k F(x_1)-\partial_{x_1}^k F(\tilde x_1)|
+C\delta|x_1-\tilde x_1|\\
&\quad\ +C\delta\max_{0\leq j\leq k}|\partial_{x_1}^j F(x_1)-\partial_{x_1}^j F(\tilde x_1)|
\\
&\leq
\Big({1\over 4}\|F\|_{C^{k,1}}+C\delta+C\delta \|F\|_{C^{k,1}}\Big)|x_1-\tilde x_1|.
\end{aligned}
$$
Since $|x_1-\tilde x_1|\leq |y_1-\tilde y_1|$ by~\eqref{e:G-1-der},
this implies~\eqref{e:derest-1.2}.
\end{proof}

The next estimate gives the contraction property of $\Phi_u$
in $C^k$ norm:
\begin{lemm}
  \label{l:derest-2}
Let $1\leq k\leq N$. Assume that $F,\widetilde F\in C^{k,1}$ satisfy $F(0)=\widetilde F(0)=0$
and $\|F\|_{C^{k,1}},\|\widetilde F\|_{C^{k,1}}\leq 1$. Then
\begin{equation}
  \label{e:derest-2}
\|\Phi_u F-\Phi_u \widetilde F\|_{C^k}\leq \Big({1\over 4}+C\delta\Big)\|F-\widetilde F\|_{C^k}.
\end{equation}
\end{lemm}
\begin{proof}
Let $G_1$ and~$\widetilde G_1$ be defined by~\eqref{e:G-12-def} using the functions
$F$ and~$\widetilde F$ respectively. 
Take $y_1\in [-1,1]$ and put $x_1:=G_1^{-1}(y_1)$, $\tilde x_1:=\widetilde G_1^{-1}(y_1)$,
see Figure~\ref{f:contractor}.
We first estimate the difference between the two inverses $x_1$, $\tilde x_1$:
\begin{equation}
  \label{e:anor-1}
|x_1-\tilde x_1|\leq C\delta\|F-\widetilde F\|_{C^1}.
\end{equation}
\begin{figure}
\includegraphics[scale=0.35]{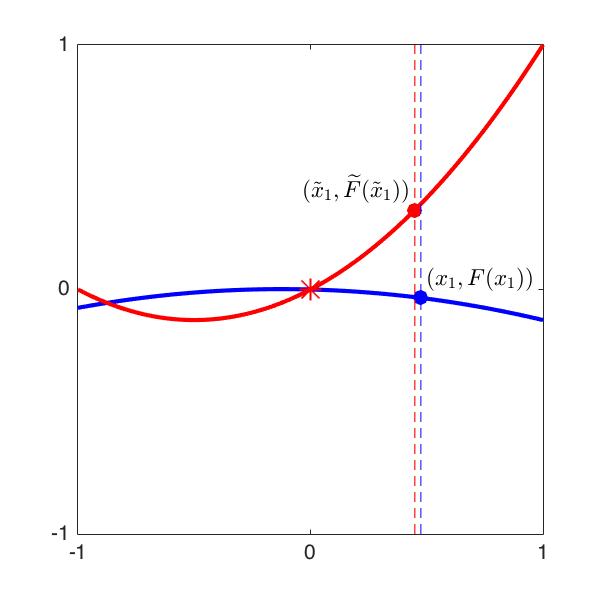}
\includegraphics[scale=0.35]{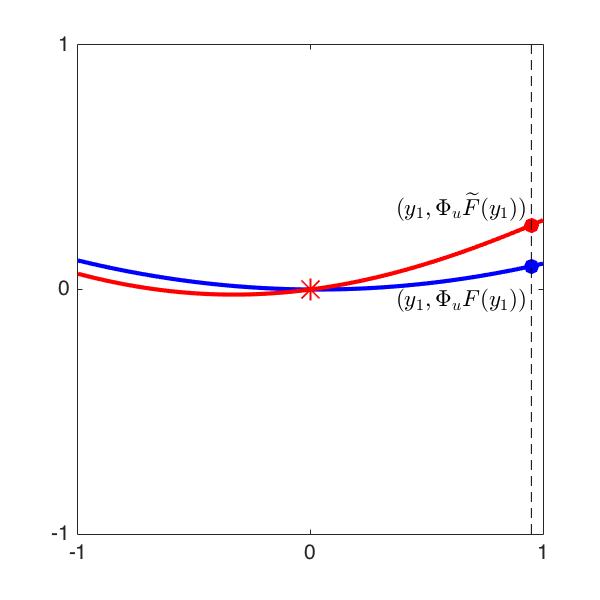}
\caption{Left: the points $(x_1,F(x_1))$, $(\tilde x_1,\widetilde F(\tilde x_1))$.
The blue curve is the graph of $F$ and the red curve is the graph of $\widetilde F$.
Right: the image of the picture on the left by $\varphi$.}
\label{f:contractor}
\end{figure}
To show~\eqref{e:anor-1}, we write
$$
\begin{aligned}
|x_1-\tilde x_1|&\leq |\widetilde G_1(x_1)-\widetilde G_1(\tilde x_1)|=
|\widetilde G_1(x_1)-G_1(x_1)|
\\&=|\varphi_1(x_1,\widetilde F(x_1))-\varphi_1(x_1,F(x_1))|\leq C\delta \|F-\widetilde F\|_{C^1}
\end{aligned}
$$
where the first inequality follows from~\eqref{e:G-1-der}
and the last inequality uses that $\partial_{x_2}\varphi_1=\mathcal O(\delta)$
by~\eqref{e:A-prop}
and $|\widetilde F(x_1)-F(x_1)|\leq \|F-\widetilde F\|_{C^1}$ by~\eqref{e:F-sup-bound}.

Next we have for all $j=0,\dots,k$
\begin{equation}
  \label{e:anor-2}
|\partial_{x_1}^j F(x_1)-\partial_{x_1}^j \widetilde F(\tilde x_1)|\leq (1+C\delta)
\|F-\widetilde F\|_{C^k}.
\end{equation}
Indeed,
$$
\begin{aligned}
|\partial_{x_1}^j F(x_1)-\partial_{x_1}^j \widetilde F(\tilde x_1)|&\leq
|\partial_{x_1}^j F(x_1)-\partial_{x_1}^j F(\tilde x_1)|
+|\partial_{x_1}^j F(\tilde x_1)-\partial_{x_1}^j\widetilde F(\tilde x_1)|\\
&\leq |x_1-\tilde x_1|+\|F-\widetilde F\|_{C^k}\leq (1+C\delta)\|F-\widetilde F\|_{C^k}
\end{aligned}
$$
where the second inequality uses the fact that $\|F\|_{C^{k,1}}\leq 1$
and the last inequality used~\eqref{e:anor-1}.

Finally, by Lemma~\ref{l:derform} we estimate for all $j=1,\dots,k$
$$
\begin{aligned}
|\partial^j_{x_1}(\Phi_uF)(y_1)-\partial^j_{x_1}(\Phi_u\widetilde F)(y_1)|&\leq
2^{-j-1}|\partial_{x_1}^j F(x_1)-\partial_{x_1}^j \widetilde F(\tilde x_1)|
\\&\quad +C\delta |x_1-\tilde x_1|+C\delta\max_{0\leq \ell\leq j}|\partial_{x_1}^\ell F(x_1)-\partial_{x_1}^\ell\widetilde F(\tilde x_1)|
\\&\leq
\Big({1\over 4}+C\delta\Big)\|F-\widetilde F\|_{C^k}
\end{aligned}
$$
where the second inequality uses~\eqref{e:anor-1} and~\eqref{e:anor-2}.
This implies~\eqref{e:derest-2}.
\end{proof}
\Remark The a priori bound $\|F\|_{C^{k,1}}\leq 1$
was used in the proof of~\eqref{e:anor-2}. Without it we would
not be able to estimate the difference
$|\partial^k_{x_1} F(x_1)-\partial^k_{x_1}\widetilde F(\tilde x_1)|$
since the functions $\partial^k_{x_1}F$, $\partial^k_{x_1}\widetilde F$
are evaluated at two different values of $x_1$. We do not
use the stronger a priori bound $\|F\|_{C^{k+1}}\leq 1$
because it would make it harder to set up a complete metric space
for the contraction mapping argument in~\eqref{e:X-N-def} below.

\subsection{Contraction mapping argument}
  \label{s:base-3}

\begin{figure}
\includegraphics[scale=0.225]{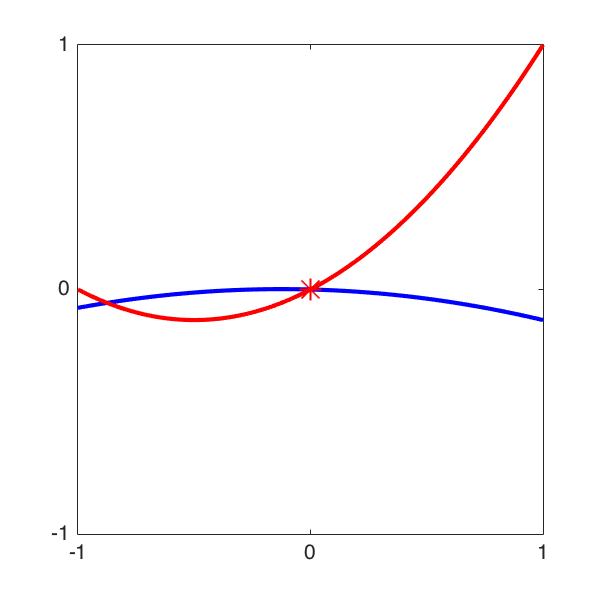}
\includegraphics[scale=0.225]{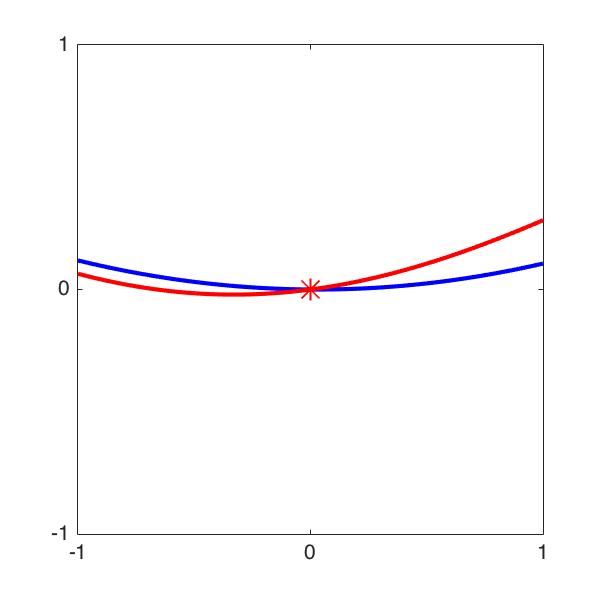}
\includegraphics[scale=0.225]{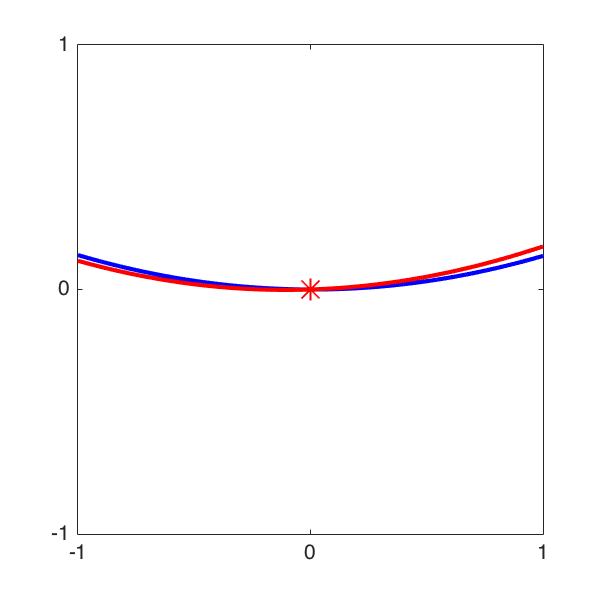}
\hbox to\hsize{\quad\hss $n=0$\hss\hss $n=1$\hss\hss $n=2$\hss\quad}
\caption{The iterations $\varphi^n(\mathcal G_{F_0})$ for two different choices
of $F_0$. Both converge to $W_u$, illustrating~\eqref{e:contramap-u}.}
\label{f:contramap}
\end{figure}

We now give the proof of Theorem~\ref{t:stun-1} using the estimates
from the previous section. We show existence of the function $F_u$;
the function $F_s$ is constructed similarly, replacing $\varphi$ by $\varphi^{-1}$
and switching the roles of $x_1$ and $x_2$. 

Consider the metric space $(\mathcal X_N,d_N)$ defined using the seminorms~\eqref{e:C-k-norm},
\eqref{e:C-k-1-norm}:
\begin{equation}
  \label{e:X-N-def}
\begin{aligned}
\mathcal X_N&:=\{F\in C^{N,1}([-1,1];\mathbb R)\colon F(0)=0,\ \|F\|_{C^{N,1}}\leq 1\},\\
d_N(F,\widetilde F)&:=\|F-\widetilde F\|_{C^N}.
\end{aligned}
\end{equation}
Then $(\mathcal X_N,d_N)$ is a complete metric space. Indeed, it is the subset
of the closed unit ball in $C^N$ defined using the closed conditions
$$
F(0)=0,\quad
|\partial^N_{x_1}F(x_1)-\partial^N_{x_1} F(\tilde x_1)|\leq |x_1-\tilde x_1|\quad\text{for all }
x_1,\tilde x_1\in [-1,1].
$$
By Lemmas~\ref{l:gmap} and~\ref{l:derest-1}, for $\delta$ small enough the graph transform defines a map
$$
\Phi_u:\mathcal X_N\to\mathcal X_N.
$$
By Lemma~\ref{l:derest-2}, for $\delta$ small enough this map is a contraction, specifically
\begin{equation}
  \label{e:contraction-achieved}
d_N(\Phi_u F,\Phi_u\widetilde F)\leq {1\over 3}d_N(F,\widetilde F)\quad\text{for all }
F,\widetilde F\in \mathcal X_N.
\end{equation}
Therefore by the Contraction Mapping Principle the map $\Phi_u$ has a unique fixed point
$$
F_u\in\mathcal X_N,\quad
\Phi_uF_u=F_u.
$$
In fact for each fixed $F_0\in \mathcal X_N$ we have
(see Figure~\ref{f:contramap})
\begin{equation}
  \label{e:contramap-u}
(\Phi_u)^n F_0\to F_u\quad\text{in }C^N\quad\text{as }n\to\infty.
\end{equation}
Let $W_u:=\mathcal G_u(F_u)\subset \overline B_\infty(0,1)$ be the unstable graph of $F_u$.
Recalling the definition~\eqref{e:gmap-def} of $\Phi_uF$, we have
$$
\varphi(W_u)\cap \{|x_1|\leq 1\}=\mathcal G_u(\Phi_u F_u)=W_u.
$$
It follows that $\varphi(W_u)\cap \overline B_\infty(0,1)=W_u$, giving~\eqref{e:graph-inv}.

Next, we see from~\eqref{e:lulla} that $F(0)=\partial_{x_1}F(0)=0$ implies
$\partial_{x_1}(\Phi_u F)(0)=0$.
Using~\eqref{e:contramap-u} with $F_0\equiv 0$,
we get $\partial_{x_1}F_u(0)=0$.

By Lemma~\ref{l:derest-1} we have the following derivative bounds on $F_u$, $F_s$:
\begin{equation}
  \label{e:F-u-derb}
\|F_u\|_{C^{N,1}}\leq C\delta,\quad
\|F_s\|_{C^{N,1}}\leq C\delta.
\end{equation}
This implies that $W_u\cap W_s=\{0\}$. Indeed, if $(x_1,x_2)\in W_u\cap W_s$,
then $x_2=F_u(x_1)$ and $x_1=F_s(x_2)$. Therefore,
$x_1=F_s(F_u(x_1))$. However~\eqref{e:F-u-derb} implies
that
\begin{equation}
  \label{e:F-comp-contra}
\sup_{[-1,1]}|\partial_{x_1}(F_s\circ F_u)|\leq C\delta.
\end{equation}
Therefore $F_s\circ F_u:[-1,1]\to[-1,1]$ is a contraction
and the equation $x_1=F_s(F_u(x_1))$ has only one solution, $x_1=0$.

\Remark The above proof shows that $F_u\in C^{N,1}$.
One can prove in fact that $F_u\in C^{N+1}$, see~\cite[\S6.2.d, Step~5]{KaHa}.

\subsection{Further properties}
  \label{s:base-4}
  
In this section we prove the following theorem which relates the manifolds
$W_u,W_s$ to the behavior of large iterates $\varphi^n$ of the map $\varphi$:
\begin{theo}
  \label{t:stun-2}
Let $\varphi$ be as in Theorem~\ref{t:stun-1} and $\delta$ be small enough depending only on $N$.
Let $W_u,W_s$ be defined in~\eqref{e:W-u-def}. Then:
\begin{enumerate}
\item If $w\in W_u$ then $\varphi^{-n}(w)\to 0$ as $n\to\infty$,
more precisely
\begin{equation}
  \label{e:stun-2-1}
|\varphi^{-n}(w)|\leq \Big({1\over 2}+C\delta\Big)^n|w|\quad\text{ for all }n\geq 0.
\end{equation}
\item If $w\in W_s$ then $\varphi^{n}(w)\to 0$ as $n\to\infty$,
more precisely
\begin{equation}
  \label{e:stun-2-2}
|\varphi^{n}(w)|\leq \Big({1\over 2}+C\delta\Big)^n|w|\quad\text{ for all }n\geq 0.
\end{equation}
\item If $w\in \overline B_\infty(0,1)$ satisfies $\varphi^{-n}(w)\in \overline B_\infty(0,1)$
for all $n\geq 0$, then $w\in W_u$.
\item If $w\in \overline B_\infty(0,1)$ satisfies $\varphi^{n}(w)\in \overline B_\infty(0,1)$
for all $n\geq 0$, then $w\in W_s$.
\end{enumerate}
\end{theo}
\Remark Theorem~\ref{t:stun-2} implies the following dynamical characterization
of the unstable/stable manifolds $W_u,W_s$:
\begin{equation}
  \label{e:stun-char}
\begin{aligned}
w\in W_u&\quad\Longleftrightarrow\quad \varphi^{-n}(w)\in \overline B_\infty(0,1)\quad\text{for all }n\geq 0;\\
w\in W_s&\quad\Longleftrightarrow\quad \varphi^{n}(w)\in \overline B_\infty(0,1)\quad\text{for all }n\geq 0.
\end{aligned}
\end{equation}
See Figure~\ref{f:dynch}.

We only give the proof of parts~(1) and~(3) in Theorem~\ref{t:stun-2}.
Parts~(2) and~(4), characterizing the stable manifold, are proved similarly, replacing
$\varphi$ by $\varphi^{-1}$ and switching the roles of $x_1$ and~$x_2$.

\begin{figure}
\includegraphics[scale=0.115]{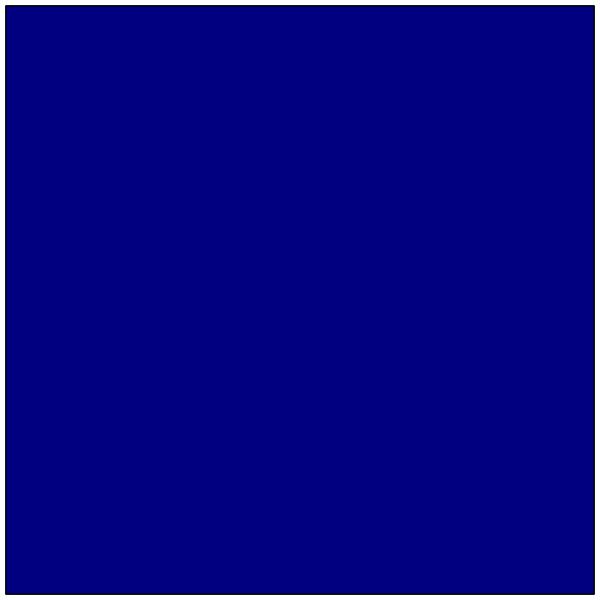}
\includegraphics[scale=0.115]{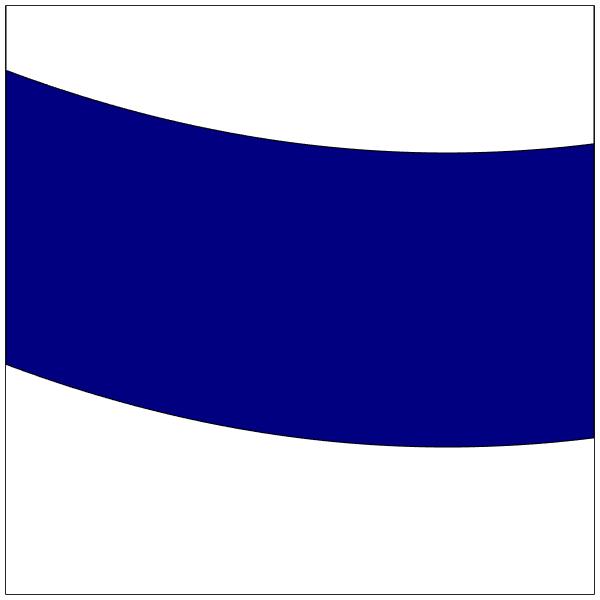}
\includegraphics[scale=0.115]{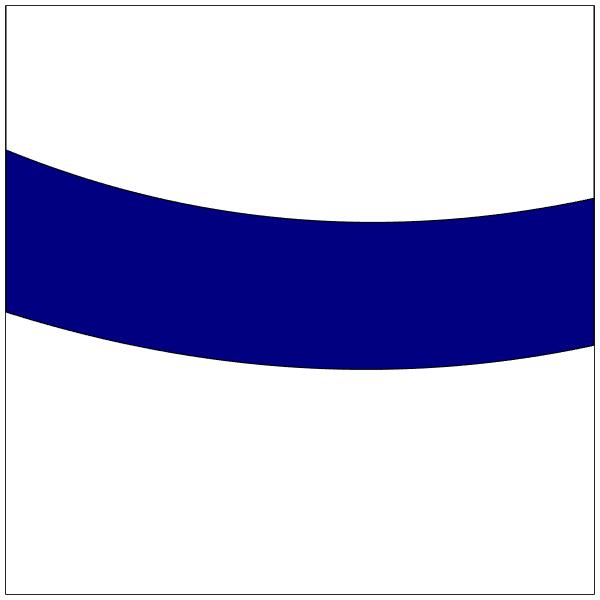}
\includegraphics[scale=0.115]{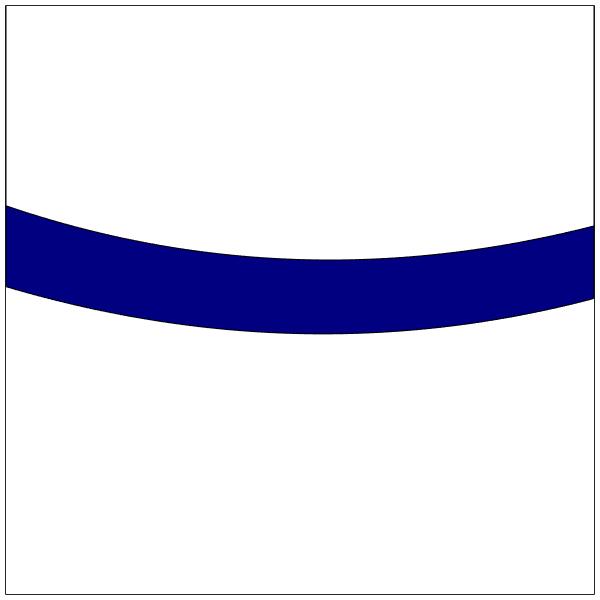}
\includegraphics[scale=0.115]{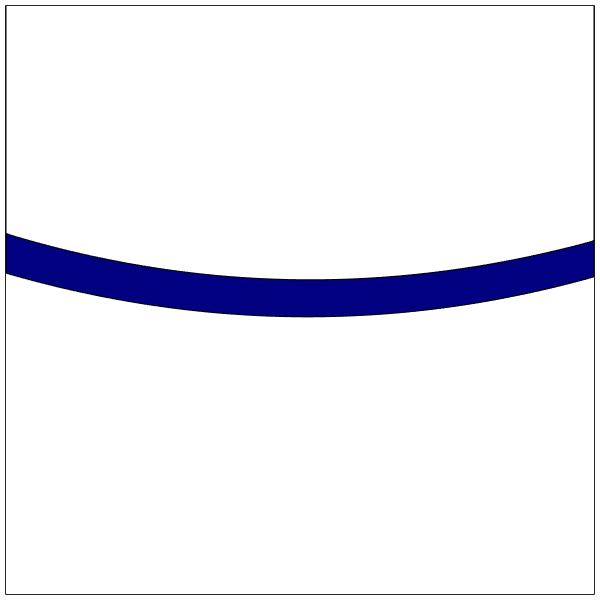}
\includegraphics[scale=0.115]{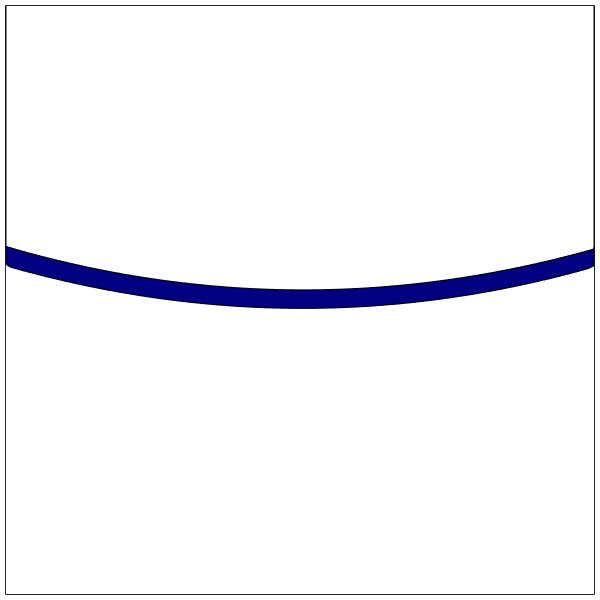}
\hbox to\hsize{\hss $n=0$ \hss\hss $n=1$ \hss\hss $n=2$ \hss\hss $n=3$ \hss\hss $n=4$ \hss\hss $n=5$ \hss}
\caption{The sets of points $w$ such that $w,\varphi^{-1}(w),\dots,\varphi^{(-n)}(w)\in\overline B_\infty(0,1)$.
By~\eqref{e:stun-char}, in the limit $n\to\infty$ we obtain $W_u$.}
\label{f:dynch}
\end{figure}

Part~(1) of Theorem~\ref{t:stun-2} follows by iteration (putting $y:=w$, $\tilde y:=0$) from
\begin{lemm}
  \label{l:dyndyn-1}
Let $y,\tilde y\in W_u$ and put $x:=\varphi^{-1}(y)$, $\tilde x:=\varphi^{-1}(\tilde y)$.
Then (see Figure~\ref{f:exc-1})
\begin{equation}
  \label{e:dyndyn-1}
|x-\tilde x|\leq \Big({1\over 2}+C\delta\Big)|y-\tilde y|.
\end{equation}
\end{lemm}
\begin{proof}
By~\eqref{e:graph-inv} we have $x,\tilde x\in W_u\subset \overline B_\infty(0,1)$.
We write
$$
x=(x_1,F_u(x_1)),\quad
\tilde x=(\tilde x_1,F_u(\tilde x_1)),\quad
y=(y_1,F_u(y_1)),\quad
\tilde y=(\tilde y_1,F_u(\tilde y_1)).
$$
Since $\|F_u\|_{C^1}\leq C\delta$ by~\eqref{e:F-u-derb}, we have
\begin{align}
  \label{e:dyndyn-1.1}
|x-\tilde x|&\leq (1+C\delta)|x_1-\tilde x_1|,\\
  \label{e:dyndyn-1.2}
|y-\tilde y|&\geq (1-C\delta)|y_1-\tilde y_1|.
\end{align}
We have $y_1=\varphi_1(x_1,F_u(x_1))$ and $\tilde y_1=\varphi_1(\tilde x_1,F_u(\tilde x_1))$.
Thus by~\eqref{e:A-prop}
\begin{equation}
  \label{e:dyndyn-1.3}
y_1-\tilde y_1=2(x_1-\tilde x_1)+\mathcal O(\delta)|x_1-\tilde x_1|.
\end{equation}
Together~\eqref{e:dyndyn-1.1}--\eqref{e:dyndyn-1.3} give~\eqref{e:dyndyn-1}.
\end{proof}
\begin{figure}
\includegraphics[scale=0.35]{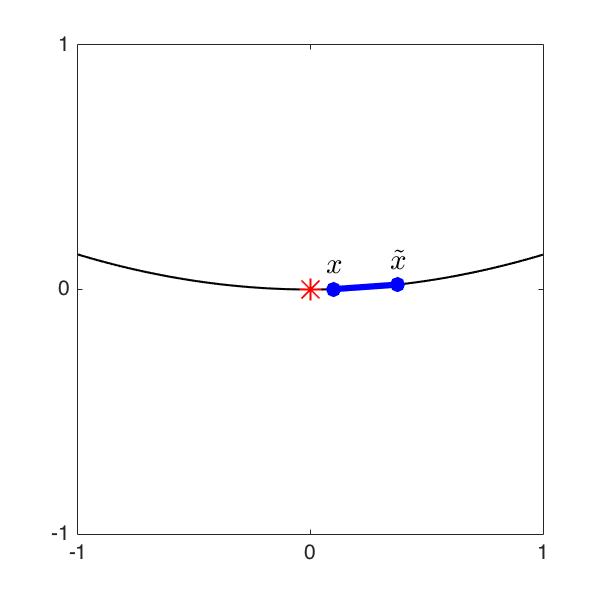}
\includegraphics[scale=0.35]{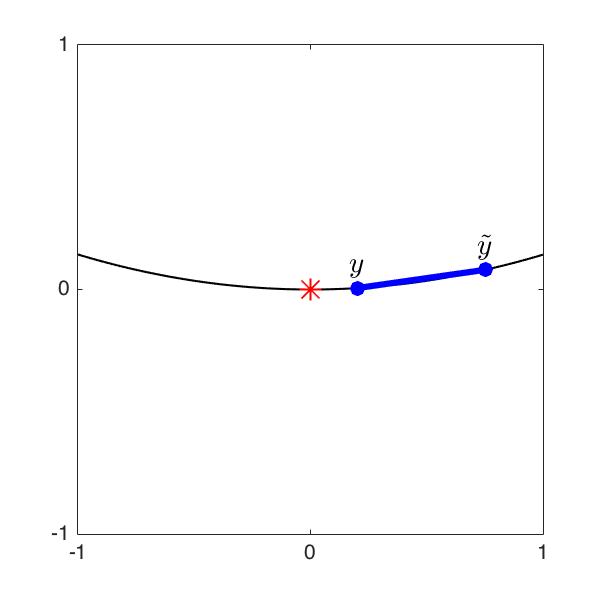}
\caption{The points $x,\tilde x,y,\tilde y$ from Lemma~\ref{l:dyndyn-1}.
The curve is $W_u$.}
\label{f:exc-1}
\end{figure}
It remains to show part~(3) of Theorem~\ref{t:stun-2}. For a point
$w=(w_1,w_2)\in \overline B_\infty(0,1)$, define the distance from it to the unstable manifold by
\begin{equation}
  \label{e:unstable-distance}
d(w,W_u):=|w_2-F_u(w_1)|.
\end{equation}
The key component of the proof is
\begin{lemm}
  \label{l:dyndyn-2}
Assume that $w\in \overline B_\infty(0,1)$ and $\varphi(w)\in\overline B_\infty(0,1)$.
Then (see Figure~\ref{f:exc-2})
\begin{equation}
  \label{e:dyndyn-2}
d(\varphi(w),W_u)\leq \Big({1\over 2}+C\delta\Big)d(w,W_u).
\end{equation}
\end{lemm}
\begin{proof}
We write
$$
w=(w_1,w_2),\quad
z:=\varphi(w)=(z_1,z_2).
$$
Define
$$
x:=(w_1,F_u(w_1)),\quad
y:=\varphi(x)=(y_1,F_u(y_1)).
$$
(Since it might happen that $y_1\notin [-1,1]$, strictly speaking
we extend the function $F_u$ to a larger interval by making $\varphi(W_u)$
the graph of the extended $F_u$. The resulting function
still satisfies the bound~\eqref{e:F-u-derb}.)

By~\eqref{e:A-prop} we have
$$
z-y=\begin{pmatrix}2 & 0 \\ 0 & 1/2\end{pmatrix}(w-x)+\mathcal O(\delta)|w-x|.
$$
Since $|w-x|=d(w,W_u)$ this implies
\begin{align}
  \label{e:frap-1}
z_1-y_1&=\mathcal O(\delta)d(w,W_u),\\
  \label{e:frap-2}
z_2-F_u(y_1)&={1\over 2}(w_2-F_u(w_1))+\mathcal O(\delta)d(w,W_u).
\end{align}
It follows from~\eqref{e:F-u-derb} and~\eqref{e:frap-1} that
$$
F_u(z_1)-F_u(y_1)=\mathcal O(\delta)d(w,W_u).
$$
From here and~\eqref{e:frap-2} we obtain
$$
z_2-F_u(z_1)={1\over 2}(w_2-F_u(w_1))+\mathcal O(\delta)d(w,W_u).
$$
Since $d(w,W_u)=|w_2-F_u(w_1)|$ and $d(z,W_u)=|z_2-F_u(z_1)|$ this
implies~\eqref{e:dyndyn-2}.
\end{proof}
\begin{figure}
\includegraphics[scale=0.35]{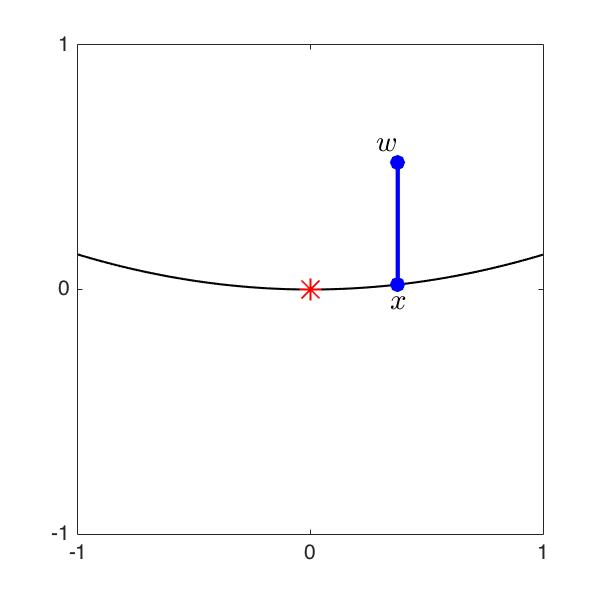}
\includegraphics[scale=0.35]{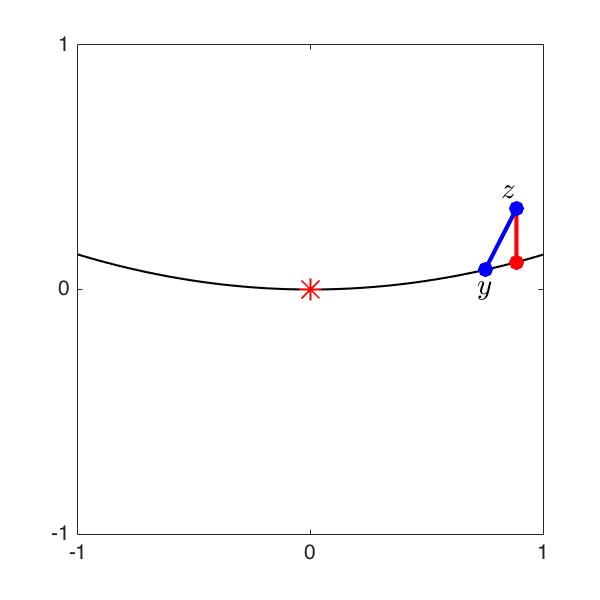}
\caption{An illustration of Lemma~\ref{l:dyndyn-2}. The curve
is $W_u$ and the picture on the right is the image
of the picture on the left under $\varphi$. The blue
segment on the left has length $d(w,W_u)$
and the red segment on the right has length $d(z,W_u)$.}
\label{f:exc-2}
\end{figure}
We now finish the proof of part~(3) of Theorem~\ref{t:stun-2}.
Assume that $w\in \overline B_\infty(0,1)$ and
$$
w^{(n)}:=\varphi^{-n}(w)\in \overline B_\infty(0,1)\quad\text{for all }n\geq 0.
$$
Since $w^{(n-1)}=\varphi(w^{(n)})$,
by Lemma~\ref{l:dyndyn-2} for small enough $\delta$ we have
\begin{equation}
  \label{e:dd3-1}
d(w^{(n-1)},W_u)\leq {2\over 3}d(w^{(n)},W_u).
\end{equation}
Since $d(w^{(n)},W_u)\leq 2$ for all $n\geq 0$, we iterate this to get
\begin{equation}
  \label{e:dd3-2}
d(w,W_u)\leq 2\cdot \Big({2\over 3}\Big)^nd(w^{(n)},W_u)\leq \Big({2\over 3}\Big)^n\quad\text{for all }n\geq 0
\end{equation}
which implies $d(w,W_u)=0$ and thus $w\in W_u$.

\Remark The above proof in fact gives stronger versions of parts~(3) and~(4) of Theorem~\ref{t:stun-2}:
if $0\leq \sigma\leq 1$ and $n\geq 0$, then
\begin{align}
  \label{e:redoer-1}
w,\varphi^{-1}(w),\dots,\varphi^{-n}(w)\in \overline B_\infty(0,\sigma)\quad\Longrightarrow\quad
d(w,W_u)\leq (2/3)^n\cdot 2\sigma,\\
  \label{e:redoer-2}
w,\varphi(w),\dots,\varphi^{n}(w)\in \overline B_\infty(0,\sigma)\quad\Longrightarrow\quad
d(w,W_s)\leq (2/3)^n\cdot 2\sigma.
\end{align}
Here we define $d(w,W_s):=|w_1-F_s(w_2)|$ similarly to~\eqref{e:unstable-distance}.

Another version of~\eqref{e:redoer-1}, \eqref{e:redoer-2} is available using the following estimate:
\begin{equation}
  \label{e:redoer-3}
\begin{aligned}
|w|\leq 4\big(d(w,W_u)+d(w,W_s)\big)\quad\text{for all}\quad w\in \overline B_\infty(0,1).
\end{aligned}
\end{equation}
To prove~\eqref{e:redoer-3} we use~\eqref{e:F-u-derb} and~\eqref{e:F-comp-contra}:
$$
\begin{aligned}
{1\over 2}|w_1|&\leq |w_1-F_s(F_u(w_1))|\leq
|w_1-F_s(w_2)|+|F_s(w_2)-F_s(F_u(w_1))|
\\&\leq |w_1-F_s(w_2)|+|w_2-F_u(w_1)|=d(w,W_s)+d(w,W_u)
\end{aligned}
$$
and $|w_2|$ is estimated similarly.

Combining~\eqref{e:redoer-1}--\eqref{e:redoer-3} we get the following bound:
for $n,r\geq 0$ and $0\leq\sigma\leq 1$
\begin{equation}
  \label{e:redoer-4}
\begin{gathered}
\varphi^{-n}(w),\dots,\varphi^{-1}(w),w,\varphi(w),\dots,\varphi^r(w)\in \overline B_\infty(0,\sigma)\\\Longrightarrow\quad
|w|\leq \big((2/3)^n+(2/3)^r\big)\cdot 8\sigma.
\end{gathered}
\end{equation}
The bound~\eqref{e:redoer-4} can be interpreted as `long time strict convexity': if the trajectory
of a point $w$ stays in the ball $\overline B_\infty(0,1)$ for long positive and negative times,
then $w$ is close to 0.

\section{The general setting}
  \label{s:general}

In this section we explain how to extend the proofs of Theorems~\ref{t:stun-1}
and~\ref{t:stun-2} to general families of hyperbolic transformations, yielding
the general Theorem~\ref{t:stun-adv} stated in~\S\ref{s:reduce-5}.
Rather than give a complete formal proof we explain below
several generalizations of Theorems~\ref{t:stun-1} and~\ref{t:stun-2} which together
give Theorem~\ref{t:stun-adv}. We refer the reader to~\cite[Theorem~6.2.8]{KaHa}
for a detailed proof.

\subsection{Making $\delta$ small by rescaling}
  \label{s:reduce-1}

We first show how to arrange for the assumption~(3) in~\S\ref{s:base-1}
(that is, $\varphi$ being close to its linearization $d\varphi(0)$) to hold
by a rescaling argument. Assume that $\varphi:U_\varphi\to V_\varphi$ is
a $C^{N+1}$ map and it satisfies assumptions~(1)--(2) in~\S\ref{s:base-1}, namely
\begin{equation}
  \label{e:stay-put}
\varphi(0)=0,\quad
d\varphi(0)=\begin{pmatrix} 2 & 0 \\ 0 & 1/2\end{pmatrix}.
\end{equation}
Fix small $\delta_1>0$ and consider the rescaling map
$$
T:\overline B_\infty(0,1)\to \overline B_\infty(0,\delta_1),\quad
T(x)=\delta_1 x.
$$
We conjugate $\varphi$ by $T$ to get the map
$$
\widetilde\varphi:=T^{-1}\circ\varphi\circ T:\overline B_\infty (0,1)\to \mathbb R^2.
$$
The map $\widetilde\varphi$ still satisfies~\eqref{e:stay-put}, and its higher derivatives
are given by
$$
\partial^\alpha\widetilde\varphi(x)=\delta_1^{|\alpha|-1}\partial^\alpha\varphi(\delta_1 x).
$$
Therefore, $\widetilde\varphi$ satisfies the assumption~(3) with $\delta=C\delta_1$,
where $C$ depends on $\varphi$. By the Inverse Mapping Theorem, $\widetilde\varphi$
also satisfies the assumption~(4); that is, it is a diffeomorphism $\widetilde U_\varphi\to \widetilde V_\varphi$
for some open sets $\widetilde U_\varphi,\widetilde V_\varphi$ containing $\overline B_\infty(0,1)$.

It follows that for $\delta_1$ small enough (depending on $\varphi$) Theorems~\ref{t:stun-1}
and~\ref{t:stun-2} apply to $\widetilde\varphi$, giving the unstable/stable manifolds
$\widetilde W_u,\widetilde W_s$. The manifolds
$$
W_{u,\delta_1}:=T(\widetilde W_u),\quad
W_{s,\delta_1}:=T(\widetilde W_s)
$$
satisfy the conclusions of Theorems~\ref{t:stun-1} and~\ref{t:stun-2} for $\varphi$
with the ball $\overline B_\infty(0,1)$ replaced by $\overline B_\infty(0,\delta_1)$.
We call these the ($\delta_1$-)\emph{local unstable/stable manifolds}
of $\varphi$ at~$0$.

\subsection{General expansion/contraction rates}
  \label{s:reduce-2}

We next explain why Theorems~\ref{t:stun-1}--\ref{t:stun-2} hold if the condition~\eqref{e:stay-put}
is replaced by
\begin{equation}
  \label{e:stay-put-2}
\varphi(0)=0,\quad
d\varphi(0)=\begin{pmatrix} \mu & 0 \\ 0 & \lambda\end{pmatrix}\quad\text{where }
0<\lambda<1<\mu\text{ are fixed.}
\end{equation}
Note that the value of $\delta$ for which Theorems~\ref{t:stun-1} and~\ref{t:stun-2} apply
will depend on $\lambda,\mu$, in particular it will go to~0 if $\lambda\to 1$
or $\mu\to 1$.

The proofs in~\S\ref{s:base} apply with the following changes:
\begin{itemize}
\item in the proof of Lemma~\ref{l:gmap}, \eqref{e:G-1-der} is replaced by
$$
\partial_{x_1}G_1(x_1)=\mu+\mathcal O(\delta)>1\quad\text{for all }x_1\in [-1,1];
$$
\item in Lemma~\ref{l:derform}, we have
$$
L_k(x_1,\tau_0,\dots,\tau_k)=\lambda\mu^{-k}\tau_k+\mathcal O(\delta)
$$
and the estimate~\eqref{e:derform-2} is changed accordingly;
\item in the estimates~\eqref{e:derest-1.1}, \eqref{e:derest-1.2},
and~\eqref{e:derest-2} the constant ${1\over 4}$ is replaced
by $\lambda\over\mu$;
\item in the contraction property~\eqref{e:contraction-achieved}
the constant $1\over 3$ is replaced by any fixed number in the interval
$({\lambda\over\mu},1)$;
\item in the estimate~\eqref{e:stun-2-1} in Theorem~\ref{t:stun-2}, as well as in Lemma~\ref{l:dyndyn-1},
the constant $1\over 2$ is replaced by~$\mu^{-1}$;
\item in the estimate~\eqref{e:stun-2-2} in Theorem~\ref{t:stun-2} the constant ${1\over 2}$ is replaced by~$\lambda$;
\item in Lemma~\ref{l:dyndyn-2} the constant ${1\over 2}$ is replaced by~$\lambda$;
\item in the estimates~\eqref{e:dd3-1} and~\eqref{e:dd3-2}, as well
as in~\eqref{e:redoer-1}, the constant ${2\over 3}$ is replaced
by any fixed number in the interval $(\lambda,1)$;
\item in~\eqref{e:redoer-2}, the constant $2\over 3$ is replaced by any fixed
number in the interval $(\mu^{-1},1)$;
\item in~\eqref{e:redoer-4}, the conclusion becomes
$|w|\leq (\tilde\lambda^n+\tilde\mu^{-r})\cdot 8\sigma$
where $\tilde\lambda\in (\lambda,1)$ and
$\tilde\mu\in (1,\mu)$ are fixed.
\end{itemize}

\subsection{Higher dimensions}
  \label{s:reduce-3}

We now generalize Theorems~\ref{t:stun-1} and~\ref{t:stun-2}
to the case of higher dimensions. More precisely, consider a diffeomorphism
$$
\varphi:U_\varphi\to V_\varphi,\quad
U_\varphi,V_\varphi\subset \mathbb R^d,\quad
\overline B_\infty(0,1)\subset U_\varphi\cap V_\varphi,
$$
where $d=d_u+d_s$, we write elements of $\mathbb R^d$
as $(x_1,x_2)$ with $x_1\in \mathbb R^{d_u}$, $x_2\in\mathbb R^{d_s}$,
and (with $|\bullet|$ denoting the Euclidean norm)
\begin{equation}
  \label{e:B-infty-new}
\overline B_\infty(0,r):=\{(x_1,x_2)\in\mathbb R^d\colon \max(|x_1|,|x_2|)\leq r\}.
\end{equation}
The condition~\eqref{e:stay-put} is replaced by
\begin{equation}
\label{e:stay-put-3}
\varphi(0)=0,\quad
d\varphi(0)=\begin{pmatrix}
A_1 & 0 \\ 0 & A_2
\end{pmatrix},
\end{equation}
where $A_1:\mathbb R^{d_u}\to\mathbb R^{d_u}$,
$A_2:\mathbb R^{d_s}\to\mathbb R^{d_s}$ are linear isomorphisms satisfying
\begin{equation}
\label{e:A-j-ineq}
\|A_1^{-1}\|\leq \mu^{-1},\quad
\|A_2\|\leq \lambda,\quad
\max(\|A_1\|,\|A_2^{-1}\|)\leq C_0
\end{equation}
for some fixed constants $\lambda,\mu,C_0$ such that $0<\lambda<1<\mu$. We assume
that the bounds~\eqref{e:hypas-2} on higher derivatives still hold
for some small $\delta>0$.

The definitions~\eqref{e:stun-graphs} of the unstable/stable graphs still apply
with the following adjustments. Define the unstable/stable balls
\begin{equation}
  \label{e:stab-the-ball}
\overline B_u(0,1):=\{x_1\in\mathbb R^{d_u}\colon |x_1|\leq 1\},\quad
\overline B_s(0,1):=\{x_2\in\mathbb R^{d_s}\colon |x_2|\leq 1\}.
\end{equation}
Then for a $C^N$ function $F:\overline B_u(0,1)\to\mathbb R^{d_s}$,
its unstable graph $\mathcal G_u(F)$ is a $d_u$-dimensional submanifold (with boundary) of $\mathbb R^d$.
If instead $F:\overline B_s(0,1)\to\mathbb R^{d_u}$,
then the stable graph $\mathcal G_s(F)$ is a $d_s$-dimensional submanifold of $\mathbb R^d$.

Theorems~\ref{t:stun-1} and~\ref{t:stun-2} still hold for the map $\varphi$,
with the constant $\delta$ now depending on $d,\lambda,\mu,C_0$ and the constants
in~\eqref{e:stun-2-1} and~\eqref{e:stun-2-2} modified as in~\S\ref{s:reduce-2}.
The proofs in~\S\ref{s:base} need to be modified as follows (in addition
to the changes described in~\S\ref{s:reduce-2}):
\begin{itemize}
\item  
in Lemma~\ref{l:gmap}, the invertibility of $G_1$ and the fact
that the image of $G_1$ contains $\overline B_u(0,1)$ follow from the 
standard
(contraction mapping principle) proof of the Inverse Mapping Theorem,
with the estimate~\eqref{e:G-1-der} replaced by
$$
\|\partial_{x_1}G_1(x_1)-A_1\|=\mathcal O(\delta)\quad\text{for all }x_1\in \overline B_u(0,1);
$$
\item we use the notation $\mathbf D^k_{x_1} F(x_1)=(\partial^\alpha_{x_1}F(x_1))_{|\alpha|=k}$,
giving a vector in a finite-dimensional space $\mathcal V^k$ which is identified
with the space of homogeneous polynomials in $d_u$ variables with values in $\mathbb R^{d_s}$ using the operation
$$
\mathbf D^k_{x_1} F(x_1)\cdot v:=\sum_{|\alpha|=k}\binom{k}{\alpha}\partial^\alpha_{x_1}F(x_1)v^\alpha
=\partial_t|_{t=0}(F(x_1+tv)),\quad v\in\mathbb R^{d_u};
$$
\item we use the following norms on $\mathcal V^k$:
\begin{equation}
  \label{e:new-norms}
\|\mathbf D^k_{x_1} F(x_1)\|:=\sup\big\{|\mathbf D^k_{x_1}F(x_1)\cdot v|\colon v\in\mathbb R^{d_u},\ |v|=1\big\};
\end{equation}
\item in Lemma~\ref{l:derform}, the derivative bounds in~\eqref{e:gmap-2}
are now on $\|\mathbf D^j_{x_1} F(x_1)\|$ for $j=1,\dots,k$;
\item in Lemma~\ref{l:derform},
the derivative formula~\eqref{e:derform} is replaced
by
$$
\mathbf D^k_{x_1}(\Phi_u F)(y_1)=L_k(x_1,F(x_1),\mathbf D^1_{x_1}F(x_1),\dots,\mathbf D^k_{x_1}F(x_1)),\quad
x_1:=G_1^{-1}(y_1)
$$
where $L_k(x_1,\tau_0,\tau_1,\dots,\tau_k)\in\mathcal V^k$ is defined for
$x_1\in \overline B_{\mathbb R^{d_u}}(0,1)$, $\tau_0\in\overline B_{\mathbb R^{d_s}}(0,1)$
and $\tau_j\in \mathcal V^j$, $\|\tau_j\|\leq 1$ for $j=1,\dots, k$;
\item in Lemma~\ref{l:derform}, the approximation for $L_k$ is changed to
the following:
$$
L_k(x_1,\tau_0,\tau_1,\dots,\tau_k)\cdot v
=A_2\big(\tau_k\cdot(A_1^{-1}v)\big)+\mathcal O(\delta)\quad\text{for all }
v\in\mathbb R^{d_u},\ |v|=1
$$
and the derivative estimates~\eqref{e:derform-2} generalize naturally to the case of higher dimensions;
\item
in the proof of Lemma~\ref{l:derform}, the formula~\eqref{e:lulla} is replaced by
$$
L_1(x_1,\tau_0,\tau_1)=(A_{21}(x_1,\tau_0)+A_{22}(x_1,\tau_0)\tau_1)(A_{11}(x_1,\tau_0)+A_{12}(x_1,\tau_0)\tau_1)^{-1}
$$
where $A_{jk}(x_1,\tau_0)$ and $\tau_1$ are now matrices;
\item in the definitions~\eqref{e:C-k-norm} and~\eqref{e:C-k-1-norm}
of $\|F\|_{C^k}$ and~$\|F\|_{C^{k,1}}$ we use
the norms~\eqref{e:new-norms} for the derivatives $\mathbf D^j_{x_1}F$,
and we have the revised inequality $\|F\|_{C^{k,1}}\leq C\|F\|_{C^{k+1}}$;
\item in the proof of Lemma~\ref{l:dyndyn-1}, the equation~\eqref{e:dyndyn-1.3}
is replaced by
$$
y_1-\tilde y_1=A_1(x_1-\tilde x_1)+\mathcal O(\delta)|x_1-\tilde x_1|;
$$
\item in the proof of Lemma~\ref{l:dyndyn-2}, the equation~\eqref{e:frap-2}
is replaced by
$$
z_2-F_u(y_1)=A_2(w_2-F_u(w_1))+\mathcal O(\delta)d(w,W_u).
$$
\end{itemize}

\subsection{Iterating different transformations}
  \label{s:reduce-4}

We next discuss a generalization of Theorems~\ref{t:stun-1} and~\ref{t:stun-2}
from the case of a single map $\varphi$ to a $\mathbb Z$-indexed family of maps.
More precisely, we assume that
$$
\varphi_m:U_\varphi\to V_\varphi,\quad
m\in\mathbb Z
$$
is a family of maps each of which satisfies the assumptions in~\S\ref{s:reduce-3}
uniformly in $m$, that is the constants $\lambda,\mu,C_0,\delta$ are independent of~$m$.
The linear maps $A_1,A_2$ in~\eqref{e:stay-put-3} are allowed to depend on~$m$.

We explain how the construction of the unstable manifold in Theorem~\ref{t:stun-1}
generalizes to the case of a family of transformations. The case of stable
manifolds is handled similarly, and Theorem~\ref{t:stun-2} generalizes naturally
to this setting, see~\S\ref{s:reduce-5} below.

Instead of a single function $F_u$ we construct
a family of functions
$$
F^u_m:\overline B_u(0,1)\to\overline B_s(0,1),\quad
m\in\mathbb Z;\quad
F^u_m(0)=0,\quad
dF^u_m(0)=0.
$$
Denote the graphs $W^u_m:=\mathcal G_u(F^u_m)$. Then the invariance property~\eqref{e:graph-inv}
generalizes to
\begin{equation}
  \label{e:graph-inv-2}
\varphi_m(W^u_m)\cap \overline B_\infty(0,1)=W^u_{m+1}.
\end{equation}
To construct $F^u_m$ we use the graph transform $\Phi^u_m$ of the map $\varphi_m$
defined in Lemma~\ref{l:gmap}. This transform satisfies the derivative bounds
of Lemmas~\ref{l:derest-1}--\ref{l:derest-2} uniformly in $m$.

As in~\S\ref{s:base-3}, to show~\eqref{e:graph-inv-2} it suffices to construct $F^u_m$ such that
for all $m\in\mathbb Z$,
\begin{equation}
  \label{e:graph-inv-2.1}
F^u_m\in\mathcal X_N,\quad
\Phi^u_mF^u_m=F^u_{m+1}.
\end{equation}
To do this we modify the argument of~\S\ref{s:base-3} as follows:
consider the space
$$
\mathcal X_N^{\mathbb Z}:=\{(F_m)_{m\in\mathbb Z}\mid F_m\in\mathcal X_N\text{ for all }m\in\mathbb Z\}
$$
with the metric
$$
d_N^{\mathbb Z}\big((F_m),(\widetilde F_m)\big):=\sup_{m\in\mathbb Z}d_N(F_m,\widetilde F_m).
$$
Then $(\mathcal X_N^{\mathbb Z},d_N^{\mathbb Z})$ is a complete metric space.
Consider the map on $\mathcal X_N^{\mathbb Z}$
$$
\Phi_u^{\mathbb Z}:(F_m)\mapsto (\widehat F_m),\quad
\widehat F_{m+1}:=\Phi^u_mF_m.
$$
It follows from Lemma~\ref{l:derest-2} that $\Phi_u^{\mathbb Z}$
is a contracting map on $(\mathcal X_N^{\mathbb Z},d_N^{\mathbb Z})$. Applying
the Contraction Mapping Principle, we obtain a fixed point
$$
(F^u_m)\in\mathcal X_N^{\mathbb Z},\quad
\Phi_u^{\mathbb Z}(F^u_m)=(F^u_m)
$$
which satisfies~\eqref{e:graph-inv-2.1}, finishing the proof.

\subsection{The general Stable/Unstable Manifold Theorem}
  \label{s:reduce-5}

We finally combine the generalizations in~\S\S\ref{s:reduce-2}--\ref{s:reduce-4}
and state the general version of the Stable/Unstable Manifold Theorem. We assume that:
\begin{enumerate}
\item $d=d_u+d_s$, $d_u,d_s\geq 0$, elements of $\mathbb R^d$ are written
as $(x_1,x_2)$ where $x_1\in\mathbb R^{d_u}$, $x_2\in\mathbb R^{d_s}$, we use the Euclidean norm
on $\mathbb R^d$, and $\overline B_\infty(0,1)$, $\overline B_u(0,1)$,
$\overline B_s(0,1)$ are defined by~\eqref{e:B-infty-new}, \eqref{e:stab-the-ball};
\item we are given a family of $C^{N+1}$ diffeomorphisms (here $N\geq 1$ is fixed)
\begin{equation}
  \label{e:gst-1}
\varphi_m:U_\varphi\to V_\varphi,\quad
m\in\mathbb Z;\quad
U_\varphi,V_\varphi\subset\mathbb R^d,\quad
\overline B_\infty(0,1)\subset U_\varphi\cap V_\varphi;
\end{equation}
\item we have for all $m\in\mathbb Z$
\begin{equation}
  \label{e:gst-2}
\varphi_m(0)=0,\quad
d\varphi_m(0)(x_1,x_2)=(A_{1,m}x_1,A_{2,m}x_2)
\end{equation}
where the linear maps $A_{1,m}:\mathbb R^{d_u}\to\mathbb R^{d_u}$,
$A_{2,m}:\mathbb R^{d_s}\to\mathbb R^{d_s}$ satisfy
\begin{equation}
  \label{e:gst-3}
\|A_{1,m}^{-1}\|\leq \mu^{-1},\quad
\|A_{2,m}\|\leq \lambda,\quad
\max(\|A_{1,m}\|,\|A_{2,m}^{-1}\|)\leq C_0
\end{equation}
for some constants $0<\lambda<1<\mu$, $C_0> 0$;
\item we have the derivative bounds for some $\delta>0$
\begin{equation}
  \label{e:gst-4}
\sup_{U_\varphi}|\partial^\alpha \varphi_m|\leq \delta\quad\text{for all}\quad
m\in\mathbb Z,\quad
2\leq|\alpha|\leq N+1.
\end{equation}
\end{enumerate}
Note that the stable/unstable spaces at~$0$ are now given by
$$
E_u(0):=\{(x_1,0)\mid x_1\in\mathbb R^{d_u}\},\quad
E_s(0):=\{(0,x_2)\mid x_2\in\mathbb R^{d_s}\}.
$$
The general form of Theorems~\ref{t:stun-1} and~\ref{t:stun-2} is then
\begin{theo}
  \label{t:stun-adv}
There exists $\delta>0$ depending only on $d,N,\lambda,\mu,C_0$ such that
the following holds. Assume that $\varphi_m$ satisfy assumptions~(1)--(4) of the present section.
Then there exist families of $C^N$ functions
\begin{equation}
  \label{e:fufsy}
\begin{aligned}
F^u_m:\overline B_u(0,1)\to\overline B_s(0,1),&\quad
F^s_m:\overline B_s(0,1)\to\overline B_u(0,1),\\
F^u_m(0)=0,\quad
dF^u_m(0)=0,&\quad
F^s_m(0)=0,\quad
dF^s_m(0)=0
\end{aligned}
\end{equation}
bounded in $C^N$ uniformly in $m$ and such that the graphs
\begin{equation}
  \label{e:wuwsy}
\begin{aligned}
W^u_m&:=\{(x_1,x_2)\colon |x_1|\leq 1,\ x_2=F^u_m(x_1)\},\\
W^s_m&:=\{(x_1,x_2)\colon |x_2|\leq 1,\ x_1=F^s_m(x_2)\}
\end{aligned}
\end{equation}
have the following properties for all $m\in\mathbb Z$:
\begin{enumerate}
\item $\varphi_m^{-1}(W^u_{m+1})\subset W^u_m$ and $\varphi_m(W^s_m)\subset W^s_{m+1}$, more precisely
\begin{equation}
  \label{e:stuna-1}
\varphi_m(W^u_m)\cap \overline B_\infty(0,1)=W^u_{m+1},\quad
\varphi_m^{-1}(W^s_{m+1})\cap \overline B_\infty(0,1)=W^s_m;
\end{equation}
\item $W^u_m\cap W^s_m=\{0\}$;
\item for each $x\in W^u_m$ we have $\varphi_{m-n}^{-1}\cdots\varphi_{m-1}^{-1}(x)\to 0$
 as $n\to\infty$;
\item for each $x\in W^s_m$ we have $\varphi_{m+n-1}\cdots\varphi_m(x)\to 0$ as $n\to\infty$;
\item if $\varphi_{m-n}^{-1}\cdots\varphi_{m-1}^{-1}(x)\in \overline B_\infty(0,1)$ for all $n\geq 0$, then
$x\in W^u_m$;
\item if $\varphi_{m+n-1}\cdots\varphi_m(x)\in\overline B_\infty(0,1)$ for all $n\geq 0$,
then $x\in W^s_m$.
\end{enumerate}
\end{theo}
\Remarks
1. Similarly to~\eqref{e:stun-char}, we obtain the following dynamical
definition of the manifolds $W^u_m,W^s_m$:
\begin{equation}
  \label{e:stun-char-2}
\begin{aligned}
x\in W^u_m&\quad\Longleftrightarrow\quad
\varphi_{m-n}^{-1}\cdots\varphi_{m-1}^{-1}(x)\in \overline B_\infty(0,1)\quad\text{for all }n\geq 0;\\
x\in W^s_m&\quad\Longleftrightarrow\quad
\varphi_{m+n-1}\cdots\varphi_{m}(x)\in \overline B_\infty(0,1)\quad\text{for all }n\geq 0.
\end{aligned}
\end{equation}

\noindent 2. Similarly to Theorem~\ref{t:stun-2} and Lemma~\ref{l:dyndyn-1}, parts~(3) and~(4) of Theorem~\ref{t:stun-adv}
can be made quantitative as follows. Fix $\tilde\lambda,\tilde\mu$ such that
\begin{equation}
  \label{e:tildes}
0<\lambda<\tilde\lambda<1<\tilde\mu<\mu.
\end{equation}
Then for $\delta$ small enough depending on $d,\lambda,\tilde\lambda,\mu,\tilde\mu,C_0$ and
all $m\in\mathbb Z$, $n\geq 0$ we have
\begin{equation}
  \label{e:stuna-2}
\begin{aligned}
x,\tilde x\in W^u_m&\quad\Longrightarrow\quad
|\varphi_{m-n}^{-1}\cdots\varphi_{m-1}^{-1}(x)-
\varphi_{m-n}^{-1}\cdots\varphi_{m-1}^{-1}(\tilde x)|\leq \tilde\mu^{-n}|x-\tilde x|,\\
x,\tilde x\in W^s_m&\quad\Longrightarrow\quad
|\varphi_{m+n-1}\cdots\varphi_{m}(x)-
\varphi_{m+n-1}\cdots\varphi_{m}(\tilde x)|\leq \tilde\lambda^{n}|x-\tilde x|.
\end{aligned}
\end{equation}

\noindent 3. Similarly to the remark at the end of~\S\ref{s:base-4}, there is a quantitative
version of parts~(5) and~(6) of Theorem~\ref{t:stun-adv} as well. Namely, if
$\tilde\lambda,\tilde\mu$ satisfy~\eqref{e:tildes} and 
$\delta$ is small enough depending on $d,\lambda,\tilde\lambda,\mu,\tilde\mu,C_0$, then
for all $0\leq\sigma\leq 1$ and $n\geq 0$
\begin{equation}
  \label{e:stuna-3}
\begin{aligned}
\varphi_{m-\ell}^{-1}\cdots\varphi_{m-1}^{-1}(x)\in \overline B_\infty(0,\sigma),\quad\ell=0,1,\dots,n
&\quad\Longrightarrow\quad
d(x,W_u)\leq \tilde\lambda^n\cdot 2\sigma;\\
\varphi_{m+\ell-1}\cdots\varphi_{m}(x)\in \overline B_\infty(0,\sigma),\quad\ell=0,1,\dots,n
&\quad\Longrightarrow\quad
d(x,W_s)\leq \tilde\mu^{-n}\cdot 2\sigma.
\end{aligned}
\end{equation}
The following analog of the estimate~\eqref{e:redoer-4} holds:
for all $n,r\geq 0$ and $0\leq\sigma\leq 1$
\begin{equation}
  \label{e:stuna-4}
\begin{gathered}
\text{if}\quad
\varphi_{m-\ell}^{-1}\cdots\varphi_{m-1}^{-1}(x)\in \overline B_\infty(0,\sigma),\quad \ell=0,1,\dots,n\\
\text{and}\quad
\varphi_{m+\ell-1}\cdots\varphi_{m}(x)\in \overline B_\infty(0,\sigma),\quad\ell=0,1,\dots,r\\
\text{then}\quad
|x|\leq \big(\tilde\lambda^n+\tilde\mu^{-r}\big)\cdot 8\sigma.
\end{gathered}
\end{equation}

\noindent 4. The rescaling argument of~\S\ref{s:reduce-1} applies to the setting
of Theorem~\ref{t:stun-adv}. That is, if $\varphi_m$ satisfy~\eqref{e:gst-1}--\eqref{e:gst-3},
then one can make~\eqref{e:gst-4} hold by zooming in to a sufficiently small neighborhood of the origin.

\subsection{The case of expansion/contraction rate~1}
  \label{s:reduce-6}

For applications to hyperbolic flows (which have a neutral direction) we also
need to discuss which parts of Theorem~\ref{t:stun-adv} still hold when either
$\lambda$ or $\mu$ is equal to~1.
Specifically, we replace the condition $\lambda<1<\mu$ with
\begin{equation}
  \label{e:degrade-1}
\lambda=1<\mu.
\end{equation}
That is, there is still expansion in the unstable directions but there does not have to be
(strict) contraction in the stable directions.

The construction of the unstable manifolds $W^u_m$ applies to the case~\eqref{e:degrade-1}
without any changes, and the resulting manifolds $W^u_m$ satisfy
conclusions~(1) and~(3) of Theorem~\ref{t:stun-adv}.
(In fact, this would work under an even weaker condition $\mu>\max(1,\lambda)$.)
The estimate~\eqref{e:stuna-2} still holds for $W^u_m$, assuming that
$\tilde\mu$ satisfies $1<\tilde\mu<\mu$.

However, we cannot construct the stable manifolds $W^s_m$ under the assumption~\eqref{e:degrade-1}.
The problem is in the proof of Lemma~\ref{l:gmap}: the projection of $\varphi^{-1}(\mathcal G_s(F))$
onto the $x_2$ variable might no longer cover the unit ball,
so the function $\Phi_sF$ cannot be defined.
Moreover, conclusion~(5) of Theorem~\ref{t:stun-adv} (as well as~\eqref{e:stuna-3}) no longer holds,
so the dynamical characterization~\eqref{e:stun-char-2}
of the unstable manifolds $W^u_m$ is no longer valid.
The `strict convexity' property~\eqref{e:stuna-4} no longer holds.

Similarly, one can still construct the stable manifolds $W^s_m$ and
establish conclusions~(1) and~(4) of Theorem~\ref{t:stun-adv} for them if we replace the condition
$\lambda<1<\mu$ with
\begin{equation}
  \label{e:degrade-2}
\lambda<1=\mu.
\end{equation}

\section{Hyperbolic maps and flows}
\label{s:maps-and-flows}

In this section we apply Theorem~\ref{t:stun-adv} to obtain
the Stable/Unstable Manifold Theorem for hyperbolic maps
(Theorem~\ref{t:stun-maps} in~\S\ref{s:hyp-maps}) and for hyperbolic flows (Theorem~\ref{t:stun-flows}
in~\S\ref{s:hyp-flows}). This involves constructing an adapted
metric (Lemma~\ref{l:adapted-metric}) and taking adapted coordinates
to bring the map/flow into the model case handled by Theorem~\ref{t:stun-adv}.

\subsection{Hyperbolic maps}
  \label{s:hyp-maps}

Assume that $M$ is a $d$-dimensional manifold without boundary
and $d=d_u+d_s$ where $d_u,d_s\geq 0$. Let $\varphi:M\to M$ be a $C^{N+1}$ diffeomorphism
(here $N\geq 1$ is fixed) and assume that $\varphi$ is hyperbolic on some
compact $\varphi$-invariant set $K\subset M$ in the following sense:
\begin{defi}
\label{d:hyp-map}
Let $K\subset M$ be a compact set such that $\varphi(K)=K$. We say that $\varphi$
is \textbf{hyperbolic} on~$K$ if there exists a splitting
\begin{equation}
  \label{e:hyp-map-1}
T_x M=E_u(x)\oplus E_s(x),\quad
x\in K
\end{equation}
where $E_u(x),E_s(x)\subset T_x M$ are subspaces of dimensions $d_u,d_s$
and:
\begin{itemize}
\item $E_u,E_s$ are invariant under $d\varphi$, namely
\begin{equation}
  \label{e:hyp-map-2}
d\varphi(x)E_u(x)=E_u(\varphi(x)),\quad
d\varphi(x)E_s(x)=E_s(\varphi(x))\quad\text{for all }x\in K;
\end{equation}
\item large negative iterates of $\varphi$ are contracting
on $E_u$, namely there exist constants $C>0$, $0<\lambda<1$ such that 
for some Riemannian metric $|\bullet|$ on $M$ we have
\begin{equation}
  \label{e:hyp-map-u}
|d\varphi^{-n}(x)v|\leq C\lambda^n |v|\quad\text{for all}\quad
v\in E_u(x),\ x\in K,\ n\geq 0;
\end{equation}
\item large positive iterates of $\varphi$ are contracting on $E_s$, namely
\begin{equation}
  \label{e:hyp-map-s}
|d\varphi^{n}(x)v|\leq C\lambda^n |v|\quad\text{for all}\quad
v\in E_s(x),\ x\in K,\ n\geq 0.
\end{equation}
\end{itemize}
\end{defi}
\Remarks 1. The contraction properties~\eqref{e:hyp-map-u}, \eqref{e:hyp-map-s}
do not depend on the choice of the metric on $M$, though the constant $C$ (but not $\lambda$)
will depend on the metric. Later in Lemma~\ref{l:adapted-metric} we construct metrics
which give~\eqref{e:hyp-map-u}, \eqref{e:hyp-map-s} with $C=1$.

\noindent 2. We do not assume a priori that the maps
$x\mapsto E_u(x),E_s(x)$ are continuous. However we show in~\S\ref{s:continuity} below
that Definition~\ref{d:hyp-map} implies that these maps are in fact H\"older continuous.

\noindent 3. The basic example of a hyperbolic set is $K=\{x_0\}$ where $x_0\in M$
is a fixed point of~$\varphi$ which is hyperbolic, namely
$d\varphi(x_0)$ has no eigenvalues on the unit circle.
More generally one can take as $K$ a hyperbolic closed trajectory of $\varphi$.
The opposite situation is when $K=M$;
in this case $\varphi$ is called an \emph{Anosov diffeomorphism}.

\noindent 4. Similarly to~\S\ref{s:reduce-5} we could take two different constants
in~\eqref{e:hyp-map-u}, \eqref{e:hyp-map-s}, corresponding to different minimal contraction
rates in the unstable and the stable directions. We do not do this here to simplify
notation, and since the examples in this note have time-reversal symmetry and
thus equal stable/unstable contraction rates.

Fix
a Riemannian metric on $M$ which induces a distance function $d(\bullet,\bullet)$;
denote for $x\in M$ and $r\geq 0$
\begin{equation}
  \label{e:B-d-def}
\overline B_d(x,r):=\{y\in M\mid d(x,y)\leq r\}.
\end{equation}
We now state the Stable/Unstable Manifold Theorem for hypebolic maps:
\begin{theo}
  \label{t:stun-maps}
Assume that $\varphi$ is hyperbolic on $K\subset M$.
Then for each $x\in K$ there exist \textbf{local unstable/stable manifolds}
$$
W_u(x),W_s(x)\subset M
$$
which have the following properties for
some $\varepsilon_0>0$
depending only on $\varphi,K$:
\begin{enumerate}
\item $W_u(x),W_s(x)$ are $C^N$ embedded disks of dimensions $d_u,d_s$,
that is images of closed balls in $\mathbb R^{d_u}$, $\mathbb R^{d_s}$
under $C^N$ embeddings, and the $C^N$ norms of these embeddings
are bounded uniformly in $x$;
\item $W_u(x)\cap W_s(x)=\{x\}$ and $T_xW_u(x)=E_u(x)$, $T_x W_s(x)=E_s(x)$;
\item the boundaries of $W_u(x),W_s(x)$ do not intersect $\overline B_d(x,\varepsilon_0)$;
\item $\varphi^{-1}(W_u(x))\subset W_u(\varphi^{-1}(x))$ and $\varphi(W_s(x))\subset W_s(\varphi(x))$;
\item for each $y\in W_u(x)$, we have $d(\varphi^{-n}(y),\varphi^{-n}(x))\to 0$ as $n\to\infty$;
\item for each $y\in W_s(x)$, we have $d(\varphi^n(y),\varphi^n(x))\to 0$ as $n\to\infty$;
\item if $y\in M$ and $d(\varphi^{-n}(y),\varphi^{-n}(x))\leq \varepsilon_0$
for all $n\geq 0$, then $y\in W_u(x)$;
\item if $y\in M$ and $d(\varphi^n(y),\varphi^n(x))\leq \varepsilon_0$
for all $n\geq 0$, then $y\in W_s(x)$;
\item if $x,y\in K$ and $d(x,y)\leq\varepsilon_0$ then $W_s(x)\cap W_u(y)$ consists
of exactly one point.
\end{enumerate} 
\end{theo}
\Remarks 1. Similarly to~\eqref{e:stuna-2} there are quantitative versions of the statements~(5)
and~(6): if we fix $\tilde\lambda$ such that $\lambda<\tilde\lambda<1$ then
for all $n\geq 0$ and $x\in K$
\begin{equation}
  \label{e:cream-1}
\begin{aligned}
y,\tilde y\in W_u(x)&\quad\Longrightarrow\quad
d(\varphi^{-n}(y),\varphi^{-n}(\tilde y))\leq C\tilde\lambda^n d(y,\tilde y);\\
y,\tilde y\in W_s(x)&\quad\Longrightarrow\quad
d(\varphi^{n}(y),\varphi^{n}(\tilde y))\leq C\tilde\lambda^n d(y,\tilde y).
\end{aligned}
\end{equation}
where $C$ is a constant depending only on $\varphi,K,\tilde\lambda$.
Here and in Remark~2 below the manifolds $W_u,W_s$
and the constant $\varepsilon_0$ depend on $\tilde\lambda$,
in particular when $\tilde\lambda\to \lambda$ the stable/unstable manifolds might
degenerate to a point and $\varepsilon_0$ might go to 0.
(One can get rid of this dependence, but it is rather tedious
and typically unnecessary.)

\noindent 2. Similarly to~\eqref{e:stuna-3} there are quantitative versions of the statements~(7) and~(8)
as well: if we fix $\tilde\lambda$ as before then
for all $n\geq 0$, $0\leq\sigma\leq\varepsilon_0$, $x\in K$, and $y\in M$
\begin{equation}
  \label{e:cream-2}
\begin{aligned}
d(\varphi^{-\ell}(y),\varphi^{-\ell}(x))\leq \sigma\text{ for all }\ell=0,\dots,n
&\quad\Longrightarrow\quad
d(y,W_u(x))\leq C\tilde\lambda^n\sigma,\\
d(\varphi^{\ell}(y),\varphi^{\ell}(x))\leq \sigma\text{ for all }\ell=0,\dots,n
&\quad\Longrightarrow\quad
d(y,W_s(x))\leq C\tilde\lambda^n\sigma
\end{aligned}
\end{equation}
where $C$ is a constant depending only on $\varphi,K,\tilde\lambda$.
We also have the following analog of the `strict convexity' property~\eqref{e:stuna-4}:
for all $n\geq 0$, $0\leq\sigma\leq\varepsilon_0$, $x\in K$, and $y\in M$
\begin{equation}
  \label{e:cream-3}
d(\varphi^\ell(y),\varphi^\ell(x))\leq\sigma\quad\text{for all}\quad|\ell|\leq n
\quad\Longrightarrow\quad
d(y,x)\leq C\tilde\lambda^n\sigma.
\end{equation}

\begin{figure}
\includegraphics[scale=0.2875]{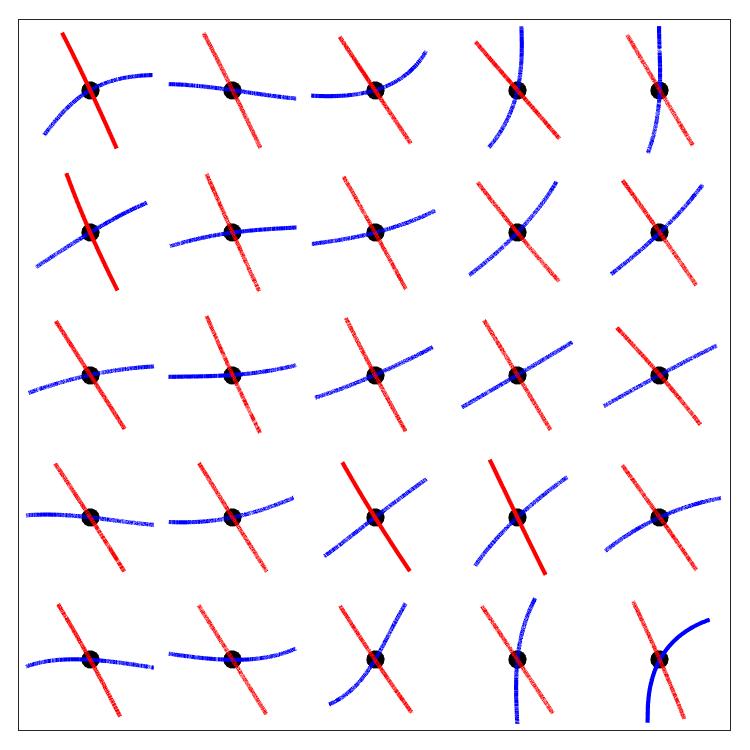}
\includegraphics[scale=0.2875]{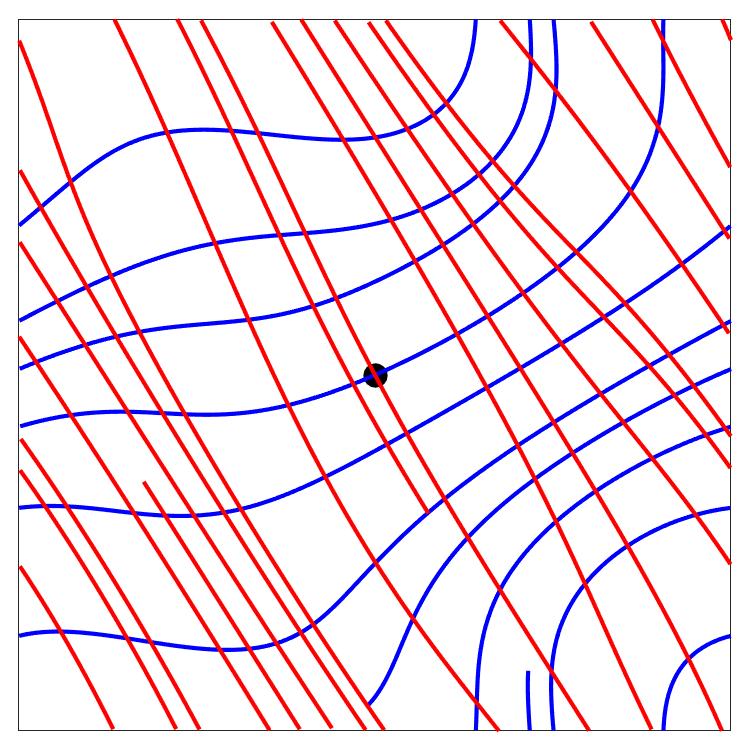}
\caption{Numerically computed unstable (blue) and stable (red) manifolds for a perturbed Arnold cat map on the torus
$\mathbb R^2/\mathbb Z^2$. On the left are the local stable/unstable manifolds $W_u(x),W_s(x)$
for several choices of $x$.
On the right are the manifolds $W_u^{(k)}(x),W_s^{(k)}(x)$ for $x=(0.5,0.5)$ and $k=4$.}
\label{f:stun-global}
\end{figure}

\noindent 3. The manifolds $W_u(x)$ are not defined canonically since they have
a somewhat arbitrarily determined boundary. For the same reason, if
$W_u(x)\cap W_u(y)\neq \emptyset$, this does not imply that $W_u(x)=W_u(y)$.
However, we can show that in this case $W_u(x)$ and $W_u(y)$ are subsets
of the same $d_u$-dimensional manifold, see~\eqref{e:intersector} below.
For $k\geq 0$ and $x\in K$ define
\begin{equation}
  \label{e:iterated-unstable}
W_u^{(k)}(x):=\varphi^k(W_u(\varphi^{-k}(x))).
\end{equation}
This is still a $d_u$-dimensional embedded disk in $M$. Moreover, by statement~(4)
in Theorem~\ref{t:stun-maps} we have
$$
W_u(x)=W_u^{(0)}(x)\subset W_u^{(1)}(x)\subset\dots\subset W_u^{(k)}(x)\subset W_u^{(k+1)}(x)\subset\dots
$$
There exists $k_0\geq 0$ such that for all $x,y\in K$
\begin{equation}
  \label{e:intersector}
W_u(x)\cap W_u(y)\neq\emptyset\quad\Longrightarrow\quad
W_u(x)\cup W_u(y)\subset W_u^{(k_0)}(x).
\end{equation}
Indeed, assume that $z\in W_u(x)\cap W_u(y)$.
By~\eqref{e:cream-1}
if $k_0$ is large enough then
$$
d(\varphi^{-k}(x),\varphi^{-k}(y))
\leq d(\varphi^{-k}(z),\varphi^{-k}(x))
+d(\varphi^{-k}(z),\varphi^{-k}(y))
\leq{\varepsilon_0\over 2}\quad\text{for all }k\geq k_0.
$$
Let $w\in W_u(x)\cup W_u(y)$. Then for $k_0$ large enough we get from~\eqref{e:cream-1}
\begin{equation}
  \label{e:prada}
d(\varphi^{-k}(w),\varphi^{-k}(x))\leq \varepsilon_0\quad\text{for all }k\geq k_0.
\end{equation}
It follows from statement~(7) in Theorem~\ref{t:stun-maps} that
$\varphi^{-k_0}(w)\in W_u(\varphi^{-k_0}(x))$ and thus
$w\in W_u^{(k_0)}(x)$, proving~\eqref{e:intersector}.

Note that~\eqref{e:intersector} implies that the tangent spaces to $W_u(x)$ are the unstable spaces:
\begin{equation}
  \label{e:tangentor}
y\in W_u(x)\cap K\quad\Longrightarrow\quad T_y(W_u(x))=E_u(y).
\end{equation}
The above discussion applies to stable manifolds where we define
\begin{equation}
  \label{e:iterated-stable}
W_s^{(k)}(x):=\varphi^{-k}(W_s(\varphi^k(x))).
\end{equation}

\noindent 4. Here is another version of `local uniqueness' of stable/unstable manifolds:
there exists $\varepsilon_1>0$ such that for all
$x,y\in K$ we have
\begin{align}
  \label{e:locally-unique-1}
W_u(x)\cap W_u(y)\neq\emptyset&\quad\Longrightarrow\quad
W_u(y)\cap \overline B_d(x,\varepsilon_1)\subset W_u(x),\\
  \label{e:locally-unique-2}
W_s(x)\cap W_s(y)\neq\emptyset&\quad\Longrightarrow\quad
W_s(y)\cap \overline B_d(x,\varepsilon_1)\subset W_s(x).
\end{align}
We show~\eqref{e:locally-unique-1}, with~\eqref{e:locally-unique-2} proved similarly.
Let $w\in W_u(y)\cap \overline B_d(x,\varepsilon_1)$. By~\eqref{e:prada}
we have $d(\varphi^{-k}(w),\varphi^{-k}(x))\leq\varepsilon_0$
for all $k\geq k_0$.
On the other hand
$$
d(\varphi^{-k}(w),\varphi^{-k}(x))\leq C d(w,x)\leq C\varepsilon_1\quad
\text{for}\quad 0\leq k <k_0.
$$
Choosing
$\varepsilon_1$ small enough we get
$d(\varphi^{-k}(w),\varphi^{-k}(x))\leq\varepsilon_0$ for all $k\geq 0$,
which by statement~(7) in Theorem~\ref{t:stun-maps}
gives $w\in W_u(x)$ as needed.

\noindent 5. One can take the unions of the manifolds~\eqref{e:iterated-unstable},
\eqref{e:iterated-stable} to obtain \emph{global unstable/stable manifolds}:
for $x\in K$,
\begin{equation}
  \label{e:global}
W_u^{(\infty)}(x):=\bigcup_{k\geq 0}W_u^{(k)}(x),\quad
W_s^{(\infty)}(x):=\bigcup_{k\geq 0}W_s^{(k)}(x).
\end{equation}
By statements (5)--(8) in Theorem~\ref{t:stun-maps} we can characterize
these dynamically as follows:
\begin{equation}
  \label{e:global-char}
\begin{aligned}
y\in W_u^{(\infty)}(x)&\quad\Longleftrightarrow\quad
d(\varphi^{-n}(y),\varphi^{-n}(x))\to 0\quad\text{as }n\to\infty;\\
y\in W_s^{(\infty)}(x)&\quad\Longleftrightarrow\quad
d(\varphi^{n}(y),\varphi^{n}(x))\to 0\quad\text{as }n\to\infty.
\end{aligned}
\end{equation}
Therefore the global stable/unstable manifolds are disjoint:
if $W_u^{(\infty)}(x)\cap W_u^{(\infty)}(y)\neq\emptyset$, then
$W_u^{(\infty)}(x)=W_u^{(\infty)}(y)$, and same is true for $W_s^{(\infty)}$.

The sets $W_u^{(\infty)}(x)$ and~$W_s^{(\infty)}(x)$ are $d_u$ and $d_s$-dimensional immersed
submanifolds without boundary in $M$, however they are typically not embedded.
In fact, in many cases these submanifolds are dense in $M$. See Figure~\ref{f:stun-global}.

\subsection{Continuity of the stable/unstable spaces}
  \label{s:continuity}

In this section we study the regularity of the maps $x\mapsto E_u(x),E_s(x)$.
To talk about these, it is convenient to introduce the Grassmanians
$$
\begin{aligned}
\mathscr G_u&:=\{(x,E)\colon x\in M,\ E\subset T_xM\text{ is a $d_u$-dimensional subspace}\},\\
\mathscr G_s&:=\{(x,E)\colon x\in M,\ E\subset T_xM\text{ is a $d_s$-dimensional subspace}\}.
\end{aligned}
$$
These are smooth manifolds fibering over $M$. If $\varphi$ is hyperbolic on $K$,
then we have the maps
\begin{equation}
  \label{e:stun-mapps}
E_u:K\to\mathscr G_u,\quad
E_s:K\to\mathscr G_s.
\end{equation}
We first show that $E_u(x),E_s(x)$ depend continuously on $x$:
\begin{lemm}
  \label{l:contin}
Assume that $\varphi$ is hyperbolic on $K$. Then
the maps~\eqref{e:stun-mapps} are continuous.
\end{lemm}
\begin{proof}
We show continuity of $E_s$, the continuity of $E_u$ is proved similarly.
It suffices to show that if
$$
x_k\in K,\quad
x_k\to x_\infty,\quad
v_k\in E_s(x_k),\quad
|v_k|=1,\quad
v_k\to v_\infty\in T_{x_\infty}M
$$
then $v_\infty\in E_s(x_\infty)$.

By~\eqref{e:hyp-map-s} we have for all $k$ and all $n\geq 0$
$$
|d\varphi^n(x_k)v_k|\leq C\lambda^n.
$$
Passing to the limit $k\to\infty$, we get for all $n\geq 0$
$$
|d\varphi^n(x_\infty)v_\infty|\leq C\lambda^n.
$$
Given~\eqref{e:hyp-map-1} and~\eqref{e:hyp-map-u} this implies
that $v_\infty\in E_s(x_\infty)$, finishing the proof.
\end{proof}
Lemma~\ref{l:contin} and the fact that
$E_u(x),E_s(x)$ are transverse for each $x$
implies that $E_u(x)$, $E_s(x)$ are uniformly transverse, namely there exists a constant $C$ such that
\begin{equation}
  \label{e:transverse}
\max\big(|v_u|,|v_s|\big)\leq C|v_u+v_s|\quad\text{for all}\quad
v_u\in E_u(x),\
v_s\in E_s(x),\
x\in K.
\end{equation}

A quantitative version of the proof of Lemma~\ref{l:contin} shows
that in fact $E_u(x),E_s(x)$ are H\"older continuous in $x$.
(This statement is not used in the rest of these notes.)
\begin{lemm}
  \label{l:holder}
Fix some smooth metrics on $\mathscr G_u,\mathscr G_s$.
Assume that $\varphi$ is hyperbolic on $K$. Then there exists $\gamma>0$
such that the maps~\eqref{e:stun-mapps} have $C^\gamma$ regularity.
\end{lemm}
\begin{proof}
In this proof we denote by $C$ constants which only depend on $\varphi,K$.

We show H\"older continuity of $E_s$, with the case of $E_u$ handled similarly.
Let $\overline TM$ be the fiber-radial compactification of $TM$,
which is a manifold with interior $TM$ and boundary diffeomorphic
to the sphere bundle $SM$, with the boundary defining function $|v|^{-1}$.
Denote by $\Phi$ the action of $d\varphi$ on $TM$:
$$
\Phi(x,v):=(\varphi(x),d\varphi(x)v),\quad
(x,v)\in TM.
$$
Note that $\Phi^n(x,v)=(\varphi^n(x),d\varphi^n(x)v)$.
Moreover, since $\Phi$ is homogeneous with respect to dilations in the fibers of $TM$,
it extends to a smooth map on $\overline TM$.

Fix a Riemannian metric on $\overline TM$ (smooth up to the boundary)
and denote by $d_{\overline TM}$ the corresponding
distance function.
Then there exists 
a constant $\Lambda\geq 1$ depending only on $\varphi,K$ such that
\begin{equation}
  \label{e:holder-1}
d_{\overline TM}\big(\Phi(x,v),\Phi(y,w)\big)\leq \Lambda\cdot d_{\overline TM}\big((x,v),(y,w)\big)\quad\text{for}\quad
x,y\in K.
\end{equation}
We will show that the map $E_s$ is H\"older continuous with exponent
$$
\gamma:=-{2\log\lambda\over\log(\Lambda/\lambda)}>0.
$$
To do this, assume that
$$
x,y\in K,\quad
v\in E_s(x),\quad
w\in T_yM,\quad
|v|=1,\quad
d_{\overline TM}\big((x,v),(y,w)\big)\leq Cd(x,y).
$$
We write
$$
w=w_u+w_s,\quad
w_u\in E_u(y),\quad
w_s\in E_s(y).
$$
Then it suffices to prove that
\begin{equation}
  \label{e:holder-goal}
|w_u|\leq Cd(x,y)^\gamma.
\end{equation}
To show~\eqref{e:holder-goal}, we use the following 
bound true for all $n\geq 0$:
\begin{equation}
  \label{e:holder-2}
\begin{aligned}
d_{\overline TM}\big(\Phi^n(y,w),(\varphi^n(x),0)\big)
&\leq C\Lambda^n d(x,y)+d_{\overline TM}\big(\Phi^n(x,v),(\varphi^n(x),0)\big)\\
&\leq C\Lambda^n d(x,y)+C\lambda^n
\end{aligned}
\end{equation}
where the first inequality uses~\eqref{e:holder-1} iterated $n$ times
and the second one uses \eqref{e:hyp-map-s}.

Assume that $d(x,y)$ is small and choose
$$
n:=\Big\lfloor -{\log d(x,y)\over\log(\Lambda/\lambda)}\Big\rfloor\geq 0.
$$
Then $\Lambda^n d(x,y)\leq \lambda^n$, so~\eqref{e:holder-2} implies
$$
d_{\overline TM}\big(\Phi^n(y,w),(\varphi^n(x),0)\big)
\leq C\lambda^n.
$$
If $d(x,y)$ is small, then $n$ is large, thus $\Phi^n(y,w)$ is close to the zero section.
Therefore
\begin{equation}
  \label{e:holder-3}
|d\varphi^n(y)w|\leq C\lambda^n.
\end{equation}
Recalling the decomposition $w=w_u+w_s$ and using the bounds~\eqref{e:hyp-map-u},
\eqref{e:hyp-map-s},
and~\eqref{e:transverse}, we see that~\eqref{e:holder-3} implies
$$
|w_u|\leq C\lambda^n |d\varphi^n(y)w_u|\leq C\lambda^n|d\varphi^n(y)w|+C\lambda^{2n}
\leq C\lambda^{2n}\leq Cd(x,y)^\gamma.
$$
This gives~\eqref{e:holder-goal}, finishing the proof.
\end{proof}

\subsection{Adapted metrics}
  \label{s:adapted}
  
In preparation for the proof of Theorem~\ref{t:stun-maps}, we
show that there exist Riemannian metrics on $M$ which are adapted
to the map $\varphi$:
\begin{lemm}
  \label{l:adapted-metric}
Assume that $\varphi$ is hyperbolic on $K$. Fix $\tilde\lambda$ such that
$\lambda<\tilde\lambda<1$ where $\lambda$ is given in Definition~\ref{d:hyp-map}.
Then there exist $C^N$ Riemannian metrics $|\bullet|_u,|\bullet|_s$ on~$M$ such that
\begin{align}
\label{e:hyp-map-u-1}
|d\varphi^{-1}(x)v|_u\leq \tilde\lambda |v|_u&\quad\text{for all}\quad
v\in E_u(x),\ x\in K;\\
\label{e:hyp-map-s-1}
|d\varphi(x)v|_s\leq \tilde\lambda |v|_s&\quad\text{for all}\quad
v\in E_s(x),\ x\in K.
\end{align}
\end{lemm}
\Remark Iterating~\eqref{e:hyp-map-u-1}, \eqref{e:hyp-map-s-1}
we get analogs of~\eqref{e:hyp-map-u}, \eqref{e:hyp-map-s} with $C=1$.
\begin{proof}
We first construct the metric $|\bullet|_s$. Let $|\bullet|$ be
a Riemannian metric on~$M$.
Take large fixed $m$ to be chosen later and define the Riemannian metric $|\bullet|_s$ by
$$
|v|_s^2:=\sum_{n=0}^{m-1} \tilde\lambda^{-2n}|d\varphi^n(x)v|^2,\quad
x\in M,\quad
v\in T_xM.
$$
Assume that $x\in K$ and $v\in E_s(x)$. Then
$$
\begin{aligned}
|d\varphi(x)v|_s^2&=\sum_{n=0}^{m-1}\tilde\lambda^{-2n}|d\varphi^{n+1}(x)v|^2
=\sum_{n=1}^m \tilde\lambda^{2-2n}|d\varphi^n(x)v|^2\\
&=\tilde\lambda^2\big(|v|_s^2-|v|^2+\tilde\lambda^{-2m}|d\varphi^m(x)v|^2\big)\\
&\leq \tilde\lambda^2\big(|v|_s^2-|v|^2+C\lambda^{2m}\tilde\lambda^{-2m}|v|^2\big)
\end{aligned}
$$
where in the last inequality we used~\eqref{e:hyp-map-s} and $C$ is a constant
depending only on $\varphi,K$. Since $\tilde\lambda>\lambda$, choosing
$m$ large enough we can guarantee that $C\lambda^{2m}\tilde\lambda^{-2m}\leq 1$,
thus~\eqref{e:hyp-map-s-1} holds.

The inequality~\eqref{e:hyp-map-u-1} is proved similarly, using the metric
$$
|v|_u^2:=\sum_{n=0}^{m-1} \tilde\lambda^{-2n}|d\varphi^{-n}(x)v|^2,\quad
x\in M,\quad
v\in T_xM.\qedhere
$$
\end{proof}

\subsection{Adapted charts}
  \label{s:adapted-coord}
  
To reduce Theorem~\ref{t:stun-maps} to Theorem~\ref{t:stun-adv},
we introduce charts on~$M$ which are adapted to the map $\varphi$.
We assume that $\varphi$ is hyperbolic on $K\subset M$, fix
$\tilde\lambda\in (\lambda,1)$ and let $|\bullet|_u$, $|\bullet|_s$
be the Riemannian metrics on $M$ constructed in Lemma~\ref{l:adapted-metric}.
We write the elements of $\mathbb R^d$ as $(x_1,x_2)$
where $x_1\in \mathbb R^{d_u}$, $x_2\in\mathbb R^{d_s}$. We use
the canonical stable/unstable subspaces of $\mathbb R^d$
\begin{equation}
  \label{e:e-su-0}
E_u(0):=\{(v_1,0)\mid v_1\in \mathbb R^{d_u}\},\quad
E_s(0):=\{(0,v_2)\mid v_2\in\mathbb R^{d_s}\}.
\end{equation}
Recall from~\eqref{e:B-infty-new} the notation
$$
\overline B_\infty(0,r)=\{(x_1,x_2)\in\mathbb R^d\colon \max(|x_1|,|x_2|)\leq r\}.
$$
\begin{defi}
  \label{d:adapted-chart}
Let $x_0\in K$. A diffeomorphism
$$
\varkappa:U_\varkappa\to V_\varkappa,\quad
x_0\in U_\varkappa\subset M,\quad
0\in V_\varkappa\subset \mathbb R^d
$$
is called an \textbf{adapted chart} for $\varphi$ centered at $x_0$, if:
\begin{enumerate}
\item $\varkappa(x_0)=0$;
\item $d\varkappa(x_0)E_u(x_0)=E_u(0)$
and the restriction of $d\varkappa(x_0)$ to $E_u(x_0)$ is an isometry
from the metric $|\bullet|_u$ to the Euclidean metric;
\item $d\varkappa(x_0)E_s(x_0)=E_s(0)$
and the restriction of $d\varkappa(x_0)$ to $E_s(x_0)$ is an isometry
from the metric $|\bullet|_s$ to the Euclidean metric.
\end{enumerate}
\end{defi}
For each $x_0\in K$, there exists an adapted chart for $\varphi$ centered at $x_0$.
Moreover, it follows from uniform transversality~\eqref{e:transverse}
of $E_u$, $E_s$ that we can select for each $x_0\in K$
an adapted chart for $\varphi$ centered at $x_0$
\begin{equation}
  \label{e:good-charts}
\varkappa_{x_0}:U_{x_0}\to V_{x_0},\quad
x_0\in U_{x_0}\subset M,\quad
0\in V_{x_0}\subset\mathbb R^d
\end{equation}
such that the set $\{\varkappa_{x_0}\mid x_0\in K\}$ is bounded in
the class of $C^{N+1}$ charts, more precisely:
\begin{enumerate}
\item there exists $\delta_0>0$ such that $\overline B_\infty(0,\delta_0)\subset V_{x_0}$
for all $x_0\in K$;
\item all order $\leq N+1$ derivatives of $\varkappa_{x_0}$ and $\varkappa_{x_0}^{-1}$ are bounded
uniformly in $x_0$.
\end{enumerate}
Note that we do not require continuous dependence of $\varkappa_{x_0}$ on $x_0$,
in fact in many cases such dependence is impossible because the bundles
$E_u,E_s$ are not topologically trivial.

We now study the action of the map $\varphi$ in adapted charts. For each $x_0\in K$, define
the diffeomorphism $\psi_{x_0}$ of neighborhoods of $0$ in $\mathbb R^d$ by
\begin{equation}
  \label{e:phi-adapted}
\psi_{x_0}:=\varkappa_{\varphi(x_0)}\circ\varphi\circ \varkappa_{x_0}^{-1}.
\end{equation}
Note that the set $\{\psi_{x_0}\mid x_0\in K\}$ is bounded in the class
of $C^{N+1}$ diffeomorphisms.

From the definition of adapted charts and the fact
that $E_u(x),E_s(x)$ are $\varphi$-invariant we see that
\begin{equation}
  \label{e:butter-1}
\psi_{x_0}(0)=0,\quad
d\psi_{x_0}(0)=\begin{pmatrix} A_{1,x_0} & 0 \\
0 & A_{2,x_0}\end{pmatrix}
\end{equation}
where $A_{1,x_0}:\mathbb R^{d_u}\to \mathbb R^{d_u}$,
$A_{2,x_0}:\mathbb R^{d_s}\to\mathbb R^{d_s}$ are linear isomorphisms.
Moreover, from the properties~\eqref{e:hyp-map-u-1},
\eqref{e:hyp-map-s-1} of the adapted metrics $|\bullet|_u$, $|\bullet|_s$
we get
\begin{equation}
  \label{e:butter-3}
\|A_{1,x_0}^{-1}\|\leq \tilde\lambda,\quad
\|A_{2,x_0}\|\leq \tilde\lambda,\quad
\max(\|A_{1,x_0}\|,\|A_{2,x_0}^{-1}\|)\leq C_0\quad\text{for all }x_0\in K
\end{equation}
for some fixed $C_0$, where $\|\bullet\|$ is the operator norm
with respect to the Euclidean norm.

\begin{figure}
\includegraphics{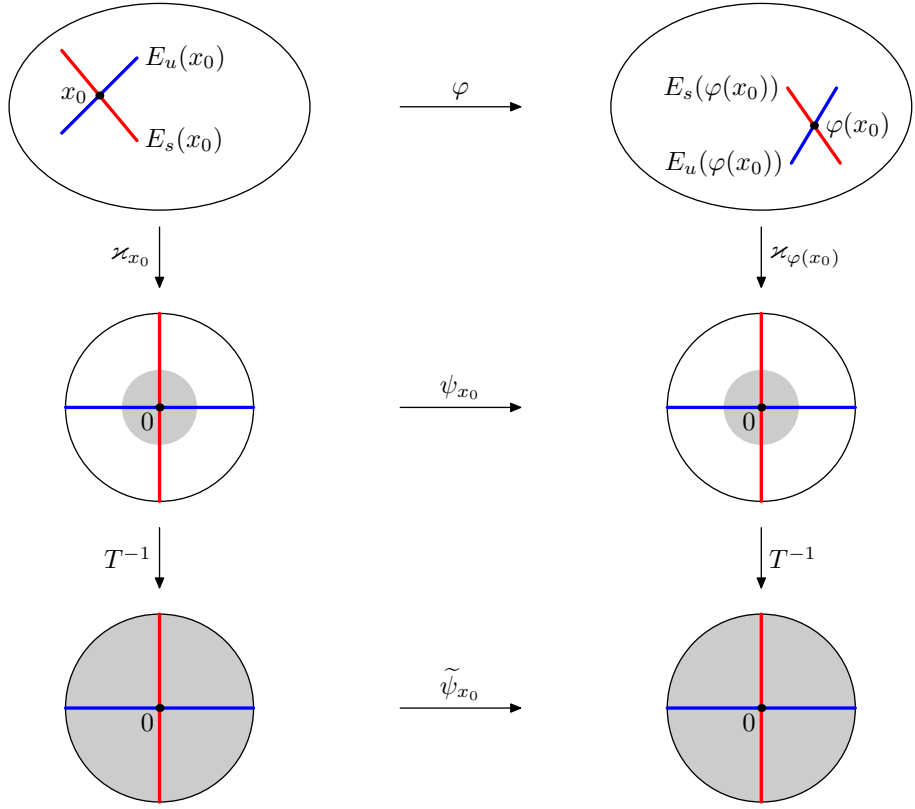}
\caption{An illustration of the commutative diagram~\eqref{e:commdiag}. The blue/red lines
are the unstable/stable subspaces of the tangent spaces.}
\label{f:charts}
\end{figure}

To make the maps $\psi_{x_0}$ close to their linearizations $d\psi_{x_0}(0)$,
we use the rescaling procedure introduced in~\S\ref{s:reduce-1}.
Fix small $\delta_1>0$ to be chosen later (when we apply Theorem~\ref{t:stun-adv})
and consider the rescaling map
$$
T:\mathbb R^d\to\mathbb R^d,\quad
T(x)=\delta_1x.
$$
Define the rescaled charts
\begin{equation}
  \label{e:rescaled-chart}
\widetilde\varkappa_{x_0}:=T^{-1}\circ\varkappa_{x_0},\quad
x_0\in K
\end{equation}
and the corresponding maps
\begin{equation}
  \label{e:phi-adapted-2}
\widetilde\psi_{x_0}:=\widetilde\varkappa_{\varphi(x_0)}\circ\varphi\circ \widetilde\varkappa_{x_0}^{-1}
=T^{-1}\circ\psi_{x_0}\circ T.
\end{equation}
That is, we have the commutative diagram (see also Figure~\ref{f:charts})
\begin{equation}
  \label{e:commdiag}
\begin{tikzcd}[column sep=large, row sep=large]
M \arrow[r, "\varphi"] \arrow[d, "\varkappa_{x_0}"] \arrow[dd, bend right,swap,"\widetilde\varkappa_{x_0}"]
& M \arrow[d, swap, "\varkappa_{\varphi(x_0)}"] \arrow[dd, bend left, "\widetilde\varkappa_{\varphi(x_0)}"]\\
\mathbb R^d \arrow[r, "\psi_{x_0}"] \arrow[d, "T^{-1}"]
& \mathbb R^d \arrow[d, swap, "T^{-1}"] \\
\mathbb R^d \arrow[r, "\widetilde\psi_{x_0}"]
& \mathbb R^d
\end{tikzcd}.
\end{equation}
We choose $\delta_1$ small enough so that $\overline B_{\infty}(0,1)$ is contained
in both the domain and the range of each $\widetilde\psi_{x_0}$. The map
$\widetilde\psi_{x_0}$ still satisfies~\eqref{e:butter-1}:
\begin{equation}
  \label{e:butter-2}
\widetilde\psi_{x_0}(0)=0,\quad
d\widetilde\psi_{x_0}(0)=\begin{pmatrix}
A_{1,x_0}&0\\
0&A_{2,x_0}\end{pmatrix}
\end{equation}
where $A_{j,x_0}$ are the same transformations as in~\eqref{e:butter-1}.

Moreover, as explained in~\S\ref{s:reduce-1}, for any given $\delta>0$,
if we choose $\delta_1$ small enough depending on $\delta,\varphi,K$
(but not on $x_0$) then we have the derivative bounds
\begin{equation}
  \label{e:delta-1-chosen}
\sup |\partial^\alpha\widetilde\psi_{x_0}|\leq \delta\quad\text{for all}\quad
x_0\in K,\quad
2\leq|\alpha|\leq N+1.
\end{equation}
In fact, one can take $\delta_1:=\delta/C$ where $C$ is some constant
depending on $\varphi,K$.

\subsection{Proof of Theorem~\ref{t:stun-maps}}
  \label{s:maps-proof}

We now give the proof of the Stable/Unstable Manifold Theorem for hyperbolic maps.
We fix $\tilde\lambda$ such that $\lambda<\tilde\lambda<1$. Let $\delta>0$
be the constant in Theorem~\ref{t:stun-adv}
where we use $\tilde\lambda,\tilde\lambda^{-1}$ in place of $\lambda,\mu$
and $C_0$ is the constant from~\eqref{e:butter-3}.
Choose the rescaling parameter $\delta_1>0$ in~\S\ref{s:adapted-coord} small enough
so that~\eqref{e:delta-1-chosen} holds.

We use the rescaled adapted charts $\widetilde\varkappa_{x_0}$, $x_0\in K$, defined
in~\eqref{e:rescaled-chart}, and the maps $\widetilde\psi_{x_0}$
giving the action of $\varphi$ in these charts, defined in~\eqref{e:phi-adapted-2}.
For every $x\in K$ and $m\in\mathbb Z$ denote
\begin{equation}
  \label{e:maps-defi}
\psi_{x,m}:=\widetilde\psi_{\varphi^m(x)}=\widetilde\varkappa_{\varphi^{m+1}(x)}\circ\varphi
\circ\widetilde\varkappa^{-1}_{\varphi^m(x)}.
\end{equation}
The dynamics of the iterates of $\varphi$ near the trajectory
$(\varphi^m(x))$
is conjugated by the charts $\widetilde\varkappa_{\varphi^m(x)}$
to the dynamics of the compositions of the maps $\psi_{x,m}$.
We will prove Theorem~\ref{t:stun-maps} by applying Theorem~\ref{t:stun-adv}
to the maps $\psi_{x,m}$ and pulling back the resulting stable/unstable manifolds
by $\widetilde\varkappa_{\varphi^m(x)}$ to get the stable/unstable manifolds
for $\varphi$.

As shown in~\S\ref{s:adapted-coord}, for each $x\in K$ the sequence of maps
$(\psi_{x,m})_{m\in\mathbb Z}$ satisfies
the assumptions in~\S\ref{s:reduce-5}, where we use $\tilde\lambda,\tilde\lambda^{-1}$
in place of $\lambda,\mu$. We apply Theorem~\ref{t:stun-adv} to get the stable/unstable
manifolds for this sequence, which we denote
$$
W^u_{x,m},W^s_{x,m}\subset \overline B_\infty(0,1)\subset\mathbb R^d.
$$
We define the unstable/stable manifolds for $\varphi$ at $x$
as follows:
\begin{equation}
  \label{e:stun-maps-defined}
W_u(x):=\widetilde\varkappa_{x}^{-1}(W^u_{x,0}),\quad
W_s(x):=\widetilde\varkappa_{x}^{-1}(W^s_{x,0}).
\end{equation}
The unstable/stable manifolds at the iterates $\varphi^n(x)$ are given by
\begin{equation}
  \label{e:works-well}
W_u(\varphi^n(x))=\widetilde\varkappa_{\varphi^n(x)}^{-1}(W^u_{x,n}),\quad
W_s(\varphi^n(x))=\widetilde\varkappa_{\varphi^n(x)}^{-1}(W^s_{x,n}),\quad
n\in\mathbb Z.
\end{equation}
The statement~\eqref{e:works-well} is not a tautology since
the left-hand sides were obtained by applying Theorem~\ref{t:stun-adv}
to the sequence of maps $(\psi_{\varphi^n(x),m})_{m\in\mathbb Z}$ while
the right-hand sides were obtained using the maps $(\psi_{x,m})_{m\in\mathbb Z}$.
To prove~\eqref{e:works-well} we note that \eqref{e:maps-defi} implies
$$
\psi_{\varphi^n(x),m}=\psi_{x,m+n}.
$$
Therefore, the sequence $(\psi_{\varphi^n(x),m})_{m\in\mathbb Z}$
is just a shift of the sequence $(\psi_{x,m})_{m\in\mathbb Z}$.
From the construction of the manifolds $W^u_{x,m}$, $W^s_{x,m}$ in~\S\ref{s:reduce-4}
we see that
$$
W^u_{\varphi^n(x),m}=W^u_{x,m+n},\quad
W^s_{\varphi^n(x),m}=W^s_{x,m+n}.
$$
Putting $m=0$ and recalling~\eqref{e:stun-maps-defined}, we get~\eqref{e:works-well}.

We now show that the manifolds $W_u(x)$, $W_s(x)$ defined in~\eqref{e:stun-maps-defined}
satisfy the statements~(1)--(8) in Theorem~\ref{t:stun-maps}.
This is straightforward since the above construction effectively
reduced Theorem~\ref{t:stun-maps} to Theorem~\ref{t:stun-adv}.
\begin{itemize}
\item[(1):] This follows from the definitions~\eqref{e:wuwsy} of $W^u_{x,m},W^s_{x,m}$. The uniform boundedness of the embeddings in $C^N$
follows from the fact that the functions $F^u_{x,m}$, $F^s_{x,m}$
used to define $W^u_{x,m}$, $W^s_{x,m}$ are bounded by~1 in $C^N$ norm,
see~\S\ref{s:reduce-4} and~\eqref{e:X-N-def}.
\smallskip
\item[(2):] We have $W^u_{x,m}\cap W^s_{x,m}=\{0\}$ by statement~(2) in Theorem~\ref{t:stun-adv}.
By~\eqref{e:fufsy} we have also $T_0 W^u_{x,m}=E_u(0)$,
$T_0 W^s_{x,m}=E_s(0)$ where $E_u(0),E_s(0)\subset\mathbb R^d$ are defined in~\eqref{e:e-su-0}.
It remains to use~\eqref{e:stun-maps-defined}
and the fact that $d\widetilde\varkappa_x(x)$ maps
$E_u(x),E_s(x)$ to $E_u(0),E_s(0)$ by Definition~\ref{d:adapted-chart}
and~\eqref{e:rescaled-chart}.
\smallskip
\item[(3):] Every $w=(w_1,w_2)\in\partial W^u_{x,m}$
satisfies $|w_1|=1$, thus $|w|\geq 1$. The latter is also true
for all $w\in \partial W^s_{x,m}$. It then suffices to choose
$\varepsilon_0$ small enough so that
\begin{equation}
  \label{e:eps-0-chosen}
x\in K,\
d(x,y)\leq \varepsilon_0
\quad\Longrightarrow\quad
|\widetilde\varkappa_x(y)|<1
\end{equation}
which is possible since $\widetilde\varkappa_x(x)=0$ and $\widetilde\varkappa_x$
are bounded in $C^N$ uniformly in $x$.
\smallskip
\item[(4):] By statement~(1) in Theorem~\ref{t:stun-adv} we have
\begin{equation}
  \label{e:yoga}
\psi_{x,-1}^{-1}(W^u_{x,0})\subset W^u_{x,-1},\quad
\psi_{x,0}(W^s_{x,0})\subset W^s_{x,1}.
\end{equation}
From~\eqref{e:maps-defi} we have
$$
\psi_{x,-1}^{-1}=\widetilde\varkappa_{\varphi^{-1}(x)}\circ\varphi^{-1}
\circ\widetilde\varkappa^{-1}_x,\quad
\psi_{x,0}=\widetilde\varkappa_{\varphi(x)}\circ\varphi\circ\widetilde\varkappa^{-1}_x.
$$
Applying $\widetilde\varkappa_{\varphi^{-1}(x)}^{-1}$ to the first statement in~\eqref{e:yoga}
and $\widetilde\varkappa_{\varphi(x)}^{-1}$ to the second one
and using~\eqref{e:works-well} we get
$\varphi^{-1}(W_u(x))\subset W_u(\varphi^{-1}(x))$ and
$\varphi(W_s(x))\subset W_s(\varphi(x))$ as needed.
\smallskip
\item[(5)--(6):] Define the closed neighborhoods of $x$
$$
\overline B_x:=\widetilde\varkappa_x^{-1}(\overline B_\infty(0,1)),\quad
x\in K.
$$
Note that $W_u(x)\cup W_s(x)\subset\overline B_x$.
By~\eqref{e:maps-defi}, if $n\geq 1$ and
$\varphi^{-\ell}(y)\in \overline B_{\varphi^{-\ell}(x)}$ for all $\ell=0,\dots,n-1$,
then
\begin{equation}
  \label{e:oid-1}
\widetilde\varkappa_{\varphi^{-n}(x)}\big(\varphi^{-n}(y)\big)
=\psi_{x,-n}^{-1}\cdots\psi_{x,-1}^{-1}(\widetilde\varkappa_x(y)).
\end{equation}
Similarly if $n\geq 1$ and $\varphi^\ell(y)\in \overline B_{\varphi^\ell(x)}$
for all $\ell=0,\dots,n-1$, then
\begin{equation}
  \label{e:oid-2}
\widetilde\varkappa_{\varphi^n(x)}\big(\varphi^n(y)\big)
=\psi_{x,n-1}\cdots\psi_{x,0}(\widetilde\varkappa_x(y)).
\end{equation}
If $y\in W_u(x)$, then for all $\ell\geq 0$ we have
$\varphi^{-\ell}(y)\in W_u(\varphi^{-\ell}(x))\subset \overline B_{\varphi^{-\ell}(x)}$.
Moreover, $\widetilde\varkappa_x(y)\in W^u_{x,0}$ by~\eqref{e:stun-maps-defined}.
Applying the statement~(3) in Theorem~\ref{t:stun-adv} with $m:=0$, we get
$\psi^{-1}_{x,-n}\cdots\psi_{x,-1}^{-1}(\widetilde\varkappa_x(y))\to 0$ as $n\to\infty$,
thus $d(\varphi^{-n}(y),\varphi^{-n}(x))\to 0$ by~\eqref{e:oid-1}.
The case $y\in W_s(x)$ is handled similarly.
\smallskip
\item[(7)--(8):] We show (7), with~(8) proved similarly. Assume that
$d(\varphi^{-n}(y),\varphi^{-n}(x))\leq\varepsilon_0$ for all $n\geq 0$.
By~\eqref{e:eps-0-chosen} this implies that
$\varphi^{-n}(y)\in \overline B_{\varphi^{-n}(x)}$ for all $n\geq 0$.
Then by~\eqref{e:oid-1} we have
$$
\psi_{x,-n}^{-1}\cdots\psi_{x,-1}^{-1}(\widetilde\varkappa_x(y))\in\overline B_\infty(0,1)\quad\text{for all }n\geq 0.
$$
By statement~(5) in Theorem~\ref{t:stun-adv} with $m:=0$ we have
$\widetilde\varkappa_x(y)\in W^u_{x,0}$ and thus
$y\in W_u(x)$ by~\eqref{e:stun-maps-defined}.
\end{itemize}
The quantitative statements~\eqref{e:cream-1}--\eqref{e:cream-3}
follow from~\eqref{e:stuna-2}--\eqref{e:stuna-4} similarly to the proofs of statements~(5)--(8) above.

Finally, statement~(9) in Theorem~\ref{t:stun-maps} essentially follows
from the continuous dependence of $E_u(x),E_s(x)$ on $x$ (Lemma~\ref{l:contin})
and the fact that $E_u(x)$, $E_s(x)$ are transversal to each other. We give
a more detailed (straightforward but slightly tedious) explanation below.

Recall from~\eqref{e:stun-maps-defined} that $W_u(x)=\widetilde\varkappa_x^{-1}(W_{x,0}^u)$
and $W_s(x)=\widetilde\varkappa_x^{-1}(W_{x,0}^s)$
where $\widetilde\varkappa_x$ is the rescaled adapted chart defined in~\eqref{e:rescaled-chart}
and $W_{x,0}^u,W_{x,0}^s\subset\overline B_\infty(0,1)$ are the unstable/stable graphs constructed
in Theorem~\ref{t:stun-adv}. Writing elements of $\mathbb R^d$ as
$(w_1,w_2)$ where $w_1\in\mathbb R^{d_u}$, $w_2\in\mathbb R^{d_s}$ we have
\begin{align}
  \label{e:yappy-1}
\widetilde\varkappa_x(W_u(x))=W_{x,0}^u&=\{(w_1,w_2)\colon |w_1|\leq 1,\
w_2=F_{x,u}(w_1)\},\\
  \label{e:yappy-2}
\widetilde\varkappa_x(W_s(x))=W_{x,0}^s&=\{(w_1,w_2)\colon |w_2|\leq 1,\
w_1=F_{x,s}(w_2)\}
\end{align}
for some functions $F_{x,u}:\overline B_u(0,1)\to\overline B_s(0,1)$,
$F_{x,s}:\overline B_s(0,1)\to\overline B_u(0,1)$,
where $\overline B_u(0,1),\overline B_s(0,1)$ are defined in~\eqref{e:stab-the-ball}, such that
(recalling~\eqref{e:fufsy}, \eqref{e:F-u-derb}, and~\eqref{e:delta-1-chosen})
\begin{equation}
  \label{e:fluffy}
F_{x,u}(0)=0,\quad
F_{x,s}(0)=0,\quad
\max\big(\|F_{x,u}\|_{C^1},\|F_{x,s}\|_{C^1}\big)\leq C\delta_1;
\end{equation}
here $\delta_1>0$ is the rescaling parameter used in the definition~\eqref{e:rescaled-chart}
of the chart $\widetilde\varkappa_x$.

Now, we assume that $x,y\in K$ and $d(x,y)\leq \varepsilon_0$
where $\varepsilon_0>0$ is small, in particular $\varepsilon_0\ll \delta_1$.
Then $W_u(y)$ is contained in the domain of the chart $\widetilde\varkappa_x$.
The points in $W_s(x)\cap W_u(y)$ have the form
$\widetilde\varkappa_x^{-1}(F_{x,s}(w_2),w_2)=\widetilde\varkappa_y^{-1}(w_1,F_{y,u}(w_1))$ where
$w=(w_1,w_2)\in\overline B_\infty(0,1)$ solves the equation
\begin{equation}
  \label{e:equator}
G_{x,y}(w)=0,\quad
G_{x,y}(w_1,w_2):=
\widetilde\varkappa_{x,y}(w_1,F_{y,u}(w_1))-(F_{x,s}(w_2),w_2)
\end{equation}
where we put
$$
\widetilde\varkappa_{x,y}:=\widetilde\varkappa_x\circ\widetilde\varkappa_y^{-1}
=T^{-1}\circ\varkappa_{x,y}\circ T,\quad
\varkappa_{x,y}:=\varkappa_x\circ\varkappa_y^{-1},\quad
T:w\mapsto\delta_1 w,
$$
and $\varkappa_x,\varkappa_y$ are the unrescaled charts defined in~\eqref{e:good-charts}.
Since $\varkappa_x,\varkappa_y$ are adapted charts (see Definition~\ref{d:adapted-chart}),
$d(x,y)\leq\varepsilon_0$, and $E_u(x),E_s(x)$ depend continuously on $x$,
for small enough $\varepsilon_0$ we have
$$
d\varkappa_{x,y}(0)=A_{x,y}+\mathcal O(\delta_1),\quad
A_{x,y}:=\begin{pmatrix}A_{1,x,y} & 0 \\
0 & A_{2,x,y} \end{pmatrix}
$$
where $A_{1,x,y}:\mathbb R^{d_u}\to\mathbb R^{d_u}$,
$A_{2,x,y}:\mathbb R^{d_s}\to\mathbb R^{d_s}$
are isometries. The rescaled change of coordinates map $\widetilde\varkappa_{x,y}$
then satisfies (assuming $\varepsilon_0\leq\delta_1^2$)
$$
|\widetilde\varkappa_{x,y}(0)|\leq C\delta_1,\quad
\sup_{w\in \overline B_\infty(0,1)} \|d\widetilde\varkappa_{x,y}(w)-A_{x,y}\|\leq C\delta_1. 
$$
Combining this with~\eqref{e:fluffy} we get
$$
|G_{x,y}(0)|\leq C\delta_1,\quad
\sup_{w\in\overline B_\infty(0,1)}\bigg\|dG_{x,y}(w)-\begin{pmatrix} A_{1,x,y}&0 \\ 0 & -I\end{pmatrix}\bigg\|\leq C\delta_1.
$$
If $\delta_1$ is small enough, then by the Contraction Mapping Principle
the equation~\eqref{e:equator} has a unique solution in $\overline B_\infty(0,1)$.
Therefore, $W_s(x)\cap W_u(y)$ consists of a single point as required.

\subsection{Hyperbolic flows}
  \label{s:hyp-flows}
  
We finally consider the setting of hyperbolic flows.
Let $M$ be a $d$-dimensional manifold without boundary
and $d=d_u+d_s+1$ where $d_u,d_s\geq 0$. Let
$$
\varphi^t=\exp(tX):M\to M,\quad
t\in\mathbb R
$$
be the flow generated by a $C^{N+1}$ vector field $X$ on $M$.
For simplicity we assume that $\varphi^t$ is globally well-defined for all $t$,
though in practice it is enough to require this in the neighborhood of the set $K$ below.
We assume that $K\subset M$ is a $\varphi^t$-invariant hyperbolic set in the following sense:
\begin{defi}
  \label{d:hyp-flows}
Let $K\subset M$ be a compact set such that $\varphi^t(K)=K$ for all $t\in\mathbb R$.
We say that the flow $\varphi^t$ is \textbf{hyperbolic} on $K$ if
the generating vector field $X$ does not vanish on $K$ and there exists a splitting
\begin{equation}
  \label{e:hypflow-split}
T_xM=E_0(x)\oplus E_u(x)\oplus E_s(x),\quad
x\in K,\quad
E_0(x):=\mathbb R X(x),
\end{equation}
where $E_u(x),E_s(x)\subset T_x M$ are subspaces of dimensions $d_u,d_s$ and:
\begin{itemize}
\item $E_u,E_s$ are invariant under the flow, namely
for all $x\in K$ and $t\in\mathbb R$
\begin{equation}
  \label{e:hypflow-inv}
d\varphi^t(x)E_u(x)=E_u(\varphi^t(x)),\quad
d\varphi^t(x)E_s(x)=E_s(\varphi^t(x));
\end{equation}
\item $d\varphi^t$ is expanding on $E_u$ and contracting on $E_s$, namely there exist constants $C>0$, $\nu>0$
such that for some Riemannian metric $|\bullet|$ on $M$ and
all $x\in K$
\begin{equation}
  \label{e:hypflow-exp}
|d\varphi^t(x)v|\leq Ce^{-\nu|t|}\cdot|v|,\quad
\begin{cases}
v\in E_u(x),& t\leq 0;\\
v\in E_s(x),& t\geq 0.
\end{cases}
\end{equation}
\end{itemize}
\end{defi}
\Remarks
1. The time-$t$ map $\varphi^t$ of a hyperbolic flow is \textbf{not} a hyperbolic map
in the sense of Definition~\ref{d:hyp-map} because of the flow direction $E_0$.
  
\noindent 2. The property~\eqref{e:hypflow-exp} does not depend on the choice of the
metric on $M$, though the constant $C$ (but not $\nu$) will depend on the metric.

\noindent 3. The basic example of a hyperbolic set is a closed trajectory
$$
K=\{\varphi^t(x_0)\colon 0\leq t\leq T\}\quad\text{for some}\quad
x_0\in M,\
T>0\quad\text{such that}\quad
\varphi^T(x_0)=x_0
$$
which is hyperbolic, namely $d\varphi^T(x_0)$ has a simple eigenvalue~$1$
and no other eigenvalues on the unit circle. The opposite situation
is when $K=M$; in this case $\varphi^t$ is called an \emph{Anosov flow}.
An important class of Anosov flows are geodesic flows
on negatively curved manifolds, discussed in~\S\ref{s:surf-neg} below.

As in~\S\ref{s:hyp-maps}, fix a distance function $d(\bullet,\bullet)$
on $M$ and define the balls $\overline B_d(x,r)$ by~\eqref{e:B-d-def}.
We now state the Stable/Unstable Manifold Theorem for hyperbolic flows,
which is similar to the case of maps (Theorem~\ref{t:stun-maps}):
\begin{theo}
  \label{t:stun-flows}
Assume that the flow $\varphi^t$ is hyperbolic on $K\subset M$. Then for each $x\in K$
there exist \textbf{local unstable/stable manifolds}
$$
W_u(x),W_s(x)\subset M
$$
which have the following properties for some $\varepsilon_0>0$ depending only on $\varphi^t,K$:
\begin{enumerate}
\item $W_u(x),W_s(x)$ are $C^N$ embedded disks of dimensions $d_u,d_s$,
and the $C^N$ norms of the embeddings are bounded uniformly in~$x$;
\item $W_u(x)\cap W_s(x)=\{x\}$ and $T_x W_u(x)=E_u(x)$, $T_xW_s(x)=E_s(x)$;
\item the boundaries of $W_u(x),W_s(x)$ do not intersect $\overline B_d(x,\varepsilon_0)$;
\item $\varphi^{-1}(W_u(x))\subset W_u(\varphi^{-1}(x))$ and
$\varphi^1(W_s(x))\subset W_s(\varphi^1(x))$;
\item for each $y\in W_u(x)$, we have $d(\varphi^t(y),\varphi^t(x))\to 0$
as $t\to -\infty$;
\item for each $y\in W_s(x)$, we have $d(\varphi^t(y),\varphi^t(x))\to 0$
as $t\to \infty$;
\item if $y\in M$, $d(\varphi^t(y),\varphi^t(x))\leq\varepsilon_0$ for all $t\leq 0$,
and $d(\varphi^{t}(y),\varphi^{t}(x))\to 0$ as $t\to-\infty$,
then $y\in W_u(x)$;
\item if $y\in M$, $d(\varphi^t(y),\varphi^t(x))\leq\varepsilon_0$ for all $t\geq 0$,
and $d(\varphi^{t}(y),\varphi^{t}(x))\to 0$ as $t\to\infty$,
then $y\in W_s(x)$.
\end{enumerate}
\end{theo}
\Remarks 1. The statement~(4) is somewhat artificial since it involves the time-one map
$\varphi^1$ and its inverse $\varphi^{-1}$. By rescaling the flow we can construct local stable/unstable manifolds
such that the statement~(4) holds with $\varphi^{t_0},\varphi^{-t_0}$ instead,
where $t_0>0$ is any fixed number. A more natural statement
would be that $\varphi^{-t}(W_u(x))\subset W_u(\varphi^{-t}(x))$
and $\varphi^t(W_s(x))\subset W_s(\varphi^t(x))$ for all $t\geq 0$,
but this would be more difficult to arrange since our method of proof is tailored to discrete time evolution.
However, below in~\S\ref{s:flow-invariance} we explain that the \emph{global} stable/unstable manifolds
are invariant under $\varphi^t$ for all $t$.

\noindent 2. Compared to Theorem~\ref{t:stun-maps}, properties~(7) and~(8) impose
the additional condition that $d(\varphi^t(y),\varphi^t(x))\to 0$. This
is due to the presence of the flow direction: for instance, if $y=\varphi^s(x)$
where $s\neq 0$ is small, then $d(\varphi^t(y),\varphi^t(x))\leq\varepsilon_0$ for all $t$
but $y\notin W_u(x)\cup W_s(x)$. Without this additional condition we can only
assert that $y$ lies in the \emph{weak} stable/unstable manifold of $x$, see~\S\ref{s:flows-weak}.

\noindent 3. The analog of statement~(9) of Theorem~\ref{t:stun-maps}
is given in~\eqref{e:weak-transverse}.
The analogs
of the quantitative statements~\eqref{e:cream-1}--\eqref{e:cream-3}
are discussed in~\S\ref{s:flows-quant}.

We now sketch the proof of Theorem~\ref{t:stun-flows}. We follow
the proof of Theorem~\ref{t:stun-maps} in~\S\S\ref{s:continuity}--\ref{s:maps-proof},
indicating the changes needed along the way.

First of all,
the proof of Lemma~\ref{l:contin} applies without change, so
the spaces $E_u(x),E_s(x)$ depend continuously on $x$.
(The H\"older continuity Lemma~\ref{l:holder} holds as well.)
This implies the following version of uniform transversality~\eqref{e:transverse}:
there exists a constant $C$ such that for all $x\in K$
\begin{equation}
  \label{e:transverse-flows}
\max\big(|v_0|,|v_u|,|v_s|\big)\leq C|v_0+v_u+v_s|\quad\text{if}\quad
v_0\in E_0(x),\
v_u\in E_u(x),\
v_s\in E_s(x).
\end{equation}
Next, existence of adapted metrics is given by the following analog of Lemma~\ref{l:adapted-metric}:
\begin{lemm}
  \label{l:adapted-metric-flow}
Assume that $\varphi^t$ is hyperbolic on~$K$. Fix $\tilde\nu$ such that $0<\tilde\nu<\nu$
where $\nu$ is given in Definition~\ref{d:hyp-flows}. Then there exist $C^N$ Riemannian metrics
$|\bullet|_u$, $|\bullet|_s$ on $M$ such that for all $x\in K$
\begin{equation}
  \label{e:adapted-metric-flow}
\begin{aligned}
|d\varphi^t(x)v|_u&\leq e^{-\tilde\nu |t|}\cdot|v|_u\quad\text{for all}\quad v\in E_u(x),\
t\leq 0;\\
|d\varphi^t(x)v|_s&\leq e^{-\tilde\nu |t|}\cdot|v|_s\quad\text{for all}\quad v\in E_s(x),\
t\geq 0.
\end{aligned}
\end{equation}
\end{lemm}
\begin{proof}
This follows by a similar argument to the proof of Lemma~\ref{l:adapted-metric},
fixing a metric $|\bullet|$ on $M$ and defining the adapted metrics as follows:
for $v\in T_xM$,
$$
|v|_u^2:=\int_0^T e^{2\tilde\nu s}|d\varphi^{-s}(x)v|^2\,ds,\quad
|v|_s^2:=\int_0^T e^{2\tilde\nu s}|d\varphi^s(x)v|^2\,ds
$$
where $T>0$ is a sufficiently large constant.
\end{proof}
We next define adapted charts, similarly to~\S\ref{s:adapted-coord}.
Fix $\tilde\nu\in (\nu,1)$ and let $|\bullet|_u$, $|\bullet|_s$
be the adapted metrics constructed in Lemma~\ref{l:adapted-metric-flow}.
We write elements of $\mathbb R^d$
as $(x_0,x_1,x_2)$ where $x_0\in\mathbb R$, $x_1\in\mathbb R^{d_u}$,
and $x_2\in \mathbb R^{d_s}$. Consider the subspaces of $\mathbb R^d$
$$
E_u(0):=\{(0,v_1,0)\mid v_1\in\mathbb R^{d_u}\},\quad
E_s(0):=\{(0,0,v_2)\mid v_2\in\mathbb R^{d_s}\}.
$$
\begin{defi}
  \label{d:adapted-chart-flows}
Let $x\in K$. A $C^{N+1}$ diffeomorphism
$$
\varkappa:U_\varkappa\to V_\varkappa,\quad
x\in U_\varkappa\subset M,\quad
0\in V_\varkappa\subset\mathbb R^d
$$
is called an \textbf{adapted chart} for $\varphi^t$ centered at $x$, if:
\begin{enumerate}
\item $\varkappa(x)=0$;
\item for each $y\in U_\varkappa$, $d\varkappa(y)$ sends the generator of the flow
$X(y)$ to $\partial_{x_0}$;
\item $d\varkappa(x)E_u(x)=E_u(0)$ and the restriction of $d\varkappa(x)$ to $E_u(x)$
is an isometry from the metric $|\bullet|_u$ to the Euclidean metric;
\item $d\varkappa(x)E_s(x)=E_s(0)$ and the restriction of $d\varkappa(x)$ to $E_s(x)$
is an isometry from the metric $|\bullet|_s$ to the Euclidean metric.
\end{enumerate}
\end{defi}
Similarly to~\S\ref{s:adapted-coord}, it follows from the uniform transversality property~\eqref{e:transverse-flows}
that we can select for each $x\in K$ an adapted chart for $\varphi^t$
centered at $x$
$$
\varkappa_x:U_x\to V_x,\quad
x\in K
$$
such that the set $\{\varkappa_x\mid x\in K\}$ is bounded in the class of $C^{N+1}$ charts.

Similarly to~\eqref{e:rescaled-chart} we next define rescaled charts
\begin{equation}
  \label{e:risky-charts}
\widetilde\varkappa_x:=T^{-1}\circ \varkappa_x,\quad
x\in K;\quad
T(w)=\delta_1 w
\end{equation}
where $\delta_1>0$ is chosen small depending only on $\varphi^t,K$
when we apply Theorem~\ref{t:stun-adv}.
The action of the time-one map $\varphi^1$ in the charts
$\widetilde\varkappa_x$ is given by the maps
\begin{equation}
  \label{e:flow-acting}
\widetilde\psi_x:=\widetilde\varkappa_{\varphi^1(x)}\circ\varphi^1\circ\widetilde\varkappa_x^{-1}.
\end{equation}
Arguing as in~\S\ref{s:adapted-coord}, we see that
$\widetilde\psi_x$ has the following properties:
\begin{itemize}
\item the domain and the range of $\widetilde\psi_x$ contain the closed
ball in $\mathbb R^d$;
\item we have
\begin{equation}
  \label{e:flow-differ}
\widetilde\psi_x(0)=0,\quad
d\widetilde\psi_x(0)=\begin{pmatrix}
1 & 0 & 0\\
0 & A_{1,x} & 0\\
0 & 0 & A_{2,x}
\end{pmatrix}
\end{equation}
where the linear maps $A_{1,x}:\mathbb R^{d_u}\to\mathbb R^{d_u}$,
$A_{2,x}:\mathbb R^{d_s}\to\mathbb R^{d_s}$ satisfy
for all $x\in K$
$$
\|A_{1,x}^{-1}\|\leq e^{-\tilde\nu},\quad
\|A_{2,x}\|\leq e^{-\tilde\nu},\quad
\max(\|A_{1,x}\|,\|A_{2,x}\|^{-1})\leq C_0
$$
where $C_0$ is some constant depending only on $\varphi^t,K$;
\item for any given $\delta>0$, if we choose $\delta_1$ small enough
depending on $\delta$, then
\begin{equation}
  \label{e:flow-derbounds}
\sup|\partial^\alpha\widetilde\psi_x|\leq\delta\quad\text{for all}\quad
x\in K,\quad
2\leq |\alpha|\leq N+1.
\end{equation}
\end{itemize}
The matrix $d\widetilde\psi_x(0)$ in~\eqref{e:flow-differ} has eigenvalue 1
coming from the flow direction, which makes some parts of the proof problematic
because Theorem~\ref{t:stun-adv} only partly applies when either
the expansion or the contraction rate is equal to~1~-- see~\S\ref{s:reduce-6}.
To deal with this problem we introduce the \emph{reduced} maps
$\omega_x$, which act on subsets of $\mathbb R^{d-1}$ and correspond to the
action of the flow on Poincar\'e sections. Define the projection map
$$
\pi_{us}:(w_0,w_1,w_2)\mapsto (w_1,w_2)
$$
where $w_0\in\mathbb R$, $w_1\in \mathbb R^{d_u}$, $w_2\in\mathbb R^{d_s}$.
From the definition of adapted charts and~\eqref{e:flow-acting} we see that
\begin{equation}
  \label{e:flow-flowing}
d\widetilde\psi_x(w)\partial_{x_0}=\partial_{x_0}\quad\text{for all }w.
\end{equation}
Therefore (shrinking the domain of $\widetilde\psi_x$ if necessary) there exists
a diffeomorphism $\omega_x$ of open neighborhoods of $\{(x_1,x_2)\in\mathbb R^{d-1}\colon \max(|x_1|,|x_2|)\leq 1\}$
such that
\begin{equation}
  \label{e:omega-x-def}
\pi_{us}(\widetilde\psi_x(w))=\omega_x(\pi_{us}(w))\quad\text{for all }w.
\end{equation}
The maps $\omega_x$ satisfy
$$
\omega_x(0)=0,\quad
d\omega_x(0)=\begin{pmatrix} A_{1,x} & 0 \\ 0 & A_{2,x}\end{pmatrix}
$$
where $A_{1,x},A_{2,x}$ are the same matrices as in~\eqref{e:flow-differ}.
They also satisfy the derivative bounds~\eqref{e:flow-derbounds}.

We can finally give the 
\begin{proof}[Proof of Theorem~\ref{t:stun-flows}]
We argue similarly to~\S\ref{s:maps-proof}. For each
$x\in K$ and $m\in\mathbb Z$ define the maps
$$
\psi_{x,m}:=\widetilde\psi_{\varphi^m(x)}=\widetilde\varkappa_{\varphi^{m+1}(x)}\circ\varphi^1
\circ\widetilde\varkappa_{\varphi^m(x)}^{-1}.
$$
The sequence of maps $(\psi_{x,m})_{m\in\mathbb Z}$ satisfies the assumptions in~\S\ref{s:reduce-5}
where we absorb the $\partial_{x_0}$ direction into the stable space and put
$$
\lambda:=1,\quad
\mu:=e^{\tilde\nu}>1.
$$
As explained in~\S\ref{s:reduce-6}, Theorem~\ref{t:stun-adv}
still partially applies to the maps $(\psi_{x,m})_{m\in\mathbb Z}$
even though $\lambda=1$, yielding the unstable manifolds
$W^u_{x,m}\subset\mathbb R^d$. We then define the unstable manifold for $\varphi^t$
at $x$ by
\begin{equation}
  \label{e:flow-u-def}
W_u(x):=\widetilde\varkappa_x^{-1}(W^u_{x,0}).
\end{equation}
If we instead absorb the $\partial_{x_0}$ direction into the unstable space,
then the sequence $(\psi_{x,m})_{m\in\mathbb Z}$ satisfies the assumptions in~\S\ref{s:reduce-5} where
$$
\lambda:=e^{-\tilde\nu}<1,\quad
\mu:=1.
$$
As explained in~\S\ref{s:reduce-6},
Theorem~\ref{t:stun-adv} gives the stable manifolds $W^s_{x,m}\subset\mathbb R^d$,
and we define the stable manifold for $\varphi^t$ at $x$ by
\begin{equation}
  \label{e:flow-s-def}
W_s(x):=\widetilde\varkappa_x^{-1}(W^s_{x,0}).
\end{equation}
Statements~(1)--(4) of Theorem~\ref{t:stun-flows} are then proved
in the same way as for Theorem~\ref{t:stun-maps}. The proof of Theorem~\ref{t:stun-maps}
also gives the convergence statements~(5)--(6) for integer~$t$, which imply
these statements for all $t$.

To show statements~(7)--(8) we use the reduced maps $\omega_x$. 
The sequence
$$
\omega_{x,m}:=\omega_{\varphi^m(x)}
$$
satisfies the assumptions
in~\S\ref{s:reduce-5} with
$$
\lambda:=e^{-\tilde\nu}<1<\mu:=e^{\tilde\nu}.
$$
Therefore Theorem~\ref{t:stun-adv} applies to give unstable/stable manifolds
for the sequence $(\omega_{x,m})_{m\in\mathbb Z}$. Recalling the construction
of these manifolds in~\S\ref{s:reduce-4} it is straightforward to see that
the unstable/stable manifolds for $(\omega_{x,m})_{m\in\mathbb Z}$ are equal
to $\pi_{us}(W^u_{x,m})$, $\pi_{us}(W^s_{x,m})$ where
$W^u_{x,m}$, $W^s_{x,m}$ are the unstable/stable manifolds for the sequence
$(\psi_{x,m})_{m\in\mathbb Z}$.

We now show statement~(7), with the statement~(8) proved similarly.
Assume that $y\in M$ and $d(\varphi^t(y),\varphi^t(x))\leq\varepsilon_0$ for all $t\leq 0$. Arguing
similarly to the proof of Theorem~\ref{t:stun-maps} we see that
$$
\psi_{x,-n}^{-1}\cdots\psi_{x,-1}^{-1}(\widetilde\varkappa_x(y))\in \{(w_0,w_1,w_2)\colon \max(|w_1|,|(w_0,w_2)|)\leq 1\}\quad\text{for all}\quad n\geq 0.
$$
It follows from~\eqref{e:omega-x-def} that
$$
\omega_{x,-n}^{-1}\cdots\omega_{x,-1}^{-1}(\pi_{us}(\widetilde\varkappa_x(y)))
\in\{(w_1,w_2)\colon \max(|w_1|,|w_2|)\leq 1\}\quad\text{for all}\quad n\geq 0.
$$
Applying statement~(5) in Theorem~\ref{t:stun-adv} for the maps $(\omega_{x,m})_{m\in\mathbb Z}$,
we see that $\pi_{us}(\widetilde\varkappa_x(y))\in \pi_{us}(W^u_{x,0})$ and thus
\begin{equation}
  \label{e:fiji}
\widetilde\varkappa_x(y)=w+(s,0,0)\quad\text{for some}\quad
w\in W^u_{x,0},\quad
s\in [-2,2].
\end{equation}
Now, assume additionally that $d(\varphi^t(y),\varphi^t(x))\to 0$ as $t\to -\infty$.
Then
\begin{equation}
  \label{e:fiji2}
\psi_{x,-n}^{-1}\cdots\psi_{x,-1}^{-1}(\widetilde\varkappa_x(y))\to 0\quad\text{as}\quad n\to\infty.
\end{equation}
Using~\eqref{e:flow-flowing} and~\eqref{e:fiji}, we see that
$$
\psi_{x,-n}^{-1}\cdots\psi_{x,-1}^{-1}(\widetilde\varkappa_x(y))=
\psi_{x,-n}^{-1}\cdots\psi_{x,-1}^{-1}(w)+(s,0,0).
$$
Since $w\in W^u_{x,0}$, by part~(3) of Theorem~\ref{t:stun-adv} we have
$\psi_{x,-n}^{-1}\cdots\psi_{x,-1}^{-1}(w)\to 0$ as $n\to\infty$.
Together with~\eqref{e:fiji2} this shows
that $s=0$. Therefore
$\widetilde\varkappa_x(y)=w\in W^u_{x,0}$ which implies
that $y\in W_u(x)$ as needed.
\end{proof}

\subsection{Further properties of hyperbolic flows}
  \label{s:flows-more}

We now discuss some further properties of the stable/unstable manifolds constructed
in Theorem~\ref{t:stun-flows}. Throughout this section we assume
that the conditions of Theorem~\ref{t:stun-flows} hold.

\subsubsection{Weak stable/unstable manifolds}
  \label{s:flows-weak}

Let $\delta_1>0$ be the rescaling parameter used in the proof of Theorem~\ref{t:stun-flows},
see~\eqref{e:risky-charts}. Recall that $\delta_1$ is chosen small depending only on $\varphi^t,K$.
For each $x\in K$ define the \emph{weak unstable/stable manifolds}
$$
W_{u0}(x):=\bigcup_{|s|\leq 2\delta_1} \varphi^s(W_u(x)),\quad
W_{s0}(x):=\bigcup_{|s|\leq 2\delta_1} \varphi^s(W_s(x)).
$$
In the adapted chart $\widetilde\varkappa_x$ from~\eqref{e:risky-charts}, we have
(using~\eqref{e:flow-u-def}--\eqref{e:flow-s-def} and the fact that $d\widetilde\varkappa_x$ maps the generator $X$ of the flow
to $\delta_1^{-1}\partial_{x_0}$)
\begin{equation}
  \label{e:weak-in-chart}
\begin{aligned}
W_{u0}(x)=\widetilde\varkappa_x^{-1}(W^{u0}_{x,0}),&\quad
W^{u0}_{x,0}:=\{w+(s,0,0)\colon w\in W^u_{x,0},\ |s|\leq 2\};\\
W_{s0}(x)=\widetilde\varkappa_x^{-1}(W^{s0}_{x,0}),&\quad
W^{s0}_{x,0}:=\{w+(s,0,0)\colon w\in W^s_{x,0},\ |s|\leq 2\}
\end{aligned}
\end{equation}
where $W^u_{x,0}$, $W^s_{x,0}$ are the unstable/stable manifolds
for the maps $\widetilde \psi_x$, see the proof of Theorem~\ref{t:stun-flows}.
Since $W^u_{x,0}$, $W^s_{x,0}$ are graphs of functions of $x_1$, $x_2$
(see~\eqref{e:wuwsy}), we see that $W^{u0}_{x,0},W^{s0}_{x,0}$,
and thus $W_{u0}(x),W_{s0}(x)$, are embedded $C^N$ submanifolds
of dimensions $d_u+1,d_s+1$.
It is also clear that
\begin{equation}
  \label{e:weak-tangent}
T_x W_{u0}(x)=E_u(x)\oplus E_0(x),\quad
T_x W_{s0}(x)=E_s(x)\oplus E_0(x).
\end{equation}
It follows from statement~(4) in Theorem~\ref{t:stun-flows} that
the weak stable/unstable manifolds are invariant under the positive/negative integer
time maps of the flow $\varphi^t$:
\begin{equation}
  \label{e:weak-invariant}
\varphi^{-1}(W_{u0}(x))\subset W_{u0}(\varphi^{-1}(x)),\quad
\varphi^1(W_{s0}(x))\subset W_{s0}(\varphi^1(x)).
\end{equation}
We next have the following versions of the statements~(7)--(8) in Theorem~\ref{t:stun-flows}:
for all $y\in M$,
\begin{equation}
  \label{e:weak-characterized}
\begin{aligned}
d(\varphi^t(y),\varphi^t(x))\leq\varepsilon_0\quad\text{for all }t\leq 0
&\quad\Longrightarrow\quad y\in W_{u0}(x);\\
d(\varphi^t(y),\varphi^t(x))\leq\varepsilon_0\quad\text{for all }t\geq 0
&\quad\Longrightarrow\quad y\in W_{s0}(x).
\end{aligned}
\end{equation}
The properties~\eqref{e:weak-characterized} follow immediately from~\eqref{e:fiji}
(and its analog for stable manifolds) and~\eqref{e:weak-in-chart}.

Similarly to the proof of statement~(9) of Theorem~\ref{t:stun-maps}
one can show the following transversality properties:
if $x,y\in K$ and $d(x,y)\leq\varepsilon_0$ then
\begin{equation}
  \label{e:weak-transverse}
W_{s0}(x)\cap W_u(y),W_s(x)\cap W_{u0}(y)\quad\text{have exactly one point each}.
\end{equation}

\subsubsection{Quantitative statements}
  \label{s:flows-quant}

We now discuss quantitative versions of the statements of Theorem~\ref{t:stun-flows},
which are the analogs of~\eqref{e:cream-1}--\eqref{e:cream-3}. We fix $\tilde\nu$ such that
$$
0<\tilde\nu<\nu
$$
and allow the manifolds $W_u,W_s$ and the constant $\varepsilon_0$ to depend on $\tilde\nu$.

The analog of~\eqref{e:cream-1} is given by the following: there exists
a constant $C$ such that for all $x\in K$ and $t\geq 0$
\begin{equation}
  \label{e:milk-1}
\begin{aligned}
y,\tilde y\in W_u(x)&\quad\Longrightarrow\quad
d(\varphi^{-t}(y),\varphi^{-t}(\tilde y))\leq Ce^{-\tilde\nu t}d(y,\tilde y);\\
y,\tilde y\in W_s(x)&\quad\Longrightarrow\quad
d(\varphi^{t}(y),\varphi^{t}(\tilde y))\leq Ce^{-\tilde\nu t}d(y,\tilde y).
\end{aligned}
\end{equation}
To show~\eqref{e:milk-1}, it suffices to consider the case of integer~$t$. The
latter case is proved similarly to~\eqref{e:cream-1} (see~\S\ref{s:maps-proof}),
since~\eqref{e:stuna-2} still applies to the maps $\psi_{x,m}$.

The analog of~\eqref{e:cream-2} is given by the following: for all $x\in K$, $y\in M$,
$t\geq 0$, and $0\leq\sigma\leq\varepsilon_0$
\begin{equation}
  \label{e:milk-2}
\begin{aligned}
d(\varphi^{-s}(y),\varphi^{-s}(x))\leq\sigma\text{ for all }s\in [0,t]
&\quad\Longrightarrow\quad
d(y,W_{u0}(x))\leq Ce^{-\tilde\nu t}\sigma,\\
d(\varphi^{s}(y),\varphi^{s}(x))\leq\sigma\text{ for all }s\in [0,t]
&\quad\Longrightarrow\quad
d(y,W_{s0}(x))\leq Ce^{-\tilde\nu t}\sigma.
\end{aligned}
\end{equation}
This is proved by applying~\eqref{e:stuna-3} for the reduced maps $\omega_{x,m}$ defined in~\eqref{e:omega-x-def},
see the proof of~\eqref{e:fiji}.

Finally the analog of~\eqref{e:cream-3} is given by the following:
for all $x\in K$, $y\in M$, $t\geq 0$, and $0\leq \sigma\leq \varepsilon_0$,
\begin{equation}
  \label{e:milk-3}
\begin{gathered}
d(\varphi^s(y),\varphi^s(x))\leq\sigma\quad\text{for all}\quad s\in [-t,t]\\
\Longrightarrow\quad
d(y,\varphi^r(x))\leq Ce^{-\tilde\nu t}\sigma\quad\text{for some }r\in [-2\delta_1,2\delta_1].
\end{gathered}
\end{equation}
This is proved by applying~\eqref{e:stuna-4} to the reduced maps $\omega_{x,m}$,
arguing similarly to the proof of~\eqref{e:fiji}.

\subsubsection{Local invariance and global stable/unstable manifolds}
\label{s:flow-invariance}

In this section we discuss local invariance of the stable/unstable manifolds.
We first discuss invariance under the flow $\varphi^t$. Theorem~\ref{t:stun-flows}
does not imply that $\varphi^{-t}(W_u(x))\subset W_u(x)$,
$\varphi^t(W_s(x))\subset W_s(x)$ for non-integer $t\geq 0$,
however local versions of this statement are established in~\eqref{e:victor-1u}--\eqref{e:victor-2s} below.

Similarly to~\eqref{e:iterated-unstable}, \eqref{e:iterated-stable}
for each $x\in K$ and integer $k\geq 0$ define the iterated unstable/stable manifolds
\begin{equation}
  \label{e:iterated-stun-flows}
W^{(k)}_u(x):=\varphi^k(W_u(\varphi^{-k}(x))),\quad
W^{(k)}_s(x):=\varphi^{-k}(W_s(\varphi^k(x))).
\end{equation}
It follows from statement~(4) in Theorem~\ref{t:stun-flows}
that for all $k\geq 0$
$$
W^{(k)}_u(x)\subset W^{(k+1)}_u(x),\quad
W^{(k)}_s(x)\subset W^{(k+1)}_s(x).
$$
Local invariance of the unstable/stable manifolds under the flow is given by the following statements:
there exist $k_0\geq 0$ and $\varepsilon_1>0$ depending only on $\varphi^t,K$ such that
for all $x\in K$ and $s\in [-1,1]$
\begin{align}
  \label{e:victor-1u}
\varphi^s(W_u(x))&\subset W_u^{(k_0)}(\varphi^s(x)),\\
  \label{e:victor-1s}
\varphi^s(W_s(x))&\subset W_s^{(k_0)}(\varphi^s(x)),\\
  \label{e:victor-2u}
\varphi^s(W_u(x))\cap \overline B_d(\varphi^s(x),\varepsilon_1)&\subset W_u(\varphi^s(x)),\\
  \label{e:victor-2s}
\varphi^s(W_s(x))\cap \overline B_d(\varphi^s(x),\varepsilon_1)&\subset W_s(\varphi^s(x)).
\end{align}
Note that~\eqref{e:victor-1u}, \eqref{e:victor-2u} can be extended to all $s\leq 0$
and~\eqref{e:victor-1s}, \eqref{e:victor-2s} to all $s\geq 0$ using statement~(4) in Theorem~\ref{t:stun-flows}.
 
We show~\eqref{e:victor-1s}, \eqref{e:victor-2s}, with~\eqref{e:victor-1u}, \eqref{e:victor-2u} proved similarly.
We start with~\eqref{e:victor-1s}. Assume that $y\in W_s(x)$. Then
by~\eqref{e:milk-1} we have for all $t\geq 0$
\begin{equation}
  \label{e:milky-way-1}
d(\varphi^{t+s}(y),\varphi^{t+s}(x))\leq
Cd(\varphi^t(y),\varphi^t(x))\leq Ce^{-\tilde\nu t}.
\end{equation}
Therefore, $d(\varphi^{t+s}(y),\varphi^{t+s}(x))\to 0$ as $t\to\infty$ and there exists $k_0\geq 0$ such that
\begin{equation}
  \label{e:milky-way-2}
d(\varphi^{t+s+k_0}(y),\varphi^{t+s+k_0}(x))\leq \varepsilon_0\quad\text{for all}\quad t\geq 0.
\end{equation}
It follows from statement~(8) in Theorem~\ref{t:stun-flows} that 
$\varphi^{s+k_0}(y)\in W_s(\varphi^{s+k_0}(x))$ and thus $\varphi^s(y)\in W_s^{(k_0)}(\varphi^s(x))$ as needed.

To show~\eqref{e:victor-2s}, assume that $y\in W_s(x)$ and $d(\varphi^s(y),\varphi^s(x))\leq\varepsilon_1$.
Then~\eqref{e:milky-way-2} holds.
If $\varepsilon_1$ is small enough, then~\eqref{e:milky-way-2} holds also for all $t\in [-k_0,0]$,
and thus for all $t\geq -k_0$. It then follows from statement~(8) in Theorem~\ref{t:stun-flows}
that $\varphi^s(y)\in W_s(\varphi^s(x))$ as needed.

Next, we note that the statements~\eqref{e:intersector}, \eqref{e:tangentor},
\eqref{e:locally-unique-1}, and~\eqref{e:locally-unique-2} relating
the stable/unstable manifolds at different points are still valid in the case of flows,
with very similar proofs.

Finally, similarly to~\eqref{e:global} we can define the \emph{global unstable/stable manifolds}:
for $x\in K$,
\begin{equation}
  \label{e:global-flows}
W^{(\infty)}_u(x):=\bigcup_{k\geq 0}W_u^{(k)}(x),\quad
W^{(\infty)}_s(x):=\bigcup_{k\geq 0}W_s^{(k)}(x).
\end{equation}
By statements~(5)--(8) in Theorem~\ref{t:stun-flows}, these can be characterized as follows:
\begin{equation}
  \label{e:global-flows-char}
\begin{aligned}
y\in W^{(\infty)}_u(x)&\quad\Longleftrightarrow\quad
d(\varphi^t(y),\varphi^t(x))\to 0\quad\text{as }t\to -\infty;\\
y\in W^{(\infty)}_s(x)&\quad\Longleftrightarrow\quad
d(\varphi^t(y),\varphi^t(x))\to 0\quad\text{as }t\to \infty.
\end{aligned}
\end{equation}
The manifolds $W^{(\infty)}_u(x)$, $W^{(\infty)}_s(x)$ are $d_u$ and $d_s$-dimensional
immersed submanifolds without boundary in $M$. They are invariant under the flow
$\varphi^t$ and disjoint from each other.
  
\section{Examples}
  \label{s:examples}

In this section we provide two examples of hyperbolic systems:
geodesic flows on negatively curved surfaces (\S\ref{s:surf-neg}) and billiard ball maps
on Euclidean domains with concave boundary (\S\ref{s:billiard}).

\subsection{Surfaces of negative curvature}
  \label{s:surf-neg}

Let $(M,g)$ be a closed (compact without boundary) oriented surface.
To simplify the computations below, we will often use (positively oriented) \emph{isothermal coordinates}
$(x,y)$ in which the metric is conformally flat:
\begin{equation}
  \label{e:flatty}
g=e^{2G(x,y)}(dx^2+dy^2).
\end{equation}
Such coordinates exist locally near each point of $M$. In isothermal coordinates, the Gauss curvature
is given by
\begin{equation}
  \label{e:kurvy}
K(x,y)=-e^{-2G(x,y)}(\partial_x^2+\partial_y^2)G(x,y).
\end{equation}
We show below in Theorem~\ref{t:geodesic-hyperbolic} that if $K<0$ then the geodesic flow
on $(M,g)$ is hyperbolic. We define
the geodesic flow as the Hamiltonian flow
\begin{equation}
  \label{e:gerry}
\begin{aligned}
\varphi^t:=\exp(tX):S^*M\to S^*M,&\qquad
X:=H_p,\\
S^*M:=\{(x,\xi)\in T^*M\colon p(x,\xi)=1\},&\qquad
p(x,\xi):=|\xi|_g
\end{aligned}
\end{equation}
where $T^*M$ is the cotangent bundle of $M$ and $S^*M$ is the unit cotangent bundle (with respect to the metric $g$).
In local coordinates $(x,y)$, if $(\xi,\eta)$ are the corresponding momentum variables (i.e. coordinates
on the fibers of $T^*M$) then
\begin{equation}
  \label{e:hammy}
X=H_p=(\partial_\xi p)\partial_x+(\partial_\eta p)\partial_y
-(\partial_x p)\partial_\xi-(\partial_y p)\partial_\eta.
\end{equation}
In isothermal coordinates~\eqref{e:flatty} we have
\begin{equation}
  \label{e:poppy}
p(x,y,\xi,\eta)=e^{-G(x,y)}\sqrt{\xi^2+\eta^2}.
\end{equation}
\begin{theo}
  \label{t:geodesic-hyperbolic}
Let $(M,g)$ be a closed oriented surface with geodesic flow $\varphi^t:S^*M\to S^*M$ defined in~\eqref{e:gerry}.
Assume that the Gauss curvature $K$ is negative everywhere. Then
$\varphi^t$ is an Anosov flow, that is $\varphi^t$ is hyperbolic on the entire $S^*M$
in the sense of Definition~\ref{d:hyp-flows}.
\end{theo}
\Remark Theorem~\ref{t:geodesic-hyperbolic} extends to higher dimensional manifolds
of \emph{negative sectional curvature}~-- see for instance~\cite[Theorem~17.6.2]{KaHa}.
The orientability hypothesis is made only for convenience of the proof, one can
remove it for instance by passing to a double cover of $M$.

In the remainder of this section we prove Theorem~\ref{t:geodesic-hyperbolic}.
We first define a convenient frame on $S^*M$. Let $V$ be the vector field on $S^*M$ generating
rotations on the circle fibers (counterclockwise with respect to the fixed orientation).
If $(x,y)$ are isothermal coordinates~\eqref{e:flatty}, we use local coordinates
$(x,y,\theta)$ on $S^*M$ where $\theta$ is defined by
$$
\xi=e^{G(x,y)}\cos\theta,\quad
\eta=e^{G(x,y)}\sin\theta,
$$
and we have (here $X$ is the generator of the geodesic flow)
$$
V=\partial_\theta,\quad
X=e^{-G(x,y)}\big(\cos\theta\,\partial_x+\sin\theta\,\partial_y
+(\partial_y G(x,y)\cos\theta-\partial_x G(x,y)\sin\theta)\partial_\theta\big).
$$
Define the vector field
$$
X_\perp:=[X,V],
$$
in isothermal coordinates
$$
X_\perp=e^{-G(x,y)}\big(\sin\theta\,\partial_x-\cos\theta\,\partial_y
+(\partial_x G(x,y)\cos\theta+\partial_y G(x,y)\sin\theta)\partial_\theta\big).
$$
The vector fields $X,V,X_\perp$ form a global frame on $S^*M$ and we have
(using~\eqref{e:kurvy})
\begin{equation}
  \label{e:commy}
[X,V]=X_\perp,\quad
[X_\perp,V]=-X,\quad
[X,X_\perp]=-KV.
\end{equation}
For any vector field $W$ on $S^*M$, we have (by the standard properties of Lie bracket)
\begin{equation}
  \label{e:lie}
\partial_t(\varphi^{-t}_*W)=\varphi^{-t}_*[X,W]\quad\text{where}\quad
\varphi^{-t}_*W(\rho):=d\varphi^{-t}(\rho)W(\varphi^t(\rho)),\ \rho\in S^*M.
\end{equation}
It then follows from~\eqref{e:commy} that the following two-dimensional
subbundle of $T(S^*M)$ is invariant under the flow $\varphi^t$:
$$
E_{us}:=\Span(V,X_\perp).
$$
The space $E_{us}$ will be the direct sum of the stable and the unstable subspaces
for the flow $\varphi^t$. For a vector $v\in E_{us}(\rho)$, $\rho\in S^*M$, we define
its coordinates $(a,b)$ with respect to the frame $V,X_\perp$:
\begin{equation}
  \label{e:cordial}
v=aV(\rho)+bX_\perp(\rho).
\end{equation}
We have
the following differential equations for the action of $d\varphi^t$ on $E_{us}$
(which are a special case of Jacobi's equations):
\begin{lemm}
  \label{l:jacobi}
Let $\rho\in S^*M$, $v\in E_{us}(\rho)$, and denote
$$
\rho(t)=(x(t),\xi(t)):=\varphi^t(\rho)\in S^*M,\quad
v(t):=d\varphi^t(\rho)v\in E_{us}(\rho(t)).
$$
Let $a(t),b(t)$ be the coordinates of~$v(t)$ defined in~\eqref{e:cordial}.
Put $K(t):=K(x(t))$.
Then, denoting by dots derivatives with respect to~$t$,
the functions $a(t),b(t)$
satisfy the ordinary differential equation
\begin{equation}
  \label{e:diffy}
\dot a=K(t)b,\quad
\dot b=-a.
\end{equation}
\end{lemm}
\begin{proof}
Define the vector field $W_t:=a(t)V+b(t)X_\perp$. Then
$W_t(\varphi^t(\rho))=v(t)=d\varphi^t(\rho)v$ and thus
$v=\varphi^{-t}_*W_t(\rho)$. By~\eqref{e:lie} we have
$$
0=\partial_t(\varphi^{-t}_*W_t)(\rho)=\varphi^{-t}_*\big(\partial_t W_t+[X,W_t]\big)(\rho).
$$
Using~\eqref{e:commy}, we get
$$
0=(\partial_t W_t+[X,W_t])(\rho(t))=\big(\dot a(t)-K(x(t))b(t)\big)V(\rho(t))+\big(\dot b(t)+a(t)\big)X_\perp(\rho(t)\big).
$$
This implies~\eqref{e:diffy}.
\end{proof}
Lemma~\ref{l:jacobi} immediately implies Theorem~\ref{t:geodesic-hyperbolic} in the special
case of \emph{constant curvature} $K\equiv -1$, with expansion rate $\nu=1$ in~\eqref{e:hypflow-exp}
and $E_u,E_s$ given by
$$
E_u=\Span(V-X_\perp),\quad
E_s=\Span(V+X_\perp).
$$

To handle the case of variable curvature we first construct cones in $E_{us}$ which are
invariant under the flow $\varphi^t$ for positive and negative times
(see Figure~\ref{f:cones}):
\begin{lemm}
  \label{l:cones-1}
For each $\rho\in S^*M$, define the closed cones
$\mathcal C^u_0(\rho),\mathcal C^s_0(\rho)\subset E_{us}(\rho)$:
$$
\mathcal C^u_0(\rho):=\{aV(\rho)+bX_\perp(\rho)\mid ab\leq 0\},\quad
\mathcal C^s_0(\rho):=\{aV(\rho)+bX_\perp(\rho)\mid ab\geq 0\}.
$$
Assume that $K\leq 0$ everywhere on $M$. Then for all $t\geq 0$
\begin{equation}
  \label{e:cones-1}
d\varphi^t(\mathcal C^u_0(\rho))\subset \mathcal C^u_0(\varphi^t(\rho)),\quad
d\varphi^{-t}(\mathcal C^s_0(\rho))\subset \mathcal C^s_0(\varphi^{-t}(\rho)).
\end{equation}
\end{lemm}
\begin{proof}
In the notation of Lemma~\ref{l:jacobi} we have
$$
\partial_t(a(t)b(t))=-a(t)^2+K(x(t))b(t)^2\leq 0
$$
and~\eqref{e:cones-1} follows immediately.
\end{proof}
We now prove an upgraded version of Lemma~\ref{l:cones-1}, constructing invariant cones
on which the differentials $d\varphi^t$, $d\varphi^{-t}$ are expanding. Fix small constants $\zeta>0$,
$\gamma>0$ to be chosen later in Lemma~\ref{l:cones-2}. Define the norm $|\bullet|$
on the fibers of $E_{us}$ as follows:
$$
|aV(\rho)+bX_\perp(\rho)|:=\sqrt{\zeta a^2+b^2}.
$$
Define also the following dilation invariant function $\Theta$ on the fibers of $E_{us}\setminus 0$:
$$
\Theta(aV(\rho)+bX_\perp(\rho)):={ab\over \zeta a^2+b^2}.
$$
The upgraded invariant cones are constructed in
\begin{figure}
\includegraphics[scale=0.25]{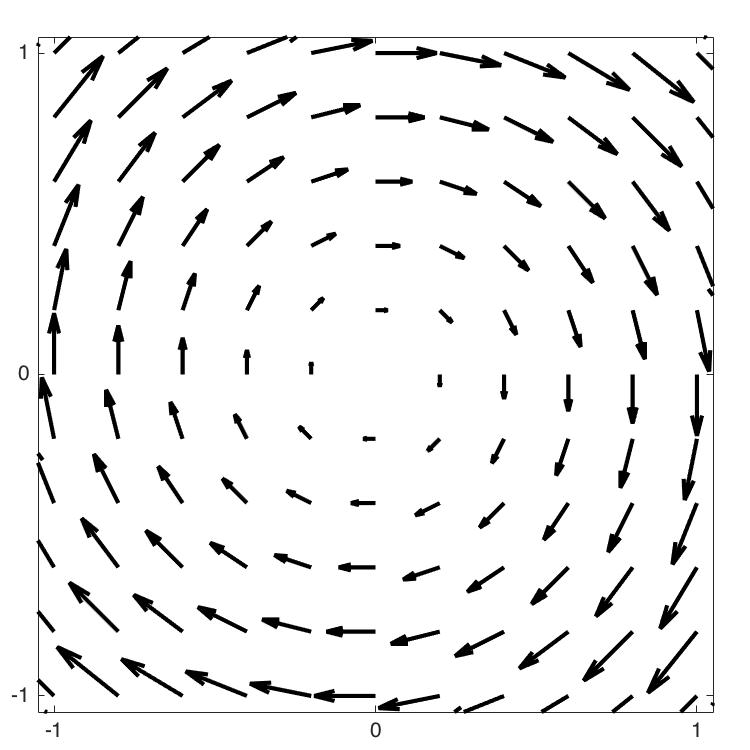}\qquad
\includegraphics[scale=0.25]{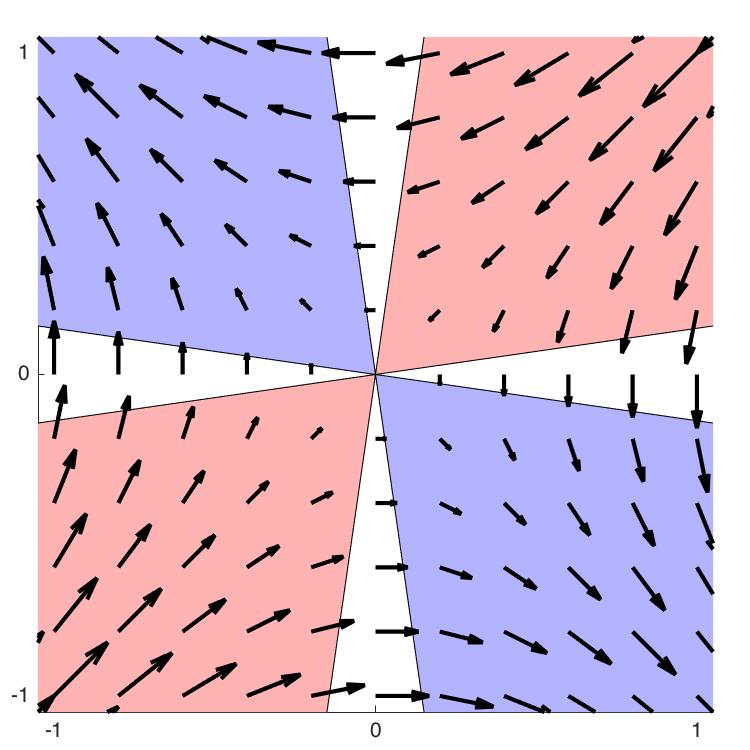}
\caption{Direction fields for the equation~\eqref{e:diffy}
for $K=1$ (left) and $K=-1$ (right) in the $(a,b)$ plane.
The cones on the right are $\mathcal C^u_\gamma$ (blue)
and $\mathcal C^s_\gamma$ (red).}
\label{f:cones}
\end{figure}
%
\begin{lemm}
  \label{l:cones-2}
Assume that $K<0$ everywhere on $M$. Then there exist $\zeta>0$, $\gamma>0$, $\nu>0$ such that
the closed cones
$\mathcal C^u_\gamma(\rho),\mathcal C^s_\gamma(\rho)\subset E_{us}(\rho)$
defined by (see Figure~\ref{f:cones})
$$
\mathcal C^u_\gamma(\rho):=\{v\in E_{us}(\rho)\colon \Theta(v)\leq -\gamma\}\cup\{0\},\quad
\mathcal C^s_\gamma(\rho):=\{v\in E_{us}(\rho)\colon \Theta(v)\geq \gamma\}\cup\{0\}
$$
have the following properties for all $\rho\in S^*M$ and $t\geq 0$
\begin{align}
  \label{e:coney-1}
d\varphi^t(\rho)\mathcal C^u_\gamma(\rho)&\subset \mathcal C^u_\gamma(\varphi^t(\rho));\\
  \label{e:coney-2}
d\varphi^{-t}(\rho)\mathcal C^s_\gamma(\rho)&\subset \mathcal C^s_\gamma(\varphi^{-t}(\rho));\\
  \label{e:coney-3}
|d\varphi^t(\rho)v|&\geq e^{\nu t}|v|\quad\text{for all}\quad v\in \mathcal C^u_\gamma(\rho);\\
  \label{e:coney-4}
|d\varphi^{-t}(\rho)v|&\geq e^{\nu t}|v|\quad\text{for all}\quad v\in \mathcal C^s_\gamma(\rho).
\end{align}
\end{lemm}
\begin{proof}
We fix constants $K_0,K_1>0$ such that 
$$
0<K_0\leq -K(x)\leq K_1\quad\text{for all}\quad x\in M
$$
and put
\begin{equation}
  \label{e:zeta-gamma-def}
\zeta:={1\over K_1},\quad
\gamma:={\sqrt K_0\over 3},\quad
\nu:=\gamma(1+\zeta K_0).
\end{equation}
Let $\rho\in S^*M$, $v\in E_{us}(\rho)$, and
put $\rho(t):=\varphi^t(\rho)$, $v(t)=d\varphi^t(\rho)v$. We write
$v(t)=a(t)V(\rho(t))+b(t)X_\perp(\rho(t))$ and recall that
$a(t),b(t)$ satisfy the differential equations~\eqref{e:diffy}.
Denote $K(t):=K(x(t))$ where $\rho(t)=(x(t),\xi(t))$ and
$$
R(t):=|v(t)|^2=\zeta a(t)^2+b(t)^2,\quad
\Theta(t):=\Theta(v(t))={a(t)b(t)\over \zeta a(t)^2+b(t)^2}.
$$
Then it follows from~\eqref{e:diffy} that (denoting by dots derivatives with respect to~$t$)
\begin{equation}
  \label{e:dotty}
\dot R=-2(1-\zeta K(t))\Theta R,\quad
\dot\Theta=-{a^2-Kb^2\over R}+2(1-\zeta K(t))\Theta^2.
\end{equation}
Therefore
$$
\dot\Theta\leq -{a^2+K_0b^2\over R}+2(1+\zeta K_1)\Theta^2
\leq -K_0+4\Theta^2.
$$
Thus if $|\Theta(t)|=\gamma$ for some $t$, then
$\dot\Theta(t)<0$. In particular, if $\Theta(0)\leq -\gamma$,
then $\Theta(t)\leq -\gamma$ for all $t\geq 0$, which implies~\eqref{e:coney-1}.
Similarly if $\Theta(0)\geq \gamma$, then
$\Theta(t)\geq\gamma$ for all $t\leq 0$, which implies~\eqref{e:coney-2}.

We next prove~\eqref{e:coney-3}. Assume that $v\in \mathcal C^u_\gamma(\rho)\setminus 0$,
then $\Theta(t)\leq -\gamma$ for all $t\geq 0$. 
From~\eqref{e:dotty} we have
$$
\dot R(t)\geq 2\gamma(1+\zeta K_0)R(t)= 2\nu R(t)\quad\text{for all}\quad t\geq 0.
$$
Therefore $R(t)\geq e^{2\nu t}R(0)$ for all $t\geq 0$, which gives~\eqref{e:coney-3}.
The bound~\eqref{e:coney-4} is proved similarly.
\end{proof}
We finally use the cones from Lemma~\ref{l:cones-2} to construct the stable/unstable spaces,
finishing the proof of Theorem~\ref{t:geodesic-hyperbolic}:
\begin{figure}
\includegraphics[scale=0.35]{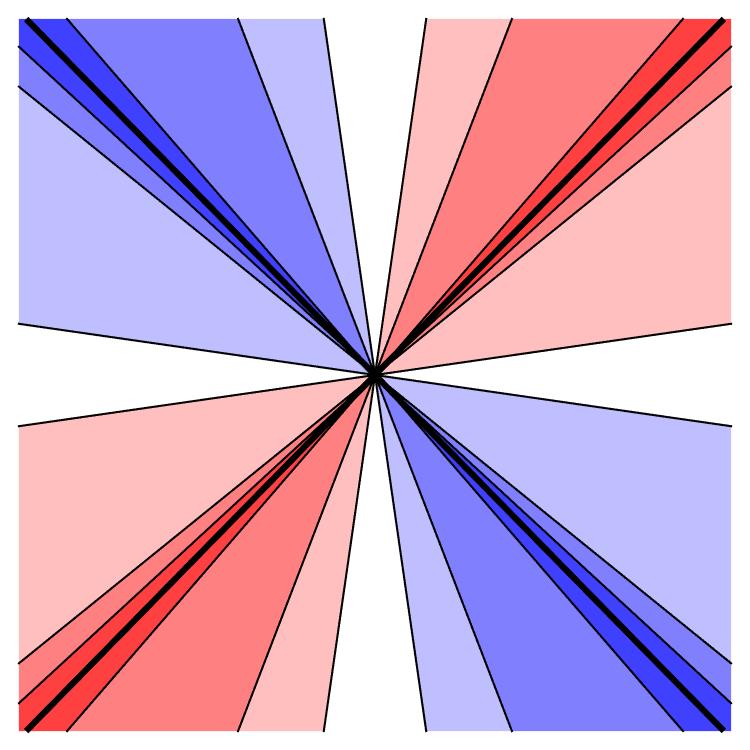}
\caption{The cones $\mathcal C^u_{\gamma,t}(\rho)$ (blue) and $\mathcal C^s_{\gamma,t}(\rho)$ (red) for several
values of~$t$, with the darker colors representing larger values of~$t$. The solid lines are
the spaces $E_u(\rho),E_s(\rho)$.}
\label{f:manycones}
\end{figure}
%
\begin{lemm}
  \label{l:geodfinal}
Let $\gamma$ be chosen in Lemma~\ref{l:cones-2}.
For each $\rho\in S^*M$ and $t\geq 0$ define the subsets of $E_{us}(\rho)$
$$
\mathcal C^u_{\gamma,t}(\rho):=d\varphi^t(\varphi^{-t}(\rho))\mathcal C^u_\gamma(\varphi^{-t}(\rho)),\quad
\mathcal C^s_{\gamma,t}(\rho):=d\varphi^{-t}(\varphi^{t}(\rho))\mathcal C^s_\gamma(\varphi^{t}(\rho)),
$$
and consider their intersections (see Figure~\ref{f:manycones})
\begin{equation}
  \label{e:geodfinal-2}
E_u(\rho):=\bigcap_{t\geq 0}\mathcal C^u_{\gamma,t}(\rho),\quad
E_s(\rho):=\bigcap_{t\geq 0}\mathcal C^s_{\gamma,t}(\rho).
\end{equation}
Then $E_u(\rho),E_s(\rho)$ are one-dimensional subspaces of $E_{us}(\rho)$
which satisfy the conditions of Definition~\ref{d:hyp-flows}.
\end{lemm}
\begin{proof}
We show the properties of $E_u(\rho)$. The properties of $E_s(\rho)$ are proved similarly
and the transversality of $E_u(\rho)$ and $E_s(\rho)$ follows from the fact that
$E_u(\rho)\subset \mathcal C^u_\gamma(\rho)$, $E_s(\rho)\subset \mathcal C^s_\gamma(\rho)$,
and $\mathcal C^u_\gamma(\rho)\cap \mathcal C^s_\gamma(\rho)=\{0\}$.

It follows from~\eqref{e:coney-1} that
$$
\mathcal C^u_{\gamma,s}(\rho)\subset\mathcal C^u_{\gamma,t}(\rho)\quad\text{when}\quad s\geq t\geq 0.
$$
We first claim that $E_u(\rho)$ contains a one-dimensional subspace of $E_{us}(\rho)$. Indeed,
let $\mathscr G$ be the Grassmanian of all one-dimensional subspaces of $E_{us}(\rho)$
and $\mathscr V_t\subset\mathscr G$ consist of the subspaces which are contained in $\mathcal C^u_{\gamma,t}(\rho)$.
Then $\mathscr V_s\subset\mathscr V_t$ for $s\geq t$ and all the sets $\mathscr V_t$ are compact.
Moreover, each $\mathscr V_t$ is nonempty since (recalling~\eqref{e:zeta-gamma-def})
\begin{equation}
  \label{e:coller0}
V^0_u(\rho):=\{
aV(\rho)+bX_\perp(\rho)\mid b=-a\sqrt{\zeta}\}
\subset \mathcal C^u_\gamma(\rho).
\end{equation}
Therefore the intersection $\bigcap_{t\geq 0}\mathscr V_t\subset \mathscr G$ is nonempty. Take an element
$V_u(\rho)$ of this intersection, then $V_u(\rho)$ is a one-dimensional subspace of $E_{us}(\rho)$
and $V_u(\rho)\subset E_u(\rho)$.

We now claim that
\begin{equation}
  \label{e:coller1}
E_u(\rho)=V_u(\rho).
\end{equation}
For that, it suffices to show that every $v\in E_u(\rho)$ lies in $V_u(\rho)$.
Define the one-dimensional space $V^0_s(\rho)\subset \mathcal C^s_\gamma(\rho)$ similarly to~\eqref{e:coller0}
but putting $b=a\sqrt{\zeta}$. Since $V_u(\rho)\subset \mathcal C^u_\gamma(\rho)$, the spaces
$V_u(\rho)$ and $V^0_s(\rho)$ are transverse to each other. Thus we can write
$$
v=v_1+v_2\quad\text{for some}\quad
v_1\in V_u(\rho),\quad
v_2\in V^0_s(\rho).
$$
Denote
$$
v(t):=d\varphi^{-t}(\rho)v,\quad
v_1(t):=d\varphi^{-t}(\rho)v_1,\quad
v_2(t):=d\varphi^{-t}(\rho) v_2.
$$
Since $v,v_1\in E_u(\rho)$, we have $v(t),v_1(t)\in\mathcal C^u_\gamma(\varphi^{-t}(\rho))$ for all $t\geq 0$.
It follows from~\eqref{e:coney-3} applied to $v(t),v_1(t)$ that
\begin{equation}
  \label{e:viper-1}
|v_2(t)|\leq |v(t)|+|v_1(t)|\leq e^{-\nu t}(|v|+|v_1|)\quad\text{for all}\quad t\geq 0.
\end{equation}
On the other hand, since $v_2\in \mathcal C^s_\gamma(\rho)$ we have by~\eqref{e:coney-4}
\begin{equation}
  \label{e:viper-2}
|v_2|\leq e^{-\nu t}|v_2(t)|\quad\text{for all}\quad t\geq 0.
\end{equation}
Combining~\eqref{e:viper-1} and~\eqref{e:viper-2} and letting $t\to\infty$ we get $v_2=0$, thus
$v=v_1\in V_u(\rho)$ which gives~\eqref{e:coller1}.

Now, \eqref{e:coller1} implies immediately that $E_u(\rho)$ is a one-dimensional subspace of $E_{us}(\rho)$.
We have $d\varphi^{-t}(\rho)E_u(\rho)\subset E_u(\varphi^{-t}(\rho))$ for all $t\geq 0$, which gives the invariance property~\eqref{e:hypflow-inv}.
The expansion property~\eqref{e:hypflow-exp}
follows from~\eqref{e:coney-3} and the inclusion $E_u(\rho)\subset\mathcal C^u_\gamma(\rho)$.
\end{proof}
\Remark The proof of Lemma~\ref{l:geodfinal} can be used to show the stability of Anosov maps/flows under perturbations,
namely a small $C^N$ perturbation of an Anosov map/flow is still Anosov. This uses the fact that (a slightly modified version of)
the properties~\eqref{e:coney-1}--\eqref{e:coney-4} is stable under perturbations; note it is enough to require these properties
for $0\leq t\leq 1$. See~\cite[Corollary~6.4.7 and Proposition~17.4.4]{KaHa} for details.

\subsection{A simple case of the billiard ball map}
  \label{s:billiard}
  
We finally briefly discuss two-dimensional billiard ball maps.
Let $\Omega\subset\mathbb R^2$ be a domain with smooth boundary,
referring the reader to~\cite{ChernovBook} for a comprehensive treatment and history of the subject.
We do not require $\Omega$ to be bounded but we require that the boundary $\partial\Omega$ be compact;
this implies that either $\Omega$ is compact (\emph{interior case}) or $\mathbb R^2\setminus\Omega^\circ$
is compact (\emph{exterior case}).

The boundary $\partial\Omega$ is diffeomorphic to the union of finitely many circles.
We parametrize it locally by a real number variable $\theta$, denoting the parametrization
$$
\mathbf x:\partial\Omega\to \mathbb R^2.
$$
We assume that $\mathbf x$ is a unit speed parametrization:
$$
|\mathbf v|\equiv 1\quad\text{where}\quad
\mathbf v(\theta):=\partial_\theta \mathbf x(\theta).
$$
We choose the inward pointing unit normal vector field
$$
\mathbf n:\partial\Omega\to\mathbb R^2.
$$
We assume that $(\mathbf v,\mathbf n)$ is a positively oriented frame
on $\mathbb R^2$ at every point of $\partial\Omega$. If $\partial\Omega$ consists of a single
circle, then the direction of increasing $\theta$ is counterclockwise in the
interior case and clockwise in the exterior case. For each $\theta\in\partial\Omega$
we define the \emph{curvature} $K(\theta)$ of the boundary at $\theta$ by the following identity:
$$
\partial_\theta\mathbf v(\theta)=K(\theta)\mathbf n(\theta).
$$
We say that the boundary is (strictly) \emph{convex} at some point $\theta$ if
$K(\theta)>0$ and (strictly) \emph{concave} if $K(\theta)<0$.
See Figure~\ref{f:obstacles}.
\begin{figure}
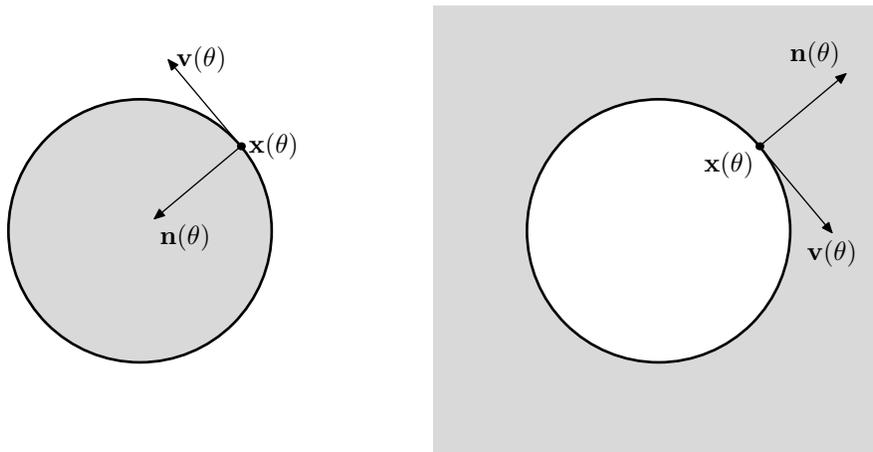

\includegraphics{stunnote.3}\qquad
\includegraphics{stunnote.4}
\caption{The interior (left) and exterior (right) case when the boundary
is a circle of radius~$r$. The domain $\Omega$ is shaded.
The curvature is equal to $1/r$ in the interior case
and to $-1/r$ in the exterior case.}
\label{f:obstacles}
\end{figure}

The billiard ball map acts on the phase space $M$ which consists
of inward pointing unit vectors at the boundary:
$$
M=\{(\theta,\mathbf w)\mid \theta\in\partial\Omega,\
\mathbf w\in\mathbb R^2,\
|\mathbf w|=1,\
\langle \mathbf w,\mathbf n(\theta)\rangle\geq 0\}.
$$
The boundary of the phase space $\partial M$ consists of \emph{glancing} directions,
where $\mathbf w$ is tangent to $\partial\Omega$. These directions are what makes
billiards much more difficult to handle than closed manifolds; in these notes we ignore
entirely the complications resulting from glancing, restricting to the interior
of $M$. The interior $M^\circ$ can be parametrized by two numbers $\theta\in\partial\Omega$
and $\sigma\in (-1,1)$ defined by
$$
\sigma:=\langle \mathbf w,\mathbf v(\theta)\rangle.
$$
In particular $\sigma=0$ corresponds to vectors which are orthogonal to the boundary.

We now define the billiard ball map
$$
\varphi:M^\circ\to M^\circ,\quad
\varphi(\theta_1,\sigma_1)=(\theta_2,\sigma_2)
$$
where $\mathbf x(\theta_2)$ is the first intersection of $\partial\Omega$
with the ray $\{\mathbf x(\theta_1)+t\mathbf w_1\mid t>0\}$
and $\mathbf w_2$ is obtained by the law of reflection~-- see Figure~\ref{f:billiards}.
The map $\varphi$ is in fact only defined on an open subset of $M^\circ$ since the ray
might intersect $\partial\Omega$ in a glancing direction or (in the exterior case)
may escape to infinity without intersecting $\partial\Omega$, we ignore here the
issues arising from this fact.
\begin{figure}
\includegraphics{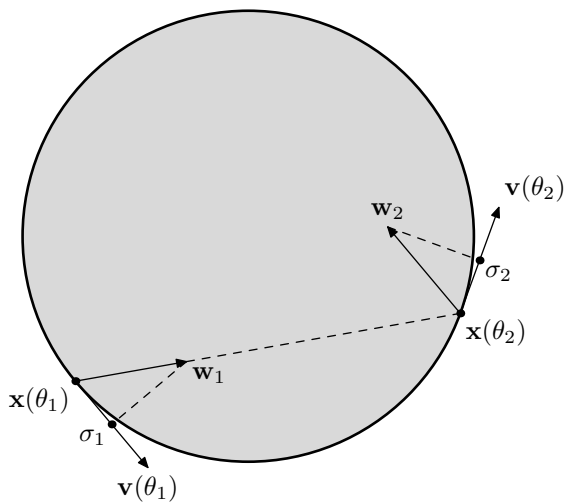}
\caption{The billiard ball map for the disk (interior case), given in case of the unit
disk by the formulas $\theta_2=\theta_1+2\arccos\sigma_1$, $\sigma_2=\sigma_1$.}
\label{f:billiards}
\end{figure}

Define the function on $\partial\Omega\times\partial\Omega$
$$
\Phi(\theta_1,\theta_2):=|\mathbf x(\theta_1)-\mathbf x(\theta_2)|.
$$
Then $\Phi$ is the generating function of $\varphi$, namely
\begin{equation}
  \label{e:billimap-1}
\varphi(\theta_1,\sigma_1)=(\theta_2,\sigma_2)\quad\Longleftrightarrow\quad
\sigma_1=-\partial_{\theta_1}\Phi(\theta_1,\theta_2),\
\sigma_2=\partial_{\theta_2}\Phi(\theta_1,\theta_2).
\end{equation}
(Strictly speaking, we should restrict the right-hand side above to $\theta_1\neq\theta_2$ such that
the interior of the line segment between $\mathbf x(\theta_1)$ and $\mathbf x(\theta_2)$
does not intersect $\partial\Omega$.)

We have
$$
\begin{gathered}
\partial_{\theta_1}^2\Phi={1-\sigma_1^2\over\Phi}-K(\theta_1)\sqrt{1-\sigma_1^2},\quad
\partial_{\theta_2}^2\Phi={1-\sigma_2^2\over\Phi}-K(\theta_2)\sqrt{1-\sigma_2^2},\\
\partial_{\theta_1\theta_2}\Phi={\sqrt{(1-\sigma_1^2)(1-\sigma_2^2)}\over\Phi}.
\end{gathered}
$$
Therefore, if $\varphi(\theta_1,\sigma_1)=(\theta_2,\sigma_2)$ and we denote
$\ell:=\Phi(\theta_1,\theta_2)$ (the distance traveled between the bounces)
and $K_1:=K(\theta_1)$, $K_2:=K(\theta_2)$, then
\begin{equation}
  \label{e:huge-matrix}
d\varphi(\theta_1,\sigma_1)=\begin{pmatrix}\displaystyle
{\ell K_1-\sqrt{1-\sigma_1^2}\over \sqrt{1-\sigma_2^2}} & \displaystyle
-{\ell\over \sqrt{(1-\sigma_1^2)(1-\sigma_2^2)}}\\ 
\noalign{\bigskip}
\displaystyle
K_1\sqrt{1-\sigma_2^2}+K_2\sqrt{1-\sigma_1^2}-\ell K_1K_2 & \displaystyle
{\ell K_2-\sqrt{1-\sigma_2^2}\over \sqrt{1-\sigma_1^2}}
\end{pmatrix}.
\end{equation}
Note that $\det d\varphi\equiv 1$.
\begin{figure}
\includegraphics{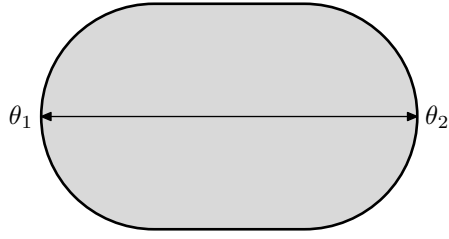}
\caption{A hyperbolic trajectory on a Bunimovich stadium}
\label{f:billimore}
\end{figure}

We now discuss under which conditions $\varphi$ is hyperbolic.
We start with the simplest case of a closed trajectory of period 2
(which is necessarily orthogonal to the boundary):
\begin{prop}
  \label{l:billiard-per2}
Assume that $(\theta_1,0)$ is a periodic point for $\varphi$ with period~2,
namely $\varphi(\theta_1,0)=(\theta_2,0)$ and $\varphi(\theta_2,0)=(\theta_1,0)$
for some $\theta_2\in\partial\Omega$. Let $\ell:=|\mathbf x(\theta_1)-\mathbf x(\theta_2)|$,
$K_1:=K(\theta_1)$, $K_2:=K(\theta_2)$.
Then $\varphi$ is hyperbolic on the closed trajectory $\{(\theta_j,0)\}$,
in the sense of Definition~\ref{d:hyp-map}, if and only if the following condition holds:
\begin{equation}
  \label{e:billiard-per2}
(1-\ell K_1)(1-\ell K_2)\notin [0, 1].
\end{equation}
\end{prop}
\Remark Condition~\eqref{e:billiard-per2} always holds in the concave case
(when $K_1,K_2<0$). However it sometimes also holds in the convex case, see Figure~\ref{f:billimore}.
\begin{proof}
Put
$$
A:=d\varphi^2(\theta_1,0)=d\varphi(\theta_2,0)d\varphi(\theta_1,0).
$$
Then $\varphi$ is hyperbolic on $\{(\theta_j,0)\}$ if and only if
$A$ has no eigenvalues on the unit circle. Since $\det A=1$, this is equivalent
to $|\tr A|>2$. Computing
$$
\tr A=2+4\ell(\ell K_1K_2-K_1-K_2)
$$
we arrive to the condition~\eqref{e:billiard-per2}.
\end{proof}
For general sets we restrict to the concave case (the corresponding billiards
are sometimes called \emph{dispersing}):
\begin{prop}
  \label{l:billiard-gen}
Assume that $\mathcal K\subset M^\circ$ is a $\varphi$-invariant compact set
and the curvature $K$ satisfies $K(\theta)<0$ for all $(\theta,\sigma)\in\mathcal K$.
Then the billiard ball map $\varphi$ is hyperbolic on~$\mathcal K$ in the sense
of Definition~\ref{d:hyp-map}.
\end{prop}
\Remark
One example of a compact $\varphi$-invariant set is a closed (non-glancing) trajectory.
Another example is when $\Omega$ is the complement of several strictly convex obstacles
(that is, the exterior concave case), we impose the \emph{no-eclipse condition}
that no obstacle intersects the convex hull of the union of any two other obstacles,
and $\mathcal K$ is the reduction to boundary of the \emph{trapped set} which
consists of all billiard ball trajectories which stay in a bounded set for all times.
The no-eclipse condition ensures that trapped trajectories cannot be glancing.
See Figure~\ref{f:several-obstacles}.
\begin{figure}
\includegraphics[width=5.5cm]{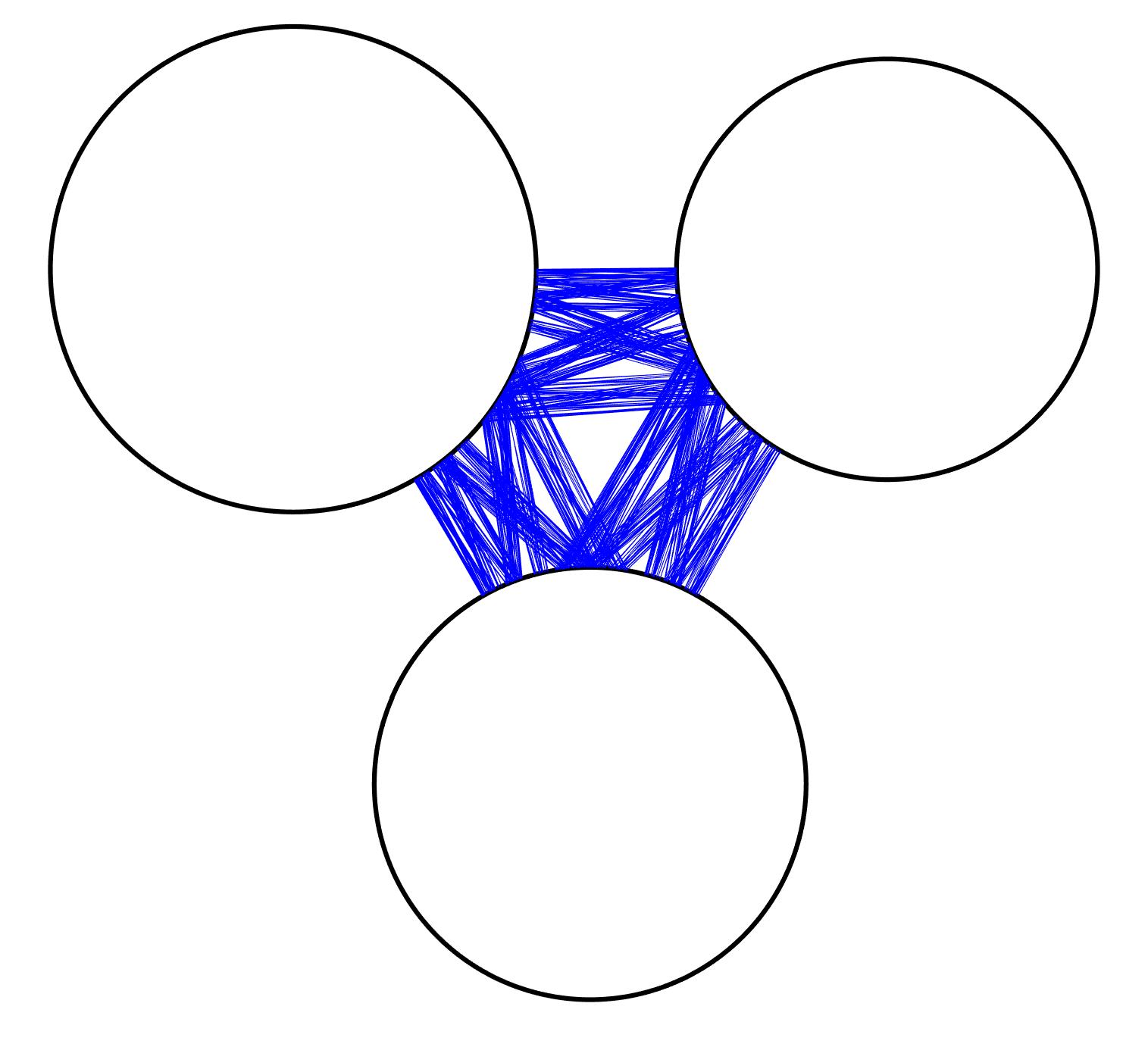}\quad
\includegraphics[width=8.25cm]{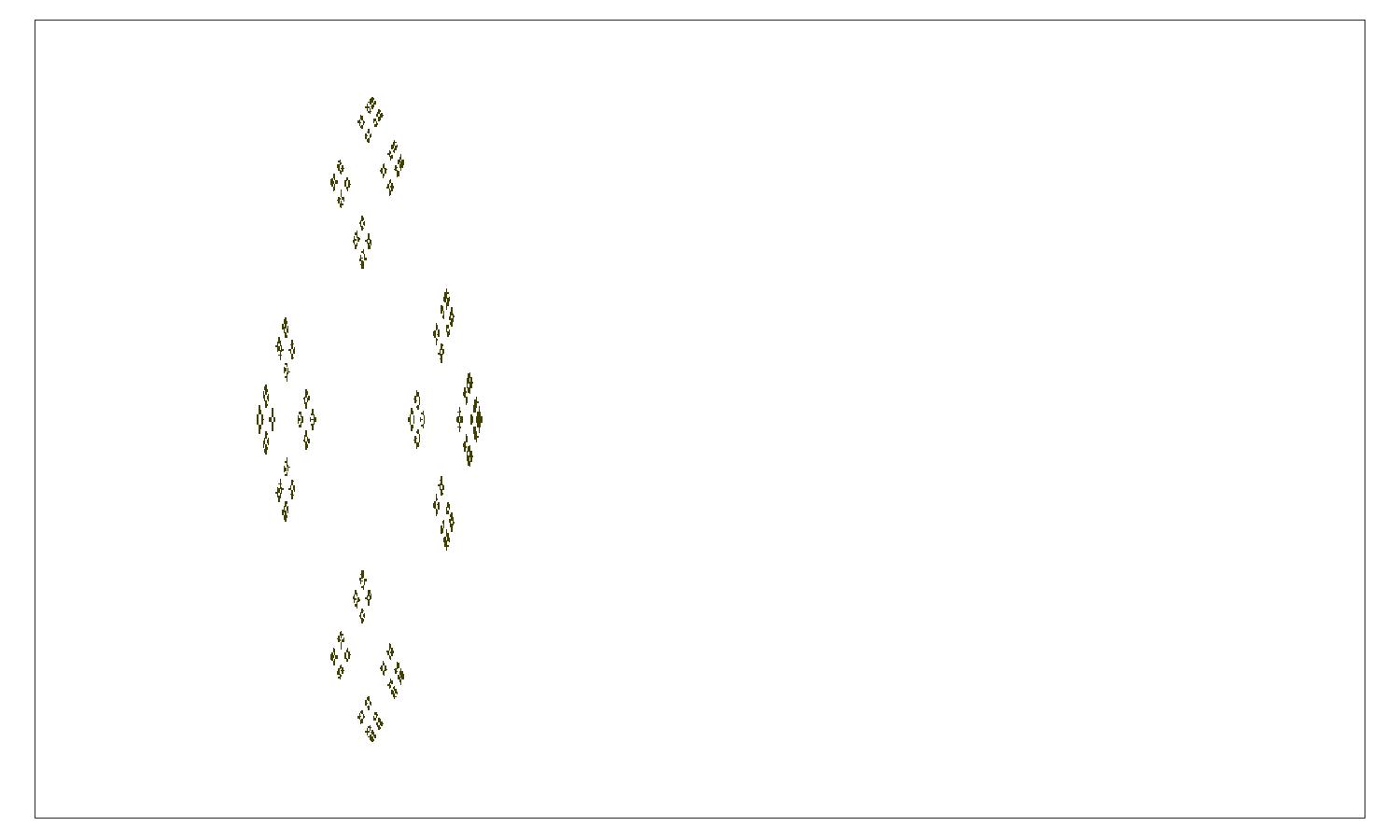}
\caption{Left: numerically computed set of trapped trajectories in the exterior of 3 disks.
Right: the corresponding reduction to the boundary $\mathcal K$,
that is the trapped set of the billiard ball map, restricted to the bottom
circle; the horizontal direction is $\theta$
and the vertical direction is $\sigma$. Both sets exhibit fractal structure.}
\label{f:several-obstacles}
\end{figure}
%
\begin{proof}[Sketch of the proof]
We argue similarly to~\S\ref{s:surf-neg}.
Consider the cones in $\mathbb R^2$
$$
\mathcal C^u_0:=\{(v_\theta,v_\sigma)\mid v_\theta\cdot v_\sigma\geq 0\},\quad
\mathcal C^s_0:=\{(v_\theta,v_\sigma)\mid v_\theta\cdot v_\sigma\leq 0\}.
$$
By the concavity condition, for all $\rho:=(\theta_1,\sigma_1)\in\mathcal K$
we have $K_1,K_2<0$ in~\eqref{e:huge-matrix}. Thus all entries
of the matrix $d\varphi(\rho)$ are negative.
It follows that
\begin{equation}
  \label{e:conifer-1}
d\varphi(\rho)\mathcal C^u_0\subset \mathcal C^u_0,\quad
d\varphi(\rho)^{-1}\mathcal C^s_0\subset\mathcal C^s_0.
\end{equation}
Next, at each $\rho=(\theta,\sigma)\in \mathcal K$ define the norm
$|\bullet|_{\rho}$ by
$$
|(v_\theta,v_\sigma)|^2_{\rho}:=(1-\sigma^2) v_\theta^2+{v_\sigma^2\over 1-\sigma^2}.
$$
Then there exists $\lambda>1$ such that for all $\rho\in \mathcal K$ and $v\in\mathbb R^2$
\begin{align}
  \label{e:conifer-2}
v\in \mathcal C^u_0&\quad\Longrightarrow\quad |d\varphi(\rho)v|_{\varphi(\rho)}\geq\lambda
|v|_\rho,\\
  \label{e:conifer-3}
v\in \mathcal C^s_0&\quad\Longrightarrow\quad |d\varphi^{-1}(\rho)v|_{\varphi(\rho)}\geq\lambda
|v|_\rho.
\end{align}
The hyperbolicity of $\varphi$ on $\mathcal K$ now follows by adapting the proof of Lemma~\ref{l:geodfinal},
using~\eqref{e:conifer-1}--\eqref{e:conifer-3} in place of~\eqref{e:coney-1}--\eqref{e:coney-4}.
\end{proof}



\begin{thebibliography}{0}

\bibitem[An67]{Anosov} Dmitri Anosov,
	\emph{Geodesic flows on closed Riemannian manifolds of negative curvature,\/}
	Trudy Mat. Inst. Steklov. \textbf{90}(1967), 3--210.
	
\bibitem[ChMa06]{ChernovBook} Nikolai Chernov and Roberto Markarian,
	\emph{Chaotic billiards,\/}
	Math. Surv. Mon. \textbf{127}, AMS, 2006.

\bibitem[KaHa97]{KaHa} Anatole Katok and Boris Hasselblatt,
	\emph{Introduction to the modern theory of dynamical systems,\/}
	Cambridge Univ. Press, 1997.

\end{thebibliography}
\end{document}